\documentclass[a4paper,12pt]{amsart}
\usepackage{amsmath}
\usepackage{amsthm}
\usepackage{amsfonts}
\usepackage{amssymb,array}
\usepackage{enumerate}
\usepackage[dvips]{graphicx}
\usepackage[all]{xy}

\theoremstyle{definition}
\newtheorem{define}{Definition}[section]
\theoremstyle{plain}
\newtheorem{Theorem}{Theorem}

\theoremstyle{plain}
\newtheorem{Corollary}[Theorem]{Corollary}

\theoremstyle{plain}
\newtheorem{theorem}[define]{Theorem}
\theoremstyle{plain}
\newtheorem{lemma}[define]{Lemma}
\theoremstyle{plain}
\newtheorem{proposition}[define]{Proposition}
\theoremstyle{plain}
\newtheorem{corollary}[define]{Corollary}
\theoremstyle{remark}
\newtheorem{rem}[define]{Remark}
\theoremstyle{definition}

\theoremstyle{plain}
\newtheorem{fact}[define]{Fact}

\theoremstyle{definition}

\numberwithin{equation}{section}
\numberwithin{figure}{section}
\numberwithin{table}{section}

\renewcommand{\labelenumi}{\textup{(\theenumi) }}
\renewcommand{\theenumi}{\arabic{enumi}}

\title[Visible actions on spherical nilpotent orbits]
{Visible actions on spherical nilpotent orbits in complex simple Lie algebras}
\author[A. Sasaki]{Atsumu SASAKI}

\thanks{This work was partly supported 
by Grant-in-Aid for Young Scientists (B) (22740029), 
The Ministry of Education, Culture, Sports, Science and Technology, Japan. 
}


\address{Department of Mathematics, School of Science, Tokai University, 
4-1-1, Kitakaname, Hiratsuka, Kanagawa, 259-1292, Japan. }
\email{atsumu@tokai-u.jp}

\date{\today}
\begin{document}

\begin{abstract}
This paper studies nilpotent orbits in complex simple Lie algebras 
from the viewpoint of strongly visible actions in the sense of T. Kobayashi. 
We prove that 
the action of a maximal compact group consisting of inner automorphisms on a nilpotent orbit 
is strongly visible 
if and only if 
it is spherical, namely, admitting an open orbit of a Borel subgroup. 
Further, we find a concrete description of a slice in the strongly visible action. 
As a corollary, 
we clarify a relationship among different notions of complex nilpotent orbits: 
actions of Borel subgroups (sphericity); 
multiplicity-free representations in regular functions; 
momentum maps; 
and actions of compact subgroups (strongly visible actions). 
\end{abstract}

\maketitle



\section{Introduction}
\label{sec:intro}

This paper studies nilpotent orbits, 
and bridging the two notions, ``spherical varieties'' studied by D. Panyushev \cite{pa1,pa2} 
and ``visible actions'' introduced by T. Kobayashi \cite{triunity}. 
We shall prove that 
\begin{center}
``spherical nilpotent orbits $=$ visible nilpotent orbits'', 
\end{center}
and give some structural results 
(``slice'' coming from the right-hand side). 

T. Kobayashi established a new theory on multiplicity-freeness 
for unitary representations of Lie groups by introducing the notion of visible actions. 
We recall briefly 
from \cite{triunity,mftheorem,propagation} 
the \textit{propagation theory of multiplicity-freeness property}. 
Let $G$ be a Lie group, and 
$\mathcal{V}$ a $G$-equivariant Hermitian holomorphic vector bundle over a complex manifold $D$. 
Then, 
we have a natural action of $G$ on the space $\mathcal{O}(D,\mathcal{V})$ 
of holomorphic sections, 
which is not necessary irreducible 
when $G$ does not act transitively on $D$. 
Suppose we are given a unitary representation $\mathcal{H}$ of $G$ 
from which there exists a continuous injective homomorphism to $\mathcal{O}(D,\mathcal{V})$. 
In general, $\mathcal{H}$ may not be multiplicity-free 
even if each fiber $\mathcal{V}_x$ is multiplicity-free 
as a representation of the isotropy subgroup $G_x$ ($x\in D$). 
However, 
$\mathcal{H}$ becomes multiplicity-free whenever 
$G$ acts on the base space $D$ strongly visibly. 
This is the propagation theory in \cite{propagation}, 
which yields a unified explanation of multiplicity-freeness 
for various kinds of multiplicity-free representations 
which have been studied by different approaches (see \cite{triunity,mftheorem,propagation}).

A holomorphic action of a Lie group $G$ on a connected complex manifold $D$ 
is called \textit{strongly visible} if there exist a real submanifold $S$ in $D$ 
and an anti-holomorphic diffeomorphism $\sigma$ of $D$ such that 
the following conditions are satisfied: 
	\begin{gather}
	D=G\cdot S,
	\tag{V.1}
	\label{visible:v1}
	\\
	{\sigma |}_S=\textup{id}_S,
	\tag{S.1}
	\label{visible:s1}
	\\
	\text{$\sigma$ preserves each $G$-orbit in $D$. }
	\tag{S.2}
	\label{visible:s2}
	\end{gather}

We say that the submanifold $S$ is a \textit{slice}. 
It is automatically totally real, 
namely, $J_x(T_xS)\cap T_xS=\{ 0\} $ for any $x\in S$ (see \cite[Remark 3.3.3]{mftheorem}). 
Here, $J$ stands for the complex structure of $D$. 
We are particularly interested in a slice of minimal dimension, 
namely, 
which coincides with the codimension of generic $G$-orbits in $D$. 

We remark that 
the original definition \cite[Definition 3.3.1]{mftheorem} of strongly visible actions 
is wider slightly, 
namely, 
it allows a complex manifold $D$ 
containing a non-empty $G$-invariant open set satisfying (\ref{visible:v1})--(\ref{visible:s2}). 
For the propagation theory of multiplicity-freeness property, 
this wider definition is sufficient. 
However, 
since we shall see that these two definitions are equivalent 
for $G_u$-actions on complex nilpotent orbits, 
we adopt the above definition for simplicity for the rest of this paper. 

Strongly visible actions arise from many different geometric settings 
(cf. \cite{triunity, mftheorem, nontube}). 
Recently, 
a classification theory of strongly visible actions has been developed for 
Hermitian symmetric spaces \cite{symmetric}, 
generalized flag varieties \cite{cartan,tanakaC,tanakaD,tanakaB,tanaka-except}, 
and linear spaces \cite{irr,red}. 

This paper deals with a new case where a complex manifold 
is a nilpotent orbit in a complex simple Lie algebra. 
In order to state our main results, 
we fix notation. 
Let $\mathfrak{g}$ be a finite-dimensional complex simple Lie algebra 
and $G_{\mathbb{C}}:=\operatorname{Int}\mathfrak{g}$ 
the inner automorphism group of $\mathfrak{g}$. 
We denote by $\mathcal{O}_X$ a $G_{\mathbb{C}}$-orbit through $X\in \mathfrak{g}$. 
Then we have a $G_{\mathbb{C}}$-isomorphism of complex manifolds 
$\mathcal{O}_X\simeq G_{\mathbb{C}}/(G_{\mathbb{C}})_X$ 
where $(G_{\mathbb{C}})_X$ stands for the isotropy subgroup at $X$. 
We say that $\mathcal{O}_X$ is a \textit{nilpotent} orbit 
if $X\in \mathfrak{g}$ is a nilpotent element, 
and is \textit{spherical} 
if a Borel subgroup of $G_{\mathbb{C}}$ has an open orbit in $\mathcal{O}_X$. 
Let $G_u$ be a compact real form of $G_{\mathbb{C}}$. 
We prove: 

\begin{Theorem}
\label{thm:visible}
If $\mathcal{O}_X$ is nilpotent and spherical, 
then the $G_u$-action on $\mathcal{O}_X$ is strongly visible. 
\end{Theorem}

The idea of our proof of Theorem \ref{thm:visible} is based on 
the \textit{induction theorem of strongly visible actions} 
which is first formulated by Kobayashi \cite[Theorem 20]{mftheorem} for Type A group. 
We generalize this idea for arbitrary complex simple Lie groups. 
For this, 
we choose an $\mathfrak{sl}_2$-triple $\{ H,X,Y\} $ containing $X$ as a nilpositive element. 
The semisimple element $H$ defines the $\mathbb{Z}$-grading of $\mathfrak{g}$, 
denoted by $\mathfrak{g}=\bigoplus _{m\in \mathbb{Z}}\mathfrak{g}(m)$. 
Then the complex subalgebra $\mathfrak{l}:=\mathfrak{g}(0)$ is reductive. 
Let $L_{\mathbb{C}}$ be an analytic subgroup of $G_{\mathbb{C}}$ with Lie algebra $\mathfrak{l}$. 
Taking a conjugation if necessary, 
we may and do assume that 
$L_u:=L_{\mathbb{C}}\cap G_u$ is a compact real form of $L_{\mathbb{C}}$. 
We set a nilpotent subalgebra by $\mathfrak{n}:=\bigoplus _{m\geq 2}\mathfrak{g}(m)$. 
Then, the nilpotent orbit 
$\mathcal{O}_X$ can be realized via the following map: 
\begin{align}
G_u\times _{L_u}\mathfrak{n}\to \overline{\mathcal{O}_X},\quad 
(g,Z)\mapsto g\cdot Z, 
\label{eq:induction}
\end{align}
in particular, the closure $\overline{\mathcal{O}_X}$ is equal to 
$G_u\cdot \mathfrak{n}=\{ g\cdot Z:g\in G_u,Z\in \mathfrak{n}\} $. 
In this setting, we generalize the induction theorem of strongly visible actions: 

\begin{Theorem}[see Theorem \ref{thm:induction}]
\label{thm:induction thm}
If the $L_u$-action on $\mathfrak{n}$ is strongly visible, 
then the $G_u$-action on $\mathcal{O}_X$ is strongly visible. 
\end{Theorem}

Theorem \ref{thm:induction thm} means that 
the strong visibility for non-linear action on $\mathcal{O}_X$ 
is induced from the strong visibility for linear action on $\mathfrak{n}$ 
via (\ref{eq:induction}). 
We then can apply the previous results \cite{irr,red} 
for the classification of linear visible actions to the $L_u$-action on $\mathfrak{n}$, 
and thus prove: 

\begin{Theorem}
\label{thm:linear}
If $\mathcal{O}_X$ is spherical, 
then the $L_u$-action on $\mathfrak{n}$ is strongly visible. 
\end{Theorem}

Therefore, Theorem \ref{thm:visible} follows 
from Theorems \ref{thm:induction thm} and \ref{thm:linear}. 

Our proof of Theorem \ref{thm:linear} applies a case-by-case analysis by using Panyushev \cite{pa1}. 
Moreover, 
we give an explicit description of a slice and an anti-holomorphic diffeomorphism 
for the $L_u$-action on $\mathfrak{n}$ when $\mathcal{O}_X$ is spherical. 

Together with the earlier results \cite{pa1,vk,vinberg}, we summarize: 

\begin{Corollary}
\label{cor:equivalent}
The following five conditions on 
a nilpotent orbit $\mathcal{O}_X$ in a complex simple Lie algebra $\mathfrak{g}$ 
are equivalent: 
\begin{enumerate}
	\renewcommand{\theenumi}{\roman{enumi}}
	\item \label{item:spherical}
	$\mathcal{O}_X$ is spherical. 
	
	\item \label{item:height}
	The height of $\mathcal{O}_X$ equals two or three. 
	
	\item \label{item:mf}
	The space of regular functions on $\mathcal{O}_X$ is multiplicity-free 
	as a representation of $G_{\mathbb{C}}=\operatorname{Int}\mathfrak{g}$. 
	
	\item \label{item:visible-n}
	The $L_u$-action on the nilpotent subalgebra $\mathfrak{n}$ is strongly visible. 
	
	\item \label{item:visible}
	The $G_u$-action on $\mathcal{O}_X$ is strongly visible. 
\end{enumerate}
\end{Corollary}

By spherical nilpotent orbits, 
we mean that 
it is a nilpotent orbit on which a maximal compact subgroup of $G_{\mathbb{C}}$ 
acts strongly visibly. 

Here, the height of $\mathcal{O}_X$ is defined 
by the maximum of $m\in \mathbb{Z}$ satisfying $\mathfrak{g}(m)\neq \{ 0\} $ 
(see Definition \ref{def:height}). 

The equivalence between (\ref{item:spherical}) and (\ref{item:height}) 
is proved by Panyushev \cite{pa1}. 
The equivalence 
(\ref{item:spherical}) $\Leftrightarrow$ (\ref{item:mf}) is 
due to Vinberg--Kimelfeld \cite{vk} and Vinberg \cite{vinberg}. 
The implication (\ref{item:visible}) $\Rightarrow$ (\ref{item:mf}) is a special case of 
the propagation theory of multiplicity-freeness property 
by Kobayashi \cite{triunity,mftheorem,propagation}. 
The implication (\ref{item:spherical}) $\Rightarrow$ (\ref{item:visible-n}) 
and (\ref{item:visible-n}) $\Rightarrow$ (\ref{item:visible}) hold 
by Theorems \ref{thm:linear} and \ref{thm:induction thm}, respectively. 

\begin{center}
\begin{tabular}{c@{}c@{}c@{}c@{}c@{}c}
(\ref{item:height}) & $\stackrel{\scriptsize{\mbox{\cite{pa1}}}}{\Leftrightarrow }$ & (\ref{item:spherical}) 
	& $\stackrel{\scriptsize{\mbox{\cite{vk,vinberg}}}}{\Leftrightarrow}$ & (\ref{item:mf}) \\
 & \scriptsize{Theorem \ref{thm:linear}} & \rotatebox[origin=c]{90}{$\Leftarrow$} & & \rotatebox[origin=c]{90}{$\Rightarrow$} & \scriptsize{\cite{triunity,mftheorem,propagation}} \\
 && (\ref{item:visible-n}) & $\Rightarrow$ & (\ref{item:visible}) \\
 &&& \scriptsize{Theorem \ref{thm:induction thm}} & 
\end{tabular}
\end{center}

Corollary \ref{cor:equivalent} has some connection with 
``small infinite-dimensional representations" of complex reductive Lie groups $G_{\mathbb{C}}$. 
If $\overline{\mathcal{O}_{\pi}}$ is the associated variety of an admissible representation $\pi$ 
of $G_{\mathbb{C}}$ (see \cite{vogan}), 
then the $G_u$-type in $\pi $ is asymptotically the same 
with the $G_u$-type in the space of regular functions in $\overline{\mathcal{O}_{\pi}}$ 
by \cite[Proposition 3.3]{ko}. 

This paper is organized as follows. 
In Section \ref{sec:preliminaries}, 
we review basic notion of nilpotent orbits in complex semisimple Lie algebras. 
In Section \ref{sec:visible-n}, 
we explain a key theorem for the proof of Theorem \ref{thm:linear} (Theorem \ref{thm:slice-n}), 
namely, 
properties of our choice of a slice and an anti-holomorphic diffeomorphism 
in the strongly visible $L_u$-action on $\mathfrak{n}$ if $\mathcal{O}_X$ is spherical. 
In Section \ref{sec:induction}, 
we show the induction theorem of strongly visible actions, 
namely, Theorem \ref{thm:induction thm}. 
In Section \ref{sec:proof-n}, 
we give a proof of Theorem \ref{thm:slice-n}. 

\section*{Acknowledgments}

The author would like to thank Prof. S. Kaneyuki 
for showing his interest in this work and helpful comments in the preparation of this paper. 

The main result of this paper was announced first in the Oberwolfach workshop 
``Infinite Dimensional Lie Theory'' 
organized by Prof. K.-H. Neeb, Prof. A. Pianzola, Prof. T. Ratiu, November 2010 
and in Lorentz Center Workshop ``Analysis, Geometry and Group Representations 
for Homogeneous Spaces'' 
organized by Prof. T. Nomura, Prof. G. Helminck, November 2010. 
The author would like to express deepest gratitude to the organizers 
for the opportunity to the presentations in the wonderful workshops. 



\section{Preliminaries}
\label{sec:preliminaries}

In this section, 
we review structural theories 
on nilpotent orbits in complex semisimple Lie algebras 
which is based on \cite{nilpotent}. 

\subsection{$\mathfrak{sl}_2$-triple of nilpotent orbit}
\label{subsec:H}

Let $\mathfrak{g}$ be a finite-dimensional complex semisimple Lie algebra 
and $G_{\mathbb{C}}:=\operatorname{Int}\mathfrak{g}$ 
the inner automorphism group of $\mathfrak{g}$. 
An element $X\in \mathfrak{g}$ is called a nilpotent element 
if $\operatorname{ad}(X)\in \operatorname{End}(\mathfrak{g})$ satisfies 
$\operatorname{ad}(X)^N=0$ for some $N\in \mathbb{N}$, 
and called a semisimple element 
if $\operatorname{ad}(X)$ is diagonalizable. 
We write $\mathcal{N}$ for the cone of nilpotent elements of $\mathfrak{g}$ 
and $\mathcal{S}$ for the set of semisimple elements of $\mathfrak{g}$. 

We denote by $\mathcal{O}_X=G_{\mathbb{C}}\cdot X$ 
the $G_{\mathbb{C}}$-orbit through $X\in \mathfrak{g}$. 
An orbit $\mathcal{O}_X$ is called a nilpotent (resp. semisimple) orbit 
if $X\in \mathcal{N}$ (resp. $X\in \mathcal{S})$, 
from which $\mathcal{O}_X\subset \mathcal{N}$ (resp. $\mathcal{O}_X\subset \mathcal{S}$). 
We write $\mathcal{N}^*/G_{\mathbb{C}}$ ($\mathcal{N}^*:=\mathcal{N}\setminus \{ 0\} $) 
and $\mathcal{S}/G_{\mathbb{C}}$ 
for the set of non-zero nilpotent orbits and of semisimple orbits, respectively. 

Suppose we are given $X\in \mathcal{N}^*$. 
By Jacobson--Morozov, 
there exist $H,Y\in \mathfrak{g}$ such that 
$\{ H,X,Y\}$ forms an $\mathfrak{sl}_2$-triple, namely, 
\begin{align*}
[H,X]=2X,~[H,Y]=-2Y,~[X,Y]=H. 
\end{align*}
Then, 
$H$ is called a neutral element, $X$ a nilpositive element, and $Y$ a nilnegative element. 
We remark that $H\in \mathcal{S}$ and $X,Y\in \mathcal{N}$. 

Two $\mathfrak{sl}_2$-triples $\{ H,X,Y\} ,\{ H',X',Y'\} $ in $\mathfrak{g}$ are 
said to be conjugate 
if there exists $g\in G_{\mathbb{C}}$ such that $H'=g\cdot H,~X'=g\cdot X$, and $Y'=g\cdot Y$. 
Clearly, 
two elements $X,X'\in \mathcal{N}^*$ are conjugate 
if two $\mathfrak{sl}_2$-triples $\{ H,X,Y\} ,\{ H',X',Y'\} $ in $\mathfrak{g}$ are conjugate. 
The opposite is also true by Kostant \cite{kostant}. 
Then, this gives rise to the following bijection 
\begin{align}
\mathcal{N}^*/G_{\mathbb{C}}\stackrel{\sim }{\to }
\{ \mbox{$\mathfrak{sl}_2$-triples in $\mathfrak{g}$}\} /G_{\mathbb{C}},\quad 
\mathcal{O}_X\mapsto \{ H,X,Y\} . 
\label{eq:kostant}
\end{align}
Moreover, 
it follows from Mal'cev \cite{malcev} that 
$\{ H,X,Y\} $ and $\{ H',X',Y'\} $ are conjugate 
if the corresponding neutral elements $H,H'$ are conjugate. 
This implies that the map 
\begin{align}
\{ \mbox{$\mathfrak{sl}_2$-triples in $\mathfrak{g}$}\} /G_{\mathbb{C}}
\to  
\mathcal{S}/G_{\mathbb{C}},\quad 
\{ H,X,Y\} \mapsto \mathcal{O}_H
\label{eq:malcev}
\end{align}
is injective. 
Hence, composing (\ref{eq:kostant}) and (\ref{eq:malcev}) yields the injective map 
from the set of nilpotent orbits to that of semisimple ones: 
\begin{align}
\Phi :\mathcal{N}^*/G_{\mathbb{C}}\to 
\mathcal{S}/G_{\mathbb{C}},\quad 
\mathcal{O}_X\mapsto \mathcal{O}_H. 
\label{eq:distinguished}
\end{align}

\subsection{$\mathbb{Z}$-grading of $\mathfrak{g}$}
\label{subsec:height}

In this subsection, 
we consider the semisimple transformation $\operatorname{ad}(H)$ on $\mathfrak{g}$ 
for $H\in \Phi (\mathcal{O}_X)$. 

By the general theory on finite-dimensional representations of $\mathfrak{sl}(2,\mathbb{C})$, 
all $\operatorname{ad}(H)$-eigenvalues are integers. 
We write $\mathfrak{g}(m)$ for the $\operatorname{ad}(H)$-eigenspace 
with eigenvalue $m$, namely, 
\begin{align}
\mathfrak{g}(m):=\{ Z\in \mathfrak{g}:\operatorname{ad}(H)Z=mZ\} \quad (m\in \mathbb{Z}).
\label{eq:eigenspace}
\end{align}
Then, $\mathfrak{g}$ is decomposed into the finite sum of $\operatorname{ad}(H)$-eigenspaces as 
\begin{align}
\mathfrak{g}=\bigoplus _{m\in \mathbb{Z}}\mathfrak{g}(m). 
\label{eq:grading}
\end{align}
Since the inclusion 
$[\mathfrak{g}(m),\mathfrak{g}(n)]\subset \mathfrak{g}(m+n)$ 
holds for $m,n\in \mathbb{Z}$, 
the decomposition 
(\ref{eq:grading}) defines a $\mathbb{Z}$-grading of $\mathfrak{g}$. 

Let us take another element $H'\in \Phi (\mathcal{O}_X)$ 
and consider the $\operatorname{ad}(H')$-eigenspace $\mathfrak{g}'(m)$. 
We note that $H'=g_0\cdot H$ for some $g_0\in G_{\mathbb{C}}$. 
Then, we have: 

\begin{lemma}
\label{lem:independent}
$\mathfrak{g}'(m)=g_0\cdot \mathfrak{g}(m)=\{ g_0\cdot Z:Z\in \mathfrak{g}(m)\} $ 
for any $m\in \mathbb{Z}$. 
\end{lemma}

\begin{proof}
For $Z\in \mathfrak{g}(m)$, 
we observe $g_0\cdot (\operatorname{ad}(H)Z)$ as follows. 
First, 
by the definition of $\mathfrak{g}(m)$, 
we have 
\begin{align}
g_0\cdot (\operatorname{ad}(H)Z)=g_0\cdot (mZ)=m(g_0\cdot Z). 
\label{eq:adH-1}
\end{align}
Second, in view of $g_0\cdot [H,Z]=[g_0\cdot H,g_0\cdot Z]$, we express as 
\begin{align}
g_0\cdot (\operatorname{ad}(H)Z)
=\operatorname{ad}(g_0\cdot H)(g_0\cdot Z)
=\operatorname{ad}(H')(g_0\cdot Z). 
\label{eq:adH-2}
\end{align}
Comparing (\ref{eq:adH-1}) and (\ref{eq:adH-2}), 
we obtain $\operatorname{ad}(H')(g_0\cdot Z)=m(g_0\cdot Z)$, 
from which $g_0\cdot Z\in \mathfrak{g}'(m)$. 
Hence, we have shown $g_0\cdot \mathfrak{g}(m)\subset \mathfrak{g}'(m)$. 

Similarly, we have $g_0^{-1}\cdot \mathfrak{g}'(m)\subset \mathfrak{g}(m)$, and then 
$\mathfrak{g}'(m)\subset g_0\cdot \mathfrak{g}(m)$. 

Therefore, we have proved $\mathfrak{g}'(m)=g_0\cdot \mathfrak{g}(m)$. 
\end{proof}

Lemma \ref{lem:independent} shows that 
the $\mathbb{Z}$-grading (\ref{eq:grading}) is determined by $\mathcal{O}_X$, 
particularly, 
independent on the choice of semisimple elements of $\Phi (\mathcal{O}_X)$.

%
%
%

\subsection{Nilpotent subalgebra}
\label{subsec:nilpotent-subalgebra}

We define a parabolic subalgebra $\mathfrak{q}$ 
arising from the $\mathbb{Z}$-grading (\ref{eq:grading}) by 
\begin{align}
\mathfrak{q}:=\bigoplus _{m\geq 0}\mathfrak{g}(m) 
\label{eq:parabolic}
\end{align}
with Levi decomposition $\mathfrak{q}=\mathfrak{l}+\mathfrak{u}$ 
where 
\begin{align}
\mathfrak{l}:=\mathfrak{z}_{\mathfrak{g}}(H)=\mathfrak{g}(0) 
\label{eq:levi}
\end{align}
is a Levi subalgebra 
and $\mathfrak{u}=\bigoplus _{m>0}\mathfrak{g}(m)$ is a nilradical. 
Then, $\mathfrak{l}$ is a complex reductive Lie algebra. 

Let $\mathfrak{n}$ be a nilpotent subalgebra by 
\begin{align}
\mathfrak{n}:=\bigoplus _{m\geq 2}\mathfrak{g}(m). 
\label{eq:nilpotent-subalgebra}
\end{align}
Clearly, $[\mathfrak{q},\mathfrak{n}]\subset \mathfrak{n}$. 
As $X\in \mathfrak{g}(2)$, 
we have $[\mathfrak{q},X]\subset \mathfrak{n}$. 
Further, the opposite inclusion also holds 
by the representation theory of $\mathfrak{sl}(2,\mathbb{C})$, 
from which we obtain: 

\begin{lemma}
\label{lem:n}
$\mathfrak{n}=[\mathfrak{q},X]$. 
\end{lemma}

\subsection{Realization of $\mathcal{O}_X$ via momentum map}
\label{subsec:realization}

Let $Q_{\mathbb{C}}$ be a parabolic subgroup of $G_{\mathbb{C}}$ 
with Lie algebra $\mathfrak{q}$, 
which acts on $\mathfrak{n}$. 
We set 
\begin{align}
\mathfrak{n}^{\circ}:=Q_{\mathbb{C}}\cdot X. 
\label{eq:regular}
\end{align}
By Lemma \ref{lem:n}, 
$\mathfrak{n}^{\circ}$ is an open set in $\mathfrak{n}$, 
in particular, its closure $\overline{\mathfrak{n}^{\circ}}$ is equal to $\mathfrak{n}$. 

We define a $G_{\mathbb{C}}$-equivariant smooth map $\varphi$ 
from the holomorphic vector bundle 
$G_{\mathbb{C}}\times _{Q_{\mathbb{C}}}\mathfrak{n}$ on the flag manifold 
$G_{\mathbb{C}}/Q_{\mathbb{C}}$ to $\mathfrak{g}$ by 
\begin{align}
\varphi :G_{\mathbb{C}}\times _{Q_{\mathbb{C}}}\mathfrak{n}\to \mathfrak{g},~
(x,Z)\mapsto x\cdot Z. 
\label{eq:regular-map}
\end{align}

\begin{lemma}
\label{lem:another realization}
$\mathcal{O}_X=\varphi (G_{\mathbb{C}}\times _{Q_{\mathbb{C}}}\mathfrak{n}^{\circ})$. 
\end{lemma}

\begin{proof}
It follows from (\ref{eq:regular}) that 
\begin{align}
\varphi (G_{\mathbb{C}}\times _{Q_{\mathbb{C}}}\mathfrak{n}^{\circ})
=\varphi (G_{\mathbb{C}}\times _{Q_{\mathbb{C}}}(Q_{\mathbb{C}}\cdot X))
=G_{\mathbb{C}}\cdot X=\mathcal{O}_X. 
\label{eq:realization-cpx}
\end{align}
Hence, Lemma \ref{lem:another realization} has been proved. 
\end{proof}

Next, 
let $L_{\mathbb{C}}$ be a Levi subgroup of $Q_{\mathbb{C}}$ with Lie algebra $\mathfrak{l}$. 
Taking a conjugation if necessary, 
we may and do assume that 
$L_u:=L_{\mathbb{C}}\cap G_u$ is a compact real form of $L_{\mathbb{C}}$. 
Then, the inclusion map $G_u\hookrightarrow G_{\mathbb{C}}$ 
induces a biholomorphic diffeomorphism $G_u/L_u\simeq G_{\mathbb{C}}/Q_{\mathbb{C}}$. 
This gives rise to an isomorphism as a $G_u$-equivariant holomorphic 
vector bundle as follows: 
\begin{align}
G_u\times _{L_u}\mathfrak{n}
\simeq G_{\mathbb{C}}\times _{Q_{\mathbb{C}}}\mathfrak{n}. 
\label{eq:bundle-isom}
\end{align}
We use the same letter $\varphi$ to denote the $G_u$-equivariant map 
$G_u\times _{L_u}\mathfrak{n}\to \mathfrak{g}$ via the isomorphism (\ref{eq:bundle-isom}). 
Then, we have: 

\begin{proposition}
\label{prop:another realization}
$\mathcal{O}_X=G_u\cdot \mathfrak{n}^{\circ}$. 
\end{proposition}

\begin{proof}
By Lemma \ref{lem:another realization} and (\ref{eq:bundle-isom}), 
we have 
\begin{align*}
\mathcal{O}_X=\varphi (G_u\times _{L_u}\mathfrak{n}^{\circ})
=G_u\cdot \mathfrak{n}^{\circ}, 
\end{align*}
from which Proposition \ref{prop:another realization} has been proved. 
\end{proof}

\begin{rem}[cf. {\cite[Theorem 20]{mftheorem}}]
We can regard $\varphi$ as the restriction of the momentum map 
of the Hamiltonian $G_u$-action on the flag manifold $G_u/L_u$. 

Let $\mathfrak{g}_u$ and $\mathfrak{l}_u$ be the Lie algebras of $G_u$ and $L_u$, 
respectively. 
Identifying $\mathfrak{g}_u/\mathfrak{l}_u$ with $\mathfrak{u}$, 
the nilpotent subalgebra $\mathfrak{n}$ seems to be an $L_u$-submodule of 
$\mathfrak{g}_u/\mathfrak{l}_u$. 
Then, the principal vector bundle $G_u\times _{L_u}(\mathfrak{g}_u/\mathfrak{l}_u)$ 
is isomorphic to the cotangent bundle $T^*(G_u/L_u)$. 
Via the identification of $\mathfrak{g}$ with the dual space $\mathfrak{g}^*$ 
by the Killing form on $\mathfrak{g}$, 
the map 
\begin{align*}
\psi :G_u\times _{L_u}\mathfrak{u}\to \mathfrak{g},\quad (g,Z)\mapsto g\cdot Z
\end{align*}
is essentially a momentum map of the Hamiltonian $G_u$-action on the flag manifold $G_u/L_u$. 
It turns out that 
$\varphi =\psi |_{G_u\times _{L_u}\mathfrak{n}}$. 
\end{rem}



\section{Visible actions on nilpotent subalgebras}
\label{sec:visible-n}

Let us retain the setting of Section \ref{sec:preliminaries}. 
A nilpotent orbit $\mathcal{O}_X$ is written as 
$\mathcal{O}_X=G_u\cdot \mathfrak{n}^{\circ}$. 
In view of this realization, 
it is crucial for the study on the $G_u$-action on $\mathcal{O}_X$ 
to understand the $L_u$-action on $\mathfrak{n}^{\circ}$. 
Here, the $Q_{\mathbb{C}}$-orbit $\mathfrak{n}^{\circ}$ through $X$ 
is equal to the closure of 
the nilpotent subalgebra $\mathfrak{n}$. 
In this section, 
we focus on the $L_u$-action on $\mathfrak{n}$.

\subsection{Normal real form of complex simple Lie algebra}

First, 
we define an anti-holomorphic involution of a complex simple Lie algebra. 

For a reductive Lie algebra $\mathfrak{g}'$, 
we denote by $\operatorname{rank}_{\mathbb{R}}\mathfrak{g}'$ 
the real rank of $\mathfrak{g}'$. 
A real form $\mathfrak{g}'_{\mathbb{R}}$ of a complex reductive Lie algebra $\mathfrak{g}'$ 
is called \textit{normal} if $\operatorname{rank}_{\mathbb{R}}\mathfrak{g}'_{\mathbb{R}}
=\operatorname{rank}\mathfrak{g}'$. 
Normal real forms of a complex simple Lie algebra exist 
and are unique up to isomorphism. 

Let $\mathfrak{g}$ be a complex simple Lie algebra. 
We take a normal real form $\mathfrak{g}_{\mathbb{R}}$ of $\mathfrak{g}$. 
We denote by $\sigma $ the complex conjugation of $\mathfrak{g}$ 
with respect to $\mathfrak{g}_{\mathbb{R}}$, namely, 
\begin{align}
\sigma (X+\sqrt{-1}Y)=X-\sqrt{-1}Y\quad (X,Y\in \mathfrak{g}_{\mathbb{R}}). 
\label{eq:sigma}
\end{align}

Let $\mathfrak{k}_{\mathbb{R}}$ be a maximal compact subalgebra of $\mathfrak{g}_{\mathbb{R}}$ 
and $\mathfrak{g}_{\mathbb{R}}=
\mathfrak{k}_{\mathbb{R}}+\mathfrak{p}_{\mathbb{R}}$ 
the corresponding Cartan decomposition. 
Then, the Lie algebra 
$\mathfrak{k}_{\mathbb{R}}+\sqrt{-1}\mathfrak{p}_{\mathbb{R}}$ 
is a $\sigma$-stable compact real form of $\mathfrak{g}$. 
Since compact real forms are unique up to isomorphism, 
we may and do assume that 
$\mathfrak{g}_u:=\mathfrak{k}_{\mathbb{R}}+\sqrt{-1}\mathfrak{p}_{\mathbb{R}}$ 
is the Lie algebra of $G_u$ 
by taking a conjugation of $\mathfrak{g}_{\mathbb{R}}$ if necessary. 

\subsection{$\sigma$-stability of subalgebra}
\label{subsec:stable}

Let $\mathfrak{a}_{\mathbb{R}}$ be a maximal abelian subspace 
in $\mathfrak{p}_{\mathbb{R}}$ and 
\begin{align}
\mathfrak{a}:=\mathfrak{a}_{\mathbb{R}}+\sqrt{-1}\mathfrak{a}_{\mathbb{R}}. 
\label{eq:cartan}
\end{align}
Then, 
$\mathfrak{a}$ is $\sigma $-stable, 
and a Cartan subalgebra of $\mathfrak{g}$
since $\operatorname{rank}_{\mathbb{R}}\mathfrak{g}_{\mathbb{R}}
=\operatorname{rank}\mathfrak{g}$. 
Then, $\mathfrak{a}_{\mathbb{R}}$ is 
a Cartan subalgebra of $\mathfrak{g}_{\mathbb{R}}$. 

Let $\Delta \equiv \Delta (\mathfrak{g},\mathfrak{a})$ 
be a root system of $\mathfrak{g}$ 
with respect to $\mathfrak{a}$. 
We denote by $\mathfrak{g}_{\alpha }$ the root space corresponding to $\alpha \in \Delta$. 
Then, we have: 

\begin{lemma}
\label{lem:rootspace}
The root space $\mathfrak{g}_{\alpha }$ is $\sigma$-stable 
for any $\alpha \in \Delta $. 
\end{lemma}

\begin{proof}
According to (\ref{eq:cartan}), 
we write $A\in \mathfrak{a}$ as $A=A_1+\sqrt{-1}A_2~(A_1,A_2\in \mathfrak{a}_{\mathbb{R}})$. 
As $\mathfrak{a}_{\mathbb{R}}\subset \mathfrak{g}_{\mathbb{R}}$, 
we have $\sigma (A)=A_1-\sqrt{-1}A_2$. Hence, 
\begin{align*}
\alpha (\sigma (A))
&=\alpha (A_1-\sqrt{-1}A_2)
=\alpha (A_1)-\sqrt{-1}\alpha (A_2). 
\end{align*}
We remark that 
all of the roots are real on $\mathfrak{a}_{\mathbb{R}}$ 
(cf. \cite[Section 4]{knapp}). 
This means that $\alpha (A_1),\alpha (A_2)\in \mathbb{R}$, and then, 
$\alpha (\sigma (A))=\alpha (A_1)-\sqrt{-1}\alpha (A_2)=\overline{\alpha (A)}$. 
Since $\sigma$ is anti-linear, we have 
\begin{align}
\sigma (\alpha (A)Z)
&=\overline{\alpha (A)}\sigma (Z) 
=\alpha (\sigma (A))\sigma (Z)\quad (Z\in \mathfrak{g}). 
\label{eq:root}
\end{align}

Let $Z_{\alpha }$ be a root vector of $\mathfrak{g}_{\alpha }$. 
It follows from (\ref{eq:root}) that 
\begin{align}
\label{eq:rootspace1}
\sigma ([A,Z_{\alpha }])&=\sigma (\alpha (A)Z_{\alpha }) 
=\alpha (\sigma (A))\sigma (Z_{\alpha }). 
\end{align}
On the other hand, 
we have $\sigma ([A,Z_{\alpha }])=[\sigma (A),\sigma (Z)]$ 
since $\sigma $ is an involution on $\mathfrak{g}$. 
Replacing $A$ with $\sigma (A)$ in the equality (\ref{eq:rootspace1}) shows 
\begin{align}
[A,\sigma (Z_{\alpha })]=\alpha (A)\sigma (Z_{\alpha }). 
\end{align}
This means that 
$\sigma (Z_{\alpha })$ lies in $\mathfrak{g}_{\alpha}$. 
Hence, Lemma \ref{lem:rootspace} has been proved. 
\end{proof}

We recall the well-known facts that 
all elements contained in a Cartan subalgebra are semisimple and 
that two Cartan subalgebras of $\mathfrak{g}$ are conjugate by $G_{\mathbb{C}}$. 
This means that $\Phi (\mathcal{O}_X)\cap \mathfrak{a}\neq \emptyset$. 
Hence, 
we take a semisimple element $H$ in $\Phi (\mathcal{O}_X)\cap \mathfrak{a}$. 
In this setting, we have: 

\begin{lemma}
\label{lem:eigenspace}
The $\operatorname{ad}(H)$-eigenspace $\mathfrak{g}(m)$ is $\sigma$-stable 
for any $m\in \mathbb{Z}$. 
\end{lemma}

\begin{proof}
Since 
$[H,Z_{\alpha }]=\alpha (H)Z_{\alpha}$ holds for any $Z_{\alpha }\in \mathfrak{g}_{\alpha }$, 
we have $\mathfrak{g}_{\alpha }\subset \mathfrak{g}(\alpha (H))$. 
This implies that the eigenspace $\mathfrak{g}(m)$ is written as 
$\mathfrak{g}(m)=\bigoplus _{\alpha \in \Delta _m}\mathfrak{g}_{\alpha }$ 
for some subset $\Delta _m\subset \Delta \cup \{ 0\} $. 
Hence, the $\sigma$-stability of $\mathfrak{g}(m)$ 
follows from Lemma \ref{lem:rootspace}. 
\end{proof}

The parabolic subalgebra $\mathfrak{q}$, 
the Levi subalgebra $\mathfrak{l}$, 
and the nilpotent subalgebra $\mathfrak{n}$, respectively, 
consist of some $\operatorname{ad}(H)$-eigenspaces. 
Then, 
the following lemma is an immediate consequence of Lemma \ref{lem:eigenspace}. 

\begin{lemma}
\label{lem:sigma-n}
The subalgebras $\mathfrak{q}$, $\mathfrak{l}$, and $\mathfrak{n}$ of $\mathfrak{g}$ 
are $\sigma$-stable. 
\end{lemma}

By Lemma \ref{lem:sigma-n}, 
the restriction of $\sigma \in \operatorname{Aut}\mathfrak{g}$ to $\mathfrak{l}$ 
becomes an anti-holomorphic involution of $\mathfrak{l}$. 
Here, we set a real form of $\mathfrak{l}$ by $\mathfrak{l}_{\mathbb{R}}
=\mathfrak{l}^{\sigma }=\mathfrak{l}\cap \mathfrak{g}_{\mathbb{R}}$. 
Then, we have: 

\begin{lemma}
\label{lem:normal-l}
The real form $\mathfrak{l}_{\mathbb{R}}$ of $\mathfrak{l}$ is normal. 
\end{lemma}

\begin{proof}
By definition, 
let us show the equality $\operatorname{rank}_{\mathbb{R}}\mathfrak{l}_{\mathbb{R}}
=\operatorname{rank}\mathfrak{l}$. 
Here, the inequality $\operatorname{rank}_{\mathbb{R}}\mathfrak{l}_{\mathbb{R}}
\leq \operatorname{rank}\mathfrak{l}$ holds in general. 
Then, it is sufficient to see $\operatorname{rank}_{\mathbb{R}}\mathfrak{l}_{\mathbb{R}}
\geq \operatorname{rank}\mathfrak{l}$. 

For this, 
we consider the maximal abelian $\mathfrak{a}_{\mathbb{R}}\subset \mathfrak{p}_{\mathbb{R}}$. 
The semisimple element $H\in \mathfrak{a}$ satisfies $[H,\mathfrak{a}]=\{ 0\} $, 
from which $\mathfrak{a}_{\mathbb{R}}\subset \mathfrak{a}\subset \mathfrak{l}$. 
Then, we obtain $\mathfrak{a}_{\mathbb{R}}\subset \mathfrak{l}\cap \mathfrak{p}_{\mathbb{R}}
=\mathfrak{l}_{\mathbb{R}}\cap \mathfrak{p}_{\mathbb{R}}$. 
As $\mathfrak{l}_{\mathbb{R}}=(\mathfrak{l}_{\mathbb{R}}\cap \mathfrak{k}_{\mathbb{R}})
+(\mathfrak{l}_{\mathbb{R}}\cap \mathfrak{p}_{\mathbb{R}})$ is 
a Cartan decomposition of $\mathfrak{l}_{\mathbb{R}}$, 
this inclusion implies that $\operatorname{rank}_{\mathbb{R}}\mathfrak{l}_{\mathbb{R}}
\geq \dim \mathfrak{a}_{\mathbb{R}}$. 
Since $\operatorname{rank}\mathfrak{g}\geq \operatorname{rank}\mathfrak{l}$ holds in general 
and $\mathfrak{g}_{\mathbb{R}}$ is a normal real form of $\mathfrak{g}$, 
we conclude 
\begin{align*}
\operatorname{rank}_{\mathbb{R}}\mathfrak{l}_{\mathbb{R}}
\geq \operatorname{rank}\mathfrak{g}\geq \operatorname{rank}\mathfrak{l}. 
\end{align*}

Therefore, Lemma \ref{lem:normal-l} has been proved. 
\end{proof}

\subsection{Compatible automorphism}
\label{subsec:lift}

We lift the anti-linear involution $\sigma $ on $\mathfrak{g}$ 
to an anti-holomorphic involution $\widetilde{\sigma}$ on the complex simple Lie group 
$G_{\mathbb{C}}=\operatorname{Int}\mathfrak{g}$. 
More precisely, we write 
$\widetilde{\sigma }(g):=\sigma g\sigma \in G_{\mathbb{C}}~(g\in G_{\mathbb{C}})$. 
Then, we have: 
\begin{align}
\sigma (g\cdot Z)=\widetilde{\sigma }(g)\cdot \sigma (Z)\quad 
(g\in G_{\mathbb{C}},~Z\in \mathfrak{g}). 
\label{eq:compatibility}
\end{align}

\begin{define}[see \cite{faraut-thomas,propagation}]
\label{def:compatible}
We say that an anti-holomorphic involution $\widetilde{\sigma}$ on $G_{\mathbb{C}}$ 
is a \textit{compatible automorphism} 
with the anti-holomorphic diffeomorphism $\sigma$ 
for the $G_{\mathbb{C}}$-action on $\mathfrak{g}$ 
if the equality (\ref{eq:compatibility}) holds. 
\end{define}

In view of Lemma \ref{lem:sigma-n}, 
the Levi subgroup $L_{\mathbb{C}}$ is $\widetilde{\sigma }$-stable, 
and 
the restriction of $\sigma \in \operatorname{Aut}\mathfrak{g}$ to $\mathfrak{n}$ 
induces an anti-holomorphic diffeomorphism on $\mathfrak{n}$. 
Hence, $\widetilde{\sigma }$ also becomes a compatible automorphism 
with $\sigma $ 
for the $L_{\mathbb{C}}$-action on $\mathfrak{n}$, namely, 
\begin{align}
\sigma (k\cdot Z)=\widetilde{\sigma }(k)\cdot \sigma (Z)\quad 
(k\in L_{\mathbb{C}},~Z\in \mathfrak{n}). 
\label{eq:compatibility-n}
\end{align}

The compact real form $G_u$ of $G_{\mathbb{C}}$ is $\widetilde{\sigma}$-stable 
because its Lie algebra $\mathfrak{g}_u$ is $\sigma $-stable. 
It follows from Lemma \ref{lem:sigma-n} that $\mathfrak{l}_u:=\mathfrak{l}\cap \mathfrak{g}_u$ 
is $\sigma$-stable, from which 
$L_u=L_{\mathbb{C}}\cap G_u$ is $\widetilde{\sigma}$-stable. 

\subsection{Visible actions on nilpotent subalgebra}
\label{subsec:property-slice}

We are ready to state our theorem on the $L_u$-action on $\mathfrak{n}$ as follows. 

\begin{theorem}
\label{thm:slice-n}
If $\mathcal{O}_X$ is nilpotent and spherical, 
then one can find $S_0\subset \mathfrak{n}$ 
such that the following properties are satisfied: 
\begin{enumerate}
	\renewcommand{\theenumi}{\alph{enumi}}
	\item $S_0$ is a real vector space. 
	\label{cond:vectorspace}
	
	\item $\mathfrak{n}=L_u\cdot S_0$. 
	\label{cond:orbit-decomp-n}
	
	\item $\sigma |_{S_0}=\operatorname{id}_{S_0}$. 
	\label{cond:s1-n}
\end{enumerate}
\end{theorem}

Our proof of Theorem \ref{thm:slice-n} uses a case-by-case analysis 
for each spherical $\mathcal{O}_X$, 
which will be given in Section \ref{sec:proof-n} separated from this section. 

\subsection{Proof of Theorem \ref{thm:linear}}

Let us see that Theorem \ref{thm:linear} follows from Theorem \ref{thm:slice-n}. 

\begin{proof}[Proof of Theorem \ref{thm:linear}]
Suppose that $\mathcal{O}_X$ is nilpotent and spherical. 
By Theorem \ref{thm:slice-n}, 
there exists a real vector subspace $S_0$ 
such that $\mathfrak{n}=L_u\cdot S_0$ and $\sigma |_{S_0}=\operatorname{id}_{S_0}$. 
We will verify that the data $(S_0,\sigma )$ 
satisfies the definition of strongly visible actions (see Section \ref{sec:intro}). 

The properties (\ref{cond:orbit-decomp-n}) and (\ref{cond:s1-n}) coincide with 
(\ref{visible:v1}) and (\ref{visible:s1}), respectively. 
To see (\ref{visible:s2}), 
we take $Z\in \mathfrak{n}$, 
and write $Z=k\cdot Z_0$ for some $k\in L_u$ and $Z_0\in S_0$ according to (\ref{visible:v1}). 
By the relation (\ref{eq:compatibility-n}) and the property (\ref{cond:s1-n}), 
we have 
\begin{align*}
\sigma (Z)&=\sigma (k\cdot Z_0)=\widetilde{\sigma }(k)\cdot \sigma (Z_0)
=\widetilde{\sigma }(k)\cdot Z_0=\widetilde{\sigma }(k)k^{-1}\cdot Z. 
\end{align*}
The element $\widetilde{\sigma }(g)g^{-1}$ lies in $L_u$, 
from which $\sigma (Z)\in L_u\cdot Z$. 
Hence, we have verified (\ref{visible:s2}). 

Consequently, Theorem \ref{thm:linear} has been proved. 
\end{proof}



\section{Induction theorem of strongly visible actions}
\label{sec:induction}

In this section, we show Theorem \ref{thm:induction thm}. 
We again reformulate Theorem \ref{thm:induction thm} as follows: 

\begin{theorem}
\label{thm:induction}
Let $\mathcal{O}_X$ be a nilpotent orbit, and $\mathfrak{l}=\mathfrak{g}(0)$ and 
$\mathfrak{n}=\bigoplus _{m\geq 2}\mathfrak{g}(m)$ are defined by the $\mathfrak{sl}_2$-triple 
as in (\ref{eq:levi}) and (\ref{eq:nilpotent-subalgebra}), respectively. 
If the $L_u$-action on $\mathfrak{n}$ is strongly visible, 
then the $G_u$-action on $\mathcal{O}_X$ is strongly visible. 
\end{theorem}

\begin{rem}
This theorem generalize \textit{induction theorem of strongly visible actions} \cite[Theorem 20]{mftheorem} 
for Type A case. 
It produces new strongly visible actions 
out of known strongly visible actions (linear visible actions \cite{irr,red}). 
\end{rem}

Suppose that the $L_u$-action on $\mathfrak{n}$ is strongly visible. 
Then, one can take a real submanifold $S_0$ in $\mathfrak{n}$ 
such that $\mathfrak{n}=L_u\cdot S_0$, $\sigma |_{S_0}=\operatorname{id}_{S_0}$. 
We set 
\begin{align}
S:=S_0\cap \mathfrak{n}^{\circ}. 
\label{eq:slice}
\end{align}
Then, $S$ is a real submanifold of $\mathfrak{n}^{\circ}$
since $\mathfrak{n}^{\circ}$ is open in $\mathfrak{n}$. 

\begin{lemma}
\label{lem:slice-n0}
$\mathfrak{n}^{\circ}=L_u\cdot S$. 
\end{lemma}

\begin{proof}
In view of the equality $\mathfrak{n}^{\circ}=\mathfrak{n}\cap \mathfrak{n}^{\circ}$, 
we have 
\begin{align*}
\mathfrak{n}^{\circ}&=(L_u\cdot S_0)\cap \mathfrak{n}^{\circ} 
=(L_u\cdot S_0)\cap (L_u\cdot \mathfrak{n}^{\circ}) 
=L_u\cdot (S_0\cap \mathfrak{n}^{\circ})
=L_u\cdot S. 
\end{align*}
Hence, Lemma \ref{lem:slice-n0} has been proved. 
\end{proof}

Combining Proposition \ref{prop:another realization} with Lemma \ref{lem:slice-n0}, we have 
\begin{align*}
\mathcal{O}_X
=G_u\cdot \mathfrak{n}^{\circ}
=G_u\cdot (L_u\cdot S)
=(G_uL_u)\cdot S
=G_u\cdot S. 
\end{align*}

Therefore, we have proved: 

\begin{proposition}
\label{prop:v1}
The submanifold $S$ satisfies the condition (\ref{visible:v1}). 
\end{proposition}

Next, we define an anti-holomorphic diffeomorphism of $\mathcal{O}_X$, 
which arises from $\sigma$ defined by (\ref{eq:sigma}) as follows. 

Let $Z$ be an element of $\mathcal{O}_X$, 
and write $Z=g\cdot Z_0$ for some $g\in G_u$ and $Z_0\in S$ due to Proposition \ref{prop:v1}. 
It is obvious from $S\subset S_0$ and $\sigma |_{S_0}=\operatorname{id}_{S_0}$ 
that 
\begin{align}
\sigma |_S=\operatorname{id}_S. 
\label{eq:s1}
\end{align}
Then, the relation (\ref{eq:compatibility-n}) shows 
\begin{align}
\sigma (Z)=\sigma (g\cdot Z_0)
=\widetilde{\sigma }(g)\cdot \sigma (Z_0)=\widetilde{\sigma }(g)\cdot Z_0. 
\label{eq:stable-nilpotent-orbit}
\end{align}
Since $G_u$ is $\widetilde{\sigma }$-stable (see Section \ref{subsec:lift}), 
the element $\widetilde{\sigma}(g)$ lies in $G_u$. 
Then, the equality (\ref{eq:stable-nilpotent-orbit}) means 
that $\sigma (Z)\in G_u\cdot S=\mathcal{O}_X$. 
Hence, $\mathcal{O}_X$ is $\sigma$-stable. 
This implies that 
the restriction of $\sigma $ to $\mathcal{O}_X$ 
becomes an anti-holomorphic diffeomorphism on $\mathcal{O}_X$, 
which we use the same letter to denote. 

Now, 
we give a proof of Theorem \ref{thm:induction}. 

\begin{proof}[Proof of Theorem \ref{thm:induction}]
It is clear that 
(\ref{visible:v1}) and (\ref{visible:s1}) hold by Proposition \ref{prop:v1} 
and (\ref{eq:s1}), respectively. 
Let $Z\in \mathcal{O}_X$ and write $Z=g\cdot Z_0\in G_u\cdot S$. 
Then, we have 
\begin{align*}
\sigma (Z)
=\widetilde{\sigma }(g)g^{-1}\cdot (g\cdot Z_0)
=\widetilde{\sigma }(g)g^{-1} \cdot Z\in G_u\cdot Z. 
\end{align*}
Hence, we have verified (\ref{visible:s2}). 

Therefore, Theorem \ref{thm:induction} has been proved. 
\end{proof}




\section{Proof of Theorem \ref{thm:slice-n}}
\label{sec:proof-n}

This section is devoted to the proof of Theorem \ref{thm:slice-n}. 

First of all, we give a short summary of our proof. 
The Dynkin--Kostant theory explains that 
a nilpotent orbit $\mathcal{O}_X$ in a complex simple Lie algebra $\mathfrak{g}$ 
is characterized by the weighted Dynkin diagram, denoted by $\Omega (\mathcal{O}_X)$, 
which is the Dynkin diagram of $\mathfrak{g}$ with numerical labels. 
A classification of nilpotent orbits is given in terms of 
the weighted Dynkin diagrams. 
Moreover, 
$\mathcal{O}_X$ defines the height, denoted by $\operatorname{ht}(\mathcal{O}_X)$. 
D. Panyushev provides a criterion for $\mathcal{O}_X$ to be spherical 
by its height. 
Since $\operatorname{ht}(\mathcal{O}_X)$ can be calculated from $\Omega (\mathcal{O}_X)$, 
the table of spherical nilpotent orbits is obtained. 

Under these theories, 
we apply case-by-case analysis on the $L_u$-action on $\mathfrak{n}$ 
for each spherical $\mathcal{O}_X$. 
Indeed, we clarify a semisimple element 
$H\in \Phi (\mathcal{O}_X)\cap \mathfrak{a}$ from $\Omega (\mathcal{O}_X)$ 
and express the $L_{\mathbb{C}}$-action on $\mathfrak{n}$. 
By using the classification of strongly visible linear actions, 
we verify the strong visibility for the $L_u$-action on $\mathfrak{n}$, and 
give an explicit description of $S_0$ satisfying (\ref{cond:vectorspace})--(\ref{cond:s1-n}). 

\subsection{Weighted Dynkin diagram of nilpotent orbit}
\label{subsec:dynkin}

In this subsection, 
we review the weighted Dynkin diagram corresponding to a nilpotent orbit $\mathcal{O}_X$ 
in a complex semisimple Lie algebra $\mathfrak{g}$. 
See \cite{nilpotent} for survey on weighted Dynkin diagrams in 
complex semisimple Lie algebras. 

Let us retain the setting of Sections \ref{sec:preliminaries} and 
\ref{sec:visible-n}. 
Let $\mathfrak{b}$ be a Borel subalgebra of $\mathfrak{g}$ 
such that $\mathfrak{b}$ contains the Cartan subalgebra $\mathfrak{a}$ 
and is contained in the parabolic subalgebra $\mathfrak{q}$. 
We fix a positive system $\Delta ^+\equiv \Delta ^+(\mathfrak{g},\mathfrak{a})$ 
satisfying $\mathfrak{b}=\mathfrak{a}\oplus \bigoplus _{\alpha \in \Delta ^+}\mathfrak{g}_{\alpha }$ 
and set a closed Weyl chamber by 
\begin{align*}
\mathfrak{a}_+=\{ A\in \mathfrak{a}:\alpha (A)\geq 0~(\forall \alpha \in \Delta ^+)\} . 
\end{align*}

As mentioned in Section \ref{subsec:stable}, the intersection of 
$\mathfrak{a}$ and the semisimple orbit $\Phi (\mathcal{O}_X)$ is non-empty. 
In particular, $\Phi (\mathcal{O}_X)\cap \mathfrak{a}_+$ is a singleton set. 
Then, we define an injective map by 
\begin{align}
\Psi :\mathcal{N}^*/G_{\mathbb{C}}\to \mathfrak{a}_+, \quad 
\mathcal{O}_X\mapsto H\in \Phi(\mathcal{O}_X)\cap \mathfrak{a}_+. 
\label{eq:characteristic}
\end{align}

\begin{define}[cf. {\cite{pa1}}]
\label{def:characteristic}
We say that $\Psi (\mathcal{O}_X)\in \Phi (\mathcal{O}_X)\cap \mathfrak{a}_+$ is the 
\textit{characteristic element} of $\mathcal{O}_X$. 
\end{define}

For the rest of this paper, 
we fix $H$ to be $H=\Psi (\mathcal{O}_X)\in \mathfrak{a}_+$. 

We write $ \alpha _1,\ldots ,\alpha _r \in \Delta ^+$ for the simple roots of $\Delta $ 
($r:=\operatorname{rank}\mathfrak{g}=\dim _{\mathbb{C}}\mathfrak{a}$). 
As $\mathfrak{b}\subset \mathfrak{q}$, 
the number $\alpha _j(H)$ is a non-negative integer. 
Moreover, 
it follows from the representation theory that 
$\alpha _j(H)\in \{ 0,1,2\}$ for $j=1,2,\ldots ,r$. 
Then, we define the injective map as follows: 
\begin{align}
\Omega :\mathcal{N}^*/G_{\mathbb{C}}\to \{ 0,1,2\} ^{r},\quad 
\Omega (\mathcal{O}_X):=
(\alpha _1(H),\ldots ,\alpha _r(H)). 
\label{eq:weighted}
\end{align}

We label the node of the Dynkin diagram of $\mathfrak{g}$ 
corresponding to each simple root $\alpha _j$ with $\alpha _j(H)$. 
The Dynkin diagram with such labels 
is called the \textit{weighted Dynkin diagram} of $\mathcal{O}_X$. 

We recall from (\ref{eq:kostant}) that 
there is a one-to-one correspondence between nilpotent orbits and conjugacy classes 
of $\mathfrak{sl}_2$-triples. 
Thanks to Dynkin's work on $\mathfrak{sl}_2$-triples \cite{dynkin}, 
the injective map (\ref{eq:weighted}) provides a characterization of 
nilpotent orbit in $\mathfrak{g}$ by corresponding weighted Dynkin diagrams. 

Next, the $\mathbb{Z}$-grading (\ref{eq:grading}) defined by $\mathcal{O}_X$ 
introduces a function on $\mathcal{N}^*/G_{\mathbb{C}}$ as follows: 
\begin{align}
\operatorname{ht}:\mathcal{N}^*/G_{\mathbb{C}}\to \mathbb{Z},\quad 
\mathcal{O}_X\mapsto \max \{ m\in \mathbb{Z}:\mathfrak{g}(m)\neq \{ 0\} \} .
\label{eq:height}
\end{align}
Since $X\in \mathfrak{g}(2)$, we obtain $\operatorname{ht}(\mathcal{O}_X)\geq 2$ 
for any $\mathcal{O}_X\in \mathcal{N}^*/G_{\mathbb{C}}$. 

\begin{define}[{\cite[Section 2]{pa1}}]
\label{def:height}
We say that 
the positive integer $\operatorname{ht}(\mathcal{O}_X)$ 
is the \textit{height} of a nilpotent orbit $\mathcal{O}_X$. 
\end{define}

D. Panyushev gives a necessary and sufficient condition 
for a nilpotent orbit to be spherical. 

\begin{fact}[{\cite[Theorem 3.1]{pa1}}]
\label{fact:panyushev}
For a nilpotent orbit $\mathcal{O}_X$, 
the following two conditions are equivalent: 
\begin{enumerate}
	\renewcommand{\theenumi}{\roman{enumi}}
	\item $\mathcal{O}_X$ is spherical. 
	\item $\operatorname{ht}(\mathcal{O}_X)\leq 3$. 
\end{enumerate}
\end{fact}

In view of Fact \ref{fact:panyushev}, 
we consider how to calculate $\operatorname{ht}(\mathcal{O}_X)$ 
by the weighted Dynkin diagram $\Omega (\mathcal{O}_X)$. 
Let $\beta \in \Delta ^+$ be the highest root. 
We write $\beta =k_1\alpha _1+\cdots +k_r\alpha _r$ 
for some positive integers $k_1,\ldots ,k_r$. 
Then, we have: 

\begin{lemma}
\label{lem:height-max}
$\operatorname{ht}(\mathcal{O}_X)=\beta (H)=k_1\alpha _1(H)+\cdots +k_r\alpha _r(H)$. 
\end{lemma}

\begin{proof}
By the proof of Lemma \ref{lem:eigenspace}, 
we formulate 
\begin{align}
\operatorname{ht}(\mathcal{O}_X)
=\max \{ m\in \mathbb{Z}:\mathfrak{g}(m)\neq \{ 0\} \} 
=\max \{ \alpha (H):\alpha \in \Delta \}. 
\label{eq:height-root}
\end{align}

Let $\alpha \in \Delta$ and 
write $\alpha =l_1\alpha _1+\cdots +l_r\alpha _r$ ($l_1,\ldots ,l_r\in \mathbb{Z}$). 
Since $\beta $ is the highest root, 
the inequality $l_j\leq k_j$ holds for any $j=1,2,\ldots ,r$. 
As $\alpha _j(H)\geq 0$ (see (\ref{eq:weighted})), 
we estimate 
\begin{align*}
\alpha (H)
&=l_1\alpha _1(H)+\cdots +l_r\alpha _r(H)\\
&\leq k_1\alpha _1(H)+\cdots +k_r\alpha _r(H)\\
&=\beta (H). 
\end{align*}
This means that 
\begin{align}
\max \{ \alpha (H):\alpha \in \Delta \} =\beta (H). 
\label{eq:height-highest}
\end{align}

Combining (\ref{eq:height-root}) and (\ref{eq:height-highest}), 
we get $\operatorname{ht}(\mathcal{O}_X)=\beta (H)$. 
\end{proof}

Using Lemma \ref{lem:height-max}, 
we list all weighted Dynkin diagrams $\Omega (\mathcal{O}_X)$ 
with $\operatorname{ht}(\mathcal{O}_X)=2,3$ 
in the first and second columns of Tables \ref{table:classical} and \ref{table:exceptional}. 

\subsection{Visible linear actions}
\label{subsec:linear}

In this subsection, 
we recall the recent works on strongly visible linear actions, 
see \cite{irr,red} for details. 

Let $K_{\mathbb{C}}$ be a connected complex reductive Lie group 
and $V$ a vector space over $\mathbb{C}$. 
Suppose we are given a holomorphic representation of $K_{\mathbb{C}}$ on $V$. 
Then, we have naturally the representation of $K_{\mathbb{C}}$ on the polynomial ring 
$\mathbb{C}[V]$ defined by $f(v)\mapsto f(g^{-1}\cdot v)$. 
We say that the $K_{\mathbb{C}}$-action on $V$ is a \textit{multiplicity-free action}, 
or, $V$ is a \textit{multiplicity-free $K_{\mathbb{C}}$-space} 
if $\mathbb{C}[V]$ is multiplicity-free as a representation of $K_{\mathbb{C}}$. 

Multiplicity-free actions are classified 
by Kac, Benson--Ratcliff, and Leahy \cite{br,kac,leahy} up to geometrically equivalences. 
Here, 
two holomorphic representations $(\pi ,V)$ and $(\pi ',V')$ 
of connected complex reductive Lie groups 
$K_{\mathbb{C}}$ and $K_{\mathbb{C}}'$, respectively, 
are \textit{geometrically equivalent} 
if the image of $\pi$ coincides with that of $\pi '$ 
under some linear isomorphism from $V$ to $V'$. 

Let $K_u$ be a compact real form of $K_{\mathbb{C}}$. 
Then, we have: 

\begin{fact}[{\cite{irr,red}}]
\label{fact:visible-linear}
For a holomorphic representation of $K_{\mathbb{C}}$ on $V$, 
the followings are equivalent: 
\begin{enumerate}
	\renewcommand{\theenumi}{\alph{enumi}}
	\item The $K_{\mathbb{C}}$-action on $V$ is a multiplicity-free action. 
	\item The $K_u$-action on $V$ is strongly visible. 
\end{enumerate}
\end{fact}

Fact \ref{fact:visible-linear} gives a classification of strongly visible linear actions. 
During them, as we will see in the proof of Theorem \ref{thm:slice-n}, 
we need only eight series of multiplicity-free actions, which are listed in Table \ref{table:linear}. 

\begin{table}[htbp]
$
\begin{array}{ccc}
\hline 
 & K_{\mathbb{C}} & V \\
\hline 
(1) & \mathbb{C}^{\times} & \mathbb{C} \\
(2) & SL(p,\mathbb{C}) & \mathbb{C}^p \\
(3) & SL(p,\mathbb{C})\times \mathbb{C}^{\times} & \operatorname{Sym}(p,\mathbb{C}) \\
(4) & SL(2p,\mathbb{C})\times \mathbb{C}^{\times} & \operatorname{Alt}(2p,\mathbb{C}) \\
(5) & SL(p,\mathbb{C})\times SL(p,\mathbb{C})\times \mathbb{C}^{\times} & M(p,\mathbb{C}) \\
(6) & SO(p,\mathbb{C})\times \mathbb{C}^{\times} & \mathbb{C}^p \\
(7) & E_6(\mathbb{C})\times \mathbb{C}^{\times} 
	& \mathfrak{J}_{\mathbb{C}} \\
(8) & SL(2p,\mathbb{C})\times \mathbb{C}^{\times} 
	& \operatorname{Alt}(2p,\mathbb{C})\oplus \mathbb{C}^{2p} \\
\hline 
\end{array}
$
\caption{Multiplicity-free $K_{\mathbb{C}}$-action on $V$}
\label{table:linear}
\end{table}

We pin down some restrictions on integer $p$ in Table \ref{table:linear} as follows: 
In (2), (3), (5), and (6), $p\geq 2$; In (4) and (8), $p\geq 1$. 

In (1), the multiplicative group $\mathbb{C}^{\times }=GL(1,\mathbb{C})$ 
acts on $\mathbb{C}$ as the standard complex multiplication. 
The special linear group $SL(p,\mathbb{C})=\{ g\in M(p,\mathbb{C}):\det g=1\} $ 
and the special complex orthogonal group $SO(p,\mathbb{C})
=\{ g\in SL(p,\mathbb{C}):{}^tgg=I_p\} $ act 
linearly on $\mathbb{C}^p$, respectively, 
where ${}^tg$ denotes the transposed matrix of $g$. 
In (3), 
$SL(p,\mathbb{C})$ acts on the space $\operatorname{Sym}(p,\mathbb{C})$ of complex symmetric 
matrices by $g\cdot A=gA\,{}^tg$. 
In (4), 
$SL(2p,\mathbb{C})$ acts on the space $\operatorname{Alt}(2p,\mathbb{C})$ of 
complex alternating matrices by $g\cdot A=gA\,{}^tg$. 
In (5), $SL(p,\mathbb{C})\times SL(p,\mathbb{C})$ acts on $M(p,\mathbb{C})$ 
by $(g,h)\cdot A=gAh^{-1}$. 
In (7), 
$\mathfrak{J}_{\mathbb{C}}=\mathfrak{J}\otimes _{\mathbb{R}}\mathbb{C}
=\operatorname{Herm}(3,\mathfrak{C})\otimes _{\mathbb{R}}\mathbb{C}
=\operatorname{Herm}(3,\mathfrak{C}_{\mathbb{C}})$ 
is the complexified exceptional Jordan algebra, 
namely, 
it consists of Hermitian matrices of degree three 
whose entries are the complexified Cayley algebra $\mathfrak{C}_{\mathbb{C}}
=\mathfrak{C}\otimes _{\mathbb{R}}\mathbb{C}$. 
Then, $\mathfrak{J}_{\mathbb{C}}$ is a vector space over $\mathbb{C}$ 
with dimension $27$. 
We denote by $E_6(\mathbb{C})$ the connected and simply connected complex simple Lie group 
of exceptional type. 
Then, $E_6(\mathbb{C})$ acts on $\mathfrak{J}_{\mathbb{C}}$ as automorphisms. 

In (3)--(7), the center $\mathbb{C}^{\times}$ of $K_u$ acts on $V$ as the scalar multiplication. 
In (8), 
the semisimple part $SL(2p,\mathbb{C})$ of $K_u$ 
acts on $\mathbb{C}^{2p}\oplus \operatorname{Alt}(2p,\mathbb{C})$ 
by $g\cdot (v,A)=(gv,gA\,{}^tg)$, and the center $\mathbb{C}^{\times}$ acts by 
$s\cdot (v,A)=(s^3v,s^2A)$. 

For $(K_{\mathbb{C}},V)$ in Table \ref{table:linear}, we present: 

\begin{lemma}[{\cite{irr,red}}]
\label{lem:linear}
Let a multiplicity-free $K_{\mathbb{C}}$-action on $V$ be one of cases in Table \ref{table:linear}. 
Then, one can take a real vector subspace $T$ and an anti-holomorphic diffeomorphism $\sigma $ 
for the strongly visible $K_u$-action on $V$ satisfying the following conditions: 
\begin{enumerate}
	\renewcommand{\theenumi}{\alph{enumi}}
	\item $V=K_u\cdot T$. 
	\label{item:orbit}
	
	\item $\sigma |_T=\operatorname{id}_T$. 
	\label{item:sigma}
	
	\item The dimension of the vector space $T$ over $\mathbb{R}$ is equal to 
	the support of the semigroup of highest weights occurring in $\mathbb{C}[V]$. 
\end{enumerate}
\end{lemma}

Indeed, 
we choose $T$ and $\sigma $ as in Table \ref{table:choice}. 
Then, we can verify that 
$(T,\sigma )$ satisfies Lemma \ref{lem:linear}. 

\begin{table}[htbp]
$
\begin{array}{cccc}
\hline 
(K_{\mathbb{C}},V) & T & \sigma & \widetilde{\sigma} \\
\hline 
(1) & \mathbb{R} & \sigma _1 & \widetilde{\sigma }_0 \\
(2) & T_1 & \sigma _1 & \widetilde{\sigma}_1 \\
(3) & D_p & \sigma _1 & \widetilde{\sigma}_1\boxtimes \widetilde{\sigma}_0\\
(4) & A_p & \sigma _1 & \widetilde{\sigma}_1\boxtimes \widetilde{\sigma}_0\\
(5) & D_p & \sigma _1 
	& \widetilde{\sigma}_1\boxtimes \widetilde{\sigma}_1\boxtimes \widetilde{\sigma}_0\\
(6) & D_{1,1} & \sigma _2 & \widetilde{\sigma}_2\boxtimes \widetilde{\sigma}_0\\
(7) & D_3 & \sigma _1 & \widetilde{\sigma}_1\boxtimes \widetilde{\sigma}_0\\
(8) & A_p\oplus T_p & \sigma _1\oplus \sigma _1 
	& \widetilde{\sigma}_1\boxtimes \widetilde{\sigma}_0\\
\hline 
\end{array}
$
\caption{Our choice of $T$, $\sigma$, $\widetilde{\sigma}$ for the $K_u$-action on $V$}
\label{table:choice}
\end{table}

Here, let us explain the notation used in Table \ref{table:choice} as follows. 

First, 
let $\{ e_1,\ldots ,e_N\} $ be the standard basis of $\mathbb{C}^N$. 
We define two real subspaces $T_p,D_{1,1}$ in $\mathbb{C}^{N}$ by 
\begin{align*}
T_p&:=\mathbb{R}e_1\oplus \mathbb{R}e_3\oplus \mathbb{R}e_5\oplus \cdots \oplus \mathbb{R}e_{2N'-1}, \\
D_{1,1}&:=\mathbb{R}e_1\oplus \sqrt{-1}\mathbb{R}e_2, 
\end{align*}
where $N':=\lfloor \frac{N+1}{2}\rfloor $ denotes the maximam of integers which are not greater than $\frac{N+1}{2}$. 
Second, 
we denote by $D_N$ the real subspace of $M(N,\mathbb{C})$ 
consisting of diagonal matrices whose entries are all real, namely, 
\begin{align*}
D_N&:=\{ \operatorname{diag}(r_1,\ldots ,r_N)\in M(N,\mathbb{C}):
	r_1,\ldots ,r_N\in \mathbb{R}\} . 
\end{align*}
Third, 
we set a real subspace $A_{p}$ in $\operatorname{Alt}(2p,\mathbb{C})$ by 
\begin{align*}
A_p&:=\{ J(r_1,\ldots ,r_p)\in \operatorname{Alt}(2p,\mathbb{C}):
	r_1,\ldots ,r_p\in \mathbb{R}\} . 
\end{align*}
where $J(r_1,\ldots ,r_p)\in \operatorname{Alt}(2p,\mathbb{C})$ stands for the following 
block diagonal matrix 
\begin{align*}
J(r_1,\ldots ,r_p):=
\operatorname{diag}(r_1J_1,\ldots ,r_pJ_1),\quad 
J_1=\left( 
	\begin{array}{cc}
	0 & -1 \\
	1 & 0
	\end{array}
\right) . 
\end{align*}

On the other hand, 
we denote by $\sigma _1$ and $\sigma _2$ the complex conjugations of $\mathbb{C}^N$ 
with respect to real forms $V_1$ and $V_2$, respectively, where 
$N'=\lfloor \frac{N+1}{2}\rfloor $, $N''=\lfloor \frac{N}{2}\rfloor $ and 
\begin{align*}
V_1&:=\bigoplus _{i=1}^N\mathbb{R}e_i ,\quad 
V_2:=\bigoplus _{i=1}^{N'} \mathbb{R}e_{2i-1}\oplus 
\bigoplus _{i=1}^{N''} \sqrt{-1}\mathbb{R}e_{2i}. 
\end{align*}
By using the coordinate with respect to the standard basis $\{ e_1,\ldots ,e_N\} $, 
we write $\sigma _1$ and $\sigma _2$, respectively, as 
\begin{align*}
\sigma _1(v)&=\overline{v},\quad \sigma _2(v)=I_{a}\overline{v}\quad 
(v\in \mathbb{C}^N). 
\end{align*}
where $\varepsilon _i:=(-1)^{i+1}~(i=1,2,\ldots ,N)$ and 
\begin{align}
I_{a}:=\operatorname{diag}(\varepsilon _1,\ldots ,\varepsilon _N)\in M(N,\mathbb{C}). 
\label{eq:anti-diag}
\end{align}

For the standard basis $\{ E_{ij}:1\leq i,j\leq N\} $ of $M(N,\mathbb{C})$, 
we also define the standard real form 
$\bigoplus _{1\leq i,j\leq N}\mathbb{R}E_{ij}=M(N,\mathbb{R})$. 
With respect to $M(N,\mathbb{R})$, 
we define the complex conjugation of $M(N,\mathbb{C})$, 
which we use the same notation $\sigma _1 $ to denote. 

Suppose we take $T$ and $\sigma$ as in Table \ref{table:choice} 
for a multiplicity-free $K_{\mathbb{C}}$-action on $V$ 
listed in Table \ref{table:linear}. 
Then, we have: 

\begin{lemma}
\label{lem:linear2}
There exists a compatible automorphism $\widetilde{\sigma }\in \operatorname{Aut}K_{\mathbb{C}}$ 
with respect to $\sigma $ for the $K_{\mathbb{C}}$-action on $V$ 
(see Definition \ref{def:compatible}) 
such that $\widetilde{\sigma}$ stabilizes $K_u$ and 
$\operatorname{rank}_{\mathbb{R}}\operatorname{Lie}(K_{\mathbb{C}}^{\widetilde{\sigma}})
=\operatorname{rank}\operatorname{Lie}(K_{\mathbb{C}})$. 
\end{lemma}

\begin{proof}
For convenience, 
we let put $K_1:=SL(p,\mathbb{C})$ and $K_2:=SO(p,\mathbb{C})$. 
By using our matrix realization of $SL(p,\mathbb{C})$ 
and $SO(p,\mathbb{C})$, 
we define anti-holomorphic involutions $\widetilde{\sigma }_{i}$ 
on $K_i$ $(i=1,2)$ by 
\begin{align*}
\widetilde{\sigma }_{1}(k)&:=\overline{k}\quad (k\in K_1),\quad 
\widetilde{\sigma }_{2}(k):=I_{a}^{-1}\overline{k}I_{a}\quad (k\in K_2). 
\end{align*}
Then, the fixed point set $K_1^{\widetilde{\sigma }_{1}}$ coincides with 
$SL(p,\mathbb{R})$, and $K_2^{\widetilde{\sigma }_{2}}$ 
is isomorphic to the indefinite special orthogonal group 
$SO(\lfloor \frac{p}{2}\rfloor ,\lfloor \frac{p+1}{2}\rfloor )$. 
Clearly, 
$\operatorname{rank}_{\mathbb{R}}\operatorname{Lie}(K_1^{\widetilde{\sigma }_{1}})
=p=\operatorname{rank}\operatorname{Lie}(K_1)$ and 
$\operatorname{rank}_{\mathbb{R}}\operatorname{Lie}(K_2^{\widetilde{\sigma }_{2}})
=\lfloor \frac{p}{2}\rfloor =\operatorname{rank}\operatorname{Lie}(K_2)$. 
Similarly, we define an anti-holomorphic involution $\widetilde{\sigma }_{0}$ 
on $\mathbb{C}^{\times}$ 
by $\widetilde{\sigma }_{0}(s)=\overline{s}$ $(s\in \mathbb{C}^{\times})$. 
Then, $(\mathbb{C}^{\times})^{\widetilde{\sigma }_{0}}=\mathbb{R}^{\times}$. 

The right column of Table \ref{table:choice} gives our choice of 
anti-holomorphic involution $\widetilde{\sigma }$ on $K_{\mathbb{C}}$ 
for each strongly visible $K_u$-action on $V$. 
The direct computation shows that $\widetilde{\sigma }$ 
satisfies Lemma \ref{lem:linear2}. 
\end{proof}

Concerning to our choice of $\widetilde{\sigma }$ as in Table \ref{table:choice}, 
we denote by $\mu \boxtimes \mu '$ 
for involutions $\mu $ and $ \mu '$ on complex Lie groups $H_{\mathbb{C}}$ and $H_{\mathbb{C}}'$, 
respectively, 
the involution on 
$H_{\mathbb{C}}\times H_{\mathbb{C}}'$ defined by 
$(\mu \boxtimes \mu ')(h,h')=(\mu (h),\mu '(h'))$. 

\begin{rem}
Lemma \ref{lem:linear2} holds 
for all multiplicity-free actions, 
on which we will discuss in forthcoming paper \cite{compatible}. 
\end{rem}

\subsection{Procedure}
\label{subsec:notation}

In this subsection, 
we explain the procedure of our proof of Theorem \ref{thm:slice-n}. 

The standard basis $\{ e_1,\ldots ,e_N\} $ of $\mathbb{C}^N$ 
defines the standard real form 
$V_{\mathbb{R}}=V_1=\mathbb{R}^N$ 
and the standard Hermitian inner product $\langle \cdot ,\cdot \rangle $ 
satisfying $\langle e_i,e_j\rangle =\delta _{ij}$ $(1\leq i,j\leq N)$. 
The dual $\mathfrak{a}_{\mathbb{R}}^*$ of the Cartan subalgebra 
$\mathfrak{a}_{\mathbb{R}}\subset \mathfrak{g}_{\mathbb{R}}$ is realized 
as a subspace of $V_{\mathbb{R}}$, denoted by $V$. 
Let $\{ E_1,\ldots ,E_N\} $ be the dual basis of $\{ e_1,\ldots ,e_N\}$ 
with $\langle e_i,E_j\rangle =\delta _{ij}$ $(1\leq i,j\leq N)$. 
Then, $\mathfrak{a}_{\mathbb{R}}$ is isomorphic to 
$V^*\subset V_{\mathbb{R}}^*=\mathbb{R}E_1\oplus \cdots \oplus \mathbb{R}E_N$. 
Hence, $\mathfrak{a}=\mathfrak{a}_{\mathbb{R}}+\sqrt{-1}\mathfrak{a}_{\mathbb{R}}
\simeq V^*+\sqrt{-1}V^*$. 

We have seen in Lemma \ref{lem:rootspace} that 
the complex conjugation $\sigma$ with respect to $\mathfrak{g}_{\mathbb{R}}$ 
stabilizes the root space $\mathfrak{g}_{\alpha }$ $(\alpha \in \Delta )$. 
Then, we decompose $\mathfrak{g}_{\alpha}$ into the $\sigma$-eigenspaces as 
$\mathfrak{g}_{\alpha }=(\mathfrak{g}_{\alpha }\cap \mathfrak{g}_{\mathbb{R}})
+(\mathfrak{g}_{\alpha }\cap \sqrt{-1}\mathfrak{g}_{\mathbb{R}})$, 
equivalently, $\mathfrak{g}_{\alpha }$ is the complexification of the Lie algebra 
$\mathfrak{g}_{\alpha }\cap \mathfrak{g}_{\mathbb{R}}$. 
As $\dim _{\mathbb{C}}\mathfrak{g}_{\alpha }=1$, 
we have $\dim (\mathfrak{g}_{\alpha }\cap \mathfrak{g}_{\mathbb{R}})=1$. 
Thus, we express $\mathfrak{g}_{\alpha }\cap \mathfrak{g}_{\mathbb{R}}=\mathbb{R}E_{\alpha }$ 
for some root vector $E_{\alpha }$. 

Under this setting, we carry out for each spherical nilpotent orbit $\mathcal{O}_X$ 
in $\mathfrak{g}$ as follows: 
\begin{enumerate}
	\renewcommand{\labelenumi}{\theenumi. }
	\item Specify the characteristic element $H\in \mathfrak{a}_+$ 
	from the corresponding weighted Dynkin diagram $\Omega (\mathcal{O}_X)$ 
	(see (\ref{eq:weighted})). 
	\item Write $\mathfrak{l}=\mathfrak{g}(0)$, 
	$\mathfrak{g}(2)$, and $\mathfrak{g}(3)$, respectively, 
	as a direct sum of root spaces 
	(see Lemma \ref{lem:eigenspace}). 
	\item Verify that the $L_{\mathbb{C}}$-action on $\mathfrak{n}$ 
	is a multiplicity-free action comparing with Table \ref{table:linear}. 
	Then, 
	the $L_u$-action on $\mathfrak{n}$ is strongly visible (Fact \ref{fact:visible-linear}). 
	\item Give a slice $S_0$ for the $L_u$-action on $\mathfrak{n}$ explicitly 
	by using Table \ref{table:choice}. 
	In particular, 
	describe $S_0$ as 
	$S_0=\bigoplus _{\alpha \in \Delta ^+(\mathcal{O}_X)}\mathbb{R}E_{\alpha }$ 
	for some subset $\Delta ^+(\mathcal{O}_X)$ in $\Delta ^+$. 
\end{enumerate}
Owing to Lemmas \ref{lem:linear} and \ref{lem:linear2}, 
Theorem \ref{thm:slice-n} holds for the subspace 
$S_0$ which is constructed according to the above procedure. 



\subsection{Type A$_{n-1}$}
\label{subsec:a}

We begin with the case $\mathfrak{g}=\mathfrak{sl}(n,\mathbb{C})$ 
for integer $n\geq 2$. 
In this case, $\mathfrak{g}_{\mathbb{R}}=\mathfrak{sl}(n,\mathbb{R})$. 
Then, $\mathfrak{a}_{\mathbb{R}}^*=\{ a_1e_1+\cdots +a_ne_n:a_1,\ldots ,a_n\in \mathbb{R},~
a_1+\cdots +a_n=0\} $. 
A root system $\Delta \equiv \Delta (\mathfrak{g},\mathfrak{a})$ 
is $\Delta =\{ \pm (e_i-e_j):1\leq i<j\leq n\} $. 
We fix a positive system as $\Delta ^+=\{ e_i-e_j:1\leq i<j\leq n\} $. 
The simple roots 
$\alpha _1,\ldots ,\alpha _{n-1}$ are given by $\alpha _i=e_i-e_{i+1}$ $(1\leq i\leq n-1)$. 
The highest root $\beta $ is written as $\beta =e_1-e_n=\alpha _1+\cdots +\alpha _{n-1}$. 

Let $\mathcal{O}_X$ be a nilpotent orbit with 
characteristic element $H=h_1E_1+\cdots +h_nE_n\in \mathfrak{a}_+$ 
where 
$h_1+\cdots +h_n=0$ and $h_1\geq h_2\geq \cdots \geq h_n$. 
Then, $\alpha _i(H)=h_i-h_{i+1}$ $(1\leq i\leq n-1)$. 
Hence, the weighted Dynkin diagram 
$\Omega (\mathcal{O}_X)=(m_1,m_2,\ldots ,m_{n-1})$ is given 
by $(h_1-h_2,h_2-h_3,\ldots ,h_{n-1}-h_n)$. 

\begin{figure}[htbp]
\begin{align*}
\SelectTips{cm}{12}
	\objectmargin={1pt}
	\xygraph{
		\circ ([]!{+(0,-.3)} {\alpha_1},[]!{+(0,+.3)} {m_1}) - [r] 
		\circ ([]!{+(0,-.3)} {\alpha_2},[]!{+(0,+.3)} {m_2}) - [r] 
		\cdots - [r] 
		\circ ([]!{+(0,-.3)} {\alpha_{n-2}},[]!{+(0,+.3)} {m_{n-2}}) - [r] 
		\circ ([]!{+(0,-.3)} {\alpha_{n-1}},[]!{+(0,+.3)} {m_{n-1}}) 
	}
\end{align*}
\caption{Weighted Dynkin diagram for $\mathfrak{sl}(n,\mathbb{C})$}
\label{fig:A}
\end{figure}

A nilpotent orbit $\mathcal{O}_X$ in $\mathfrak{sl}(n,\mathbb{C})$ is spherical 
if and only if $\Omega (\mathcal{O}_X)$ coincides 
with either (A) or (A$'$): 
\begin{enumerate}
	\item[(A)] 
	$\Omega (\mathcal{O}_X)=(\underbrace{0,\ldots ,0}_{p-1},1,0,\ldots 0,1,
	\underbrace{0,\ldots ,0}_{p-1})$
	for $1\leq p<\frac{n}{2}$, namely, 
	$m_p=m_{n-p}=1$ and $m_i=0$ for $i\neq p,n-p$. 
	\item[(A$'$)] 
	$n=2p$ and 
	$\Omega (\mathcal{O}_X)=(\underbrace{0,\ldots ,0}_{p-1},2,
	\underbrace{0,\ldots ,0}_{p-1})$, namely, 
	$m_p=2$ and $m_i=0$ for $i\neq p$. 
\end{enumerate}

For each case, it follows from Lemma \ref{lem:height-max} 
that its height $\operatorname{ht}(\mathcal{O}_X)$ equals two. 
Then, the nilpotent subalgebra $\mathfrak{n}$ coincides with $\mathfrak{g}(2)$. 

\subsubsection{Case (A)}
\label{subsubsec:a}

Let $\Omega (\mathcal{O}_X)$ satisfy Case (A) for $1\leq p<\frac{n}{2}$. 
Since $h_p-h_{p+1}=h_{n-p}-h_{n-p+1}=1$ and $h_i-h_{i+1}=0$ $(i\neq p,n-p)$, 
$H$ forms 
\begin{align*}
H=(E_1+\cdots +E_p)-(E_{n-p+1}+\cdots +E_n). 
\end{align*}
The nilpotent orbit $\mathcal{O}_X$ with the above $\Omega (\mathcal{O}_X)$ 
consists of complex matrices of degree $n$ with Jordan type $(2^p,1^{n-2p})$. 

The Levi subalgebra $\mathfrak{l}=\mathfrak{g}(0)$ is given as follows: 
\begin{multline*}
\mathfrak{l}=\mathfrak{a}\oplus \bigoplus _{1\leq i<j\leq p}\mathfrak{g}_{\pm (e_i-e_j)}
	\oplus \bigoplus _{p+1\leq i<j\leq n-p}\mathfrak{g}_{\pm (e_{i}-e_{j})}
	\oplus \bigoplus _{n-p+1\leq i<j\leq n}\mathfrak{g}_{\pm (e_{i}-e_{j})}. 
\end{multline*}
This means that 
\begin{align*}
\mathfrak{l}&\simeq \mathfrak{sl}(p,\mathbb{C})\oplus \mathfrak{sl}(n-2p,\mathbb{C})
	\oplus \mathfrak{sl}(p,\mathbb{C})\oplus \mathbb{C}^2. 
\end{align*}

The $\operatorname{ad}(H)$-eigenspace $\mathfrak{g}(2)$ is written as 
\begin{align*}
\mathfrak{g}(2)&=\bigoplus _{1\leq i,j\leq p}\mathfrak{g}_{e_i-e_{n-p+j}}. 
\end{align*}
Then, $\mathfrak{g}(2)$ is isomorphic to $M(p,\mathbb{C})$, 
from which 
\begin{align}
\mathfrak{n}&\simeq M(p,\mathbb{C}). 
\label{eq:n-a}
\end{align}

The semisimple part 
$SL(p,\mathbb{C})\times SL(n-2p,\mathbb{C})\times SL(p,\mathbb{C})$ 
of the Levi subgroup $L_{\mathbb{C}}$ acts on $M(p,\mathbb{C})$ by 
\begin{align*}
(g_1,h,g_2)\cdot A
=g_1Ag_2^{-1}, 
\end{align*}
and the center $(\mathbb{C}^{\times})^2$ of $L_{\mathbb{C}}$ as the scalar multiplication 
as follows: 
\begin{align*}
(s,t)\cdot A=stA. 
\end{align*}
Then, 
the $L_{\mathbb{C}}$-action on $\mathfrak{n}$ is geometrically equivalent to 
the irreducible action 
of $SL(p,\mathbb{C})\times SL(p,\mathbb{C})\times \mathbb{C}^{\times}$ on $M(p,\mathbb{C})$. 
It follows from (5) of Table \ref{table:linear} that 
this action is a multiplicity-free action. 

We take the subset $S_0$ in $\mathfrak{n}$ as 
\begin{align}
S_0:&=\bigoplus _{1\leq i\leq p} \mathbb{R}E_{e_i-e_{n-p+i}}. 
\label{eq:slice-a}
\end{align}
Then, $S_0$ is isomorphic to the slice $D_p$ of Table \ref{table:choice} 
for the strongly visible $(SU(p)\times SU(p)\times \mathbb{T})$-action 
on $M(p,\mathbb{C})$. 
By Lemma \ref{lem:linear}, 
the vector space $S_0$ satisfies $\mathfrak{n}=L_u\cdot S_0$. 
Therefore, we have verified Theorem \ref{thm:slice-n} for Case (A). 

\subsubsection{Case (A$'$)}

In case of $n=2p$ and 
$\Omega (\mathcal{O}_X)$ satisfies Case (A$'$), 
the characteristic element $H\in \mathfrak{a}$ is of the form 
\begin{align*}
H=(E_1+\cdots +E_p)-(E_{p+1}+\cdots +E_{2p}). 
\end{align*}
Then, $\mathfrak{l}\simeq \mathfrak{sl}(p,\mathbb{C})\oplus \mathfrak{sl}(p,\mathbb{C})\oplus \mathbb{C}^2$ 
and $\mathfrak{n}\simeq M(p,\mathbb{C})$. 
Hence, the $L_{\mathbb{C}}$-action on $\mathfrak{n}$ is a multiplicity-free action. 
Therefore, 
the subset $S_0$ defined by (\ref{eq:slice-a}) satisfies Theorem \ref{thm:slice-n}, 
from which we have verified for Case (A$'$). 


\subsection{Type B$_n$}
\label{subsec:b}

In this subsection, 
we give a proof of Theorem \ref{thm:slice-n} for $\mathfrak{g}=\mathfrak{so}(2n+1,\mathbb{C})$ 
for positive integer $n\geq 2$. 
In this case, 
$\mathfrak{g}_{\mathbb{R}}$ is isomorphic to $\mathfrak{so}(n+1,n)$. 
Then, we have 
$\mathfrak{a}^*_{\mathbb{R}}=V_{\mathbb{R}}$. 
A root system $\Delta \equiv \Delta (\mathfrak{g},\mathfrak{a})$ is 
$\Delta =\{ \pm e_i\pm e_j:1\leq i<j\leq n\} \sqcup \{ \pm e_i:1\leq i\leq n\} $. 
We fix a positive system as $\Delta ^+=\{ e_i\pm e_j:1\leq i<j\leq n\} \sqcup \{ e_i:
1\leq i\leq n\} $. 
The simple roots $\alpha _1,\ldots ,\alpha _n$ are 
given by 
$\alpha _i=e_i-e_{i+1}$ $(1\leq i\leq n-1)$, and $\alpha _n=e_n$. 
The highest root $\beta $ is written as 
$\beta =e_1+e_2=\alpha _1+2\alpha _2+\cdots +2\alpha _n$. 

Let $\mathcal{O}_X$ be a nilpotent orbit with characteristic element $H=h_1E_1+\cdots +h_nE_n\in \mathfrak{a}_+$ 
with $h_1\geq \cdots \geq h_n\geq 0$. 
Then, we have 
$\alpha _i(H)=h_i-h_{i+1}\ (1\leq i\leq n-1)$, and 
$\alpha _n(H)=h_n$. 
Thus, the weighted Dynkin diagram $\Omega (\mathcal{O}_X)
=(m_1,\ldots ,m_{n-1},m_n)$ is given by 
$(h_1-h_2,\ldots ,h_{n-1}-h_n,h_n)$. 

\begin{figure}[htbp]
\begin{align*}
\SelectTips{cm}{12}
	\objectmargin={1pt}
	\xygraph{!~:{@{=>}}
		\circ ([]!{+(0,-.3)} {\alpha_1},[]!{+(0,+.3)} {m_1}) - [r] 
		\circ ([]!{+(0,-.3)} {\alpha_2},[]!{+(0,+.3)} {m_2}) - [r] \cdots - [r]
		\circ ([]!{+(0,-.3)} {\alpha_{n - 1}},[]!{+(0,+.3)} {m_{n-1}}) : [r] 
		\circ ([]!{+(0,-.3)} {\alpha_n},[]!{+(0,+.3)} {m_n})
	}
\end{align*}
\caption{Weighted Dynkin diagram of $\mathcal{O}_X$ in $\mathfrak{so}(2n+1,\mathbb{C})$}
\end{figure}

A nilpotent orbit $\mathcal{O}_X$ in $\mathfrak{so}(2n+1,\mathbb{C})$ 
is spherical if and only if $\Omega (\mathcal{O}_X)$ forms one of the following cases: 
\begin{enumerate}
	\renewcommand{\labelenumi}{(B\theenumi) }
	\item $\Omega (\mathcal{O}_X)=(2,0,\ldots ,0)$, 
	namely, $m_1=2$ and $m_i=0~(i\neq 1)$. 
	\label{case:b1}
	
	\item $\Omega (\mathcal{O}_X)=(\underbrace{0,\ldots ,0}_{2p-1},1,0,\ldots ,0)$ 
	for $1\leq p\leq \frac{n}{2}$, 
	namely, $m_{2p}=1$ and $m_i=0~(i\neq 2p)$. 
	\label{case:b2}
	
	\item $\Omega (\mathcal{O}_X)=(1,\underbrace{0,\ldots ,0}_{2p-1},1,0,\ldots ,0)$ 
	for $1\leq p\leq \frac{n-1}{2}$, 
	namely, 
	$m_1=m_{2p+1}=1$ and $m_i=0~(i\neq 1,2p+1)$. 
	\label{case:b3}
\end{enumerate}

By Lemma \ref{lem:height-max}, 
its height $\operatorname{ht}(\mathcal{O}_X)$ equals two for Cases (B\ref{case:b1}), (B\ref{case:b2}), 
and three for Case (B\ref{case:b3}).

\subsubsection{Case (B\ref{case:b1})}

Let us consider the case $\Omega (\mathcal{O}_X)=(2,0,\ldots ,0)$. 
Then, $H$ is given by 
\begin{align*}
H=2E_1. 
\end{align*}
This $\mathcal{O}_X$ consists of complex matrices with Jordan type $(3,1^{2n-2})$. 

The Levi subalgebra $\mathfrak{l}=\mathfrak{g}(0)$ is given by 
\begin{align*}
\mathfrak{l}
=\mathfrak{a}\oplus \bigoplus _{2\leq i<j\leq n}\mathfrak{g}_{\pm e_{i}\pm e_{j}}
	\oplus \bigoplus _{2\leq i\leq n}\mathfrak{g}_{\pm e_{i}}. 
\end{align*}
Then, 
\begin{align*}
\mathfrak{l}\simeq \mathfrak{so}(2n-1,\mathbb{C})\oplus \mathbb{C}. 
\end{align*}

The $\operatorname{ad}(H)$-eigenspace $\mathfrak{g}(2)$ is written as 
\begin{align*}
\mathfrak{g}(2)=\bigoplus _{2\leq j\leq n}\mathfrak{g}_{e_1\pm e_{j}}
\oplus \mathfrak{g}_{e_1}. 
\end{align*}
Then, $\mathfrak{g}(2)$ is isomorphic to $\mathbb{C}^{2n-1}$. 
As $\operatorname{ht}(\mathcal{O}_X)=2$, 
the nilpotent subalgebra $\mathfrak{n}$ coincides with $\mathfrak{g}(2)$, namely, 
\begin{align*}
\mathfrak{n}\simeq \mathbb{C}^{2n-1}. 
\end{align*}

The semisimple part $SO(2n-1,\mathbb{C})$ of the Levi subgroup $L_{\mathbb{C}}$ acts 
on $\mathbb{C}^{2n-1}$ as the standard action, namely, 
\begin{align*}
g\cdot v=gv, 
\end{align*}
and its center $\mathbb{C}^{\times}$ acts as the scalar multiplication. 
This implies that the $L_{\mathbb{C}}$-action on $\mathfrak{n}$ is geometrically equivalent 
to the $(SO(2n-1,\mathbb{C})\times \mathbb{C}^{\times})$-action on $\mathbb{C}^{2n-1}$. 
It follows from (6) of Table \ref{table:linear} that 
this action is a multiplicity-free action. 

We take the subset $S_0$ in $\mathfrak{n}$ as 
\begin{align*}
S_0:=\mathbb{R}E_{e_1+e_2}\oplus \mathbb{R}E_{e_1-e_2}. 
\end{align*}
Then, $S_0$ is isomorphic to the slice $D_{1,1}$ of Table \ref{table:choice} 
for the strongly visible $(SO(2n-1)\times \mathbb{T})$-action on $\mathbb{C}^{2p-1}$. 
By Lemma \ref{lem:linear}, 
$S_0$ satisfies $\mathfrak{n}=L_u\cdot S_0$. 
Therefore, we have verified Theorem \ref{thm:slice-n} for Case (B\ref{case:b1}). 

\subsubsection{Case (B\ref{case:b2})}
\label{subsubsec:b}

Let $\Omega (\mathcal{O}_X)$ satisfy $m_{2p}=1$ and $m_i=0~(i\neq 2p)$ for $1\leq p\leq \frac{n}{2}$. 
Then, 
\begin{align*}
H=E_1+E_2+\cdots +E_{2p}. 
\end{align*}
This $\mathcal{O}_X$ consists of complex matrices with Jordan type $(2^{2p},1^{2n-4p+1})$. 
In particular, $\mathcal{O}_X$ with 
$\Omega (\mathcal{O}_X)=(0,1,0,0,\ldots ,0)$ ($p=1$) 
is the minimal nilpotent orbit. 

The Levi subalgebra $\mathfrak{l}=\mathfrak{g}(0)$ is given by 
\begin{align*}
\mathfrak{l}=\mathfrak{a}\oplus \bigoplus _{1\leq i<j\leq 2p}\mathfrak{g}_{\pm (e_i-e_j)} 
	\oplus \bigoplus _{2p+1\leq i<j\leq n}\mathfrak{g}_{\pm e_{i}\pm e_{j}}
	\oplus \bigoplus _{2p+1\leq i\leq n}\mathfrak{g}_{\pm e_{i}}
\end{align*}
This means that 
\begin{align*}
\mathfrak{l}\simeq \mathfrak{sl}(2p,\mathbb{C})\oplus \mathfrak{so}(2n-4p+1,\mathbb{C})
	\oplus \mathbb{C}, 
\end{align*}

The $\operatorname{ad}(H)$-eigenspace $\mathfrak{g}(2)$ is written as 
\begin{align*}
\mathfrak{g}(2)=\bigoplus _{1\leq i<j\leq 2p}\mathfrak{g}_{e_i+e_j}
\end{align*}
Then, $\mathfrak{g}(2)$ is isomorphic to $\operatorname{Alt}(2p,\mathbb{C})$ 
As $\operatorname{ht}(\mathcal{O}_X)=2$, 
\begin{align*}
\mathfrak{n}=\mathfrak{g}(2)\simeq \operatorname{Alt}(2p,\mathbb{C}). 
\end{align*}

The semisimple part of the Levi subgroup $L_{\mathbb{C}}$ is isomorphic to 
$SL(2p,\mathbb{C})\times SO(2n-4p+1,\mathbb{C})$. 
Then, $SL(2p,\mathbb{C})$ acts on $\operatorname{Alt}(2p,\mathbb{C})$ by 
\begin{align*}
g\cdot A=gA\,{}^t\!g,
\end{align*}
and $SO(2n-4p+1,\mathbb{C})$ acts trivially. 
Further, its center $\mathbb{C}^{\times}$ acts as the scalar multiplication. 
This implies that the $L_{\mathbb{C}}$-action on $\mathfrak{n}$ 
is geometrically equivalent to 
the action of $SL(2p,\mathbb{C})\times \mathbb{C}^{\times}$ on $\operatorname{Alt}(2p,\mathbb{C})$. 
It follows from (4) of Table \ref{table:linear} that 
this action is a multiplicity-free action. 

We take $S_0$ as 
\begin{align*}
S_0:=\bigoplus _{1\leq i\leq p} \mathbb{R}E_{e_{2i-1}+e_{2i}}. 
\end{align*}
Then, $S_0$ is isomorphic to the slice $A_p$ of Table \ref{table:choice} 
for the $(SU(2p)\times \mathbb{T})$-action on $\operatorname{Alt}(2p,\mathbb{C})$. 
By Lemma \ref{lem:linear}, 
Theorem \ref{thm:slice-n} holds for Case (B\ref{case:b2}). 

\subsubsection{Case (B\ref{case:b3})}

Let $\Omega (\mathcal{O}_X)$ satisfy $m_1=m_{2p+1}=1$ and 
$m_i=0$ for $1\leq p\leq \frac{n-1}{2}$. 
Then, $\operatorname{ht}(\mathcal{O}_X)=3$ and 
\begin{align*}
H=2E_1+E_2+E_3+\cdots +E_{2p+1}. 
\end{align*}
This $\mathcal{O}_X$ consists of complex matrices with Jordan type $(3,2^{2p},1^{2n-4p-2})$. 

We divide Case (B\ref{case:b3}) into two cases: $n\neq 2p-1$; and $n=2p-1$. 

First, let us consider the general case $n\neq 2p-1$. 
Then, the Levi subalgebra $\mathfrak{l}=\mathfrak{g}(0)$ is given by 
\begin{multline*}
\mathfrak{l}=\mathfrak{a}\oplus 
	\bigoplus _{2\leq i<j\leq 2p+1}\mathfrak{g}_{\pm (e_{i}-e_{j})}
	\oplus 
	\bigoplus _{2p+2\leq i<j\leq n}\mathfrak{g}_{\pm e_{i}\pm e_{j}}
	\oplus \bigoplus _{2p+2\leq i\leq n}\mathfrak{g}_{\pm e_{i}}. 
\end{multline*}
This means that 
\begin{align*}
\mathfrak{l}\simeq \mathfrak{sl}(2p,\mathbb{C})\oplus \mathfrak{so}(2n-4p-1,\mathbb{C})\oplus 
	\mathbb{C}^2. 
\end{align*}

The $\operatorname{ad}(H)$-eigenspace $\mathfrak{g}(2)$ is written as 
\begin{align*}
\mathfrak{g}(2)&=\mathfrak{g}_{e_1}
	\oplus \bigoplus _{2p+2\leq j\leq n}\mathfrak{g}_{e_1\pm e_{j}}
	\oplus \bigoplus _{2\leq i<j\leq 2p+1}\mathfrak{g}_{e_{i}+e_{j}}. 
\end{align*}
Then, $\mathfrak{g}(2)$ is isomorphic to $\mathbb{C}^{2n-4p-1}\oplus \operatorname{Alt}(2p,\mathbb{C})$. 
Further, $\mathfrak{g}(3)$ is of the form 
\begin{align*}
\mathfrak{g}(3)&=\bigoplus _{2\leq j\leq 2p+1}\mathfrak{g}_{e_1+e_{j}}
\simeq \mathbb{C}^{2p}. 
\end{align*}
Hence, $\mathfrak{n}$ is isomorphic to 
\begin{align}
\mathfrak{n}=\mathfrak{g}(2)\oplus \mathfrak{g}(3)\simeq 
\mathbb{C}^{2n-4p-1}\oplus \operatorname{Alt}(2p,\mathbb{C})\oplus \mathbb{C}^{2p}. 
\label{eq:n-b3}
\end{align}

The semisimple part $SL(2p,\mathbb{C})\times SO(2n-4p-1,\mathbb{C})$ of $L_{\mathbb{C}}$ 
acts on $\mathbb{C}^{2n-4p-1}\oplus \operatorname{Alt}(2p,\mathbb{C})\oplus \mathbb{C}^{2p}$ 
by 
\begin{align*}
(g,h)\cdot (v,A,w)=(hv,gA\,{}^t\!g,gw), 
\end{align*}
and its center $(\mathbb{C}^{\times})^2$ acts by 
\begin{align*}
(s,t)\cdot (v,A,w)=(sv,t^2A,stw). 
\end{align*}
Then, 
the $L_{\mathbb{C}}$-action on $\mathfrak{n}$ is geometrically equivalent to 
the decomposable action 
consisting of the indecomposable $(SL(2p,\mathbb{C})\times \mathbb{C}^{\times})$-action on 
$\operatorname{Alt}(2p,\mathbb{C})\oplus \mathbb{C}^{2p}$ 
((8) of Table \ref{table:linear}) 
and the irreducible $(SO(2n-4p-1,\mathbb{C})\times \mathbb{C}^{\times})$-action 
on $\mathbb{C}^{2n-4p-1}$ ((6) of Table \ref{table:linear}). 
Hence, this action is a multiplicity-free action. 

Our slice for this action is the direct sum the slices $S_0'$ and $S_0''$ 
are isomorphic to the slice $A_p\oplus T_p$ of Table \ref{table:choice} 
for the action of $SU(2p)\times \mathbb{T}$ on 
$\operatorname{Alt}(2p,\mathbb{C})\oplus \mathbb{C}^{2p}$ 
and $D_{1,1}$ for the action of $SO(2n-4p-1)\times \mathbb{T}$ on $\mathbb{C}^{2n-4p-1}$, respectively. 
In fact, we define 
\begin{align}
S_0'&:=\bigoplus _{1\leq i\leq p}\mathbb{R}E_{e_{2i}+e_{2i+1}}\oplus 
	\bigoplus _{1\leq j\leq p} \mathbb{R}E_{e_1+e_{2j}}, 
\label{eq:s0'-b3}
\end{align}
and
\begin{align*}
S_0''&:=\mathbb{R}E_{e_1+e_{2p+2}}\oplus \mathbb{R}E_{e_1-e_{2p+2}}. 
\end{align*}
By Lemma \ref{lem:linear}, the subspace 
\begin{align}
S_0:=S_0'\oplus S_0''\simeq (D_{1,1}\oplus A_p)\oplus T_p.
\label{eq:slice-b3}
\end{align}
satisfies $\mathfrak{n}=L_u\cdot S_0$. 
Therefore, Theorem \ref{thm:slice-n} has been proved for Case (B\ref{case:b3}) with $n\neq 2p-1$. 

In the special case where $\mathfrak{g}=\mathfrak{so}(4p+3,\mathbb{C})$ 
and $\Omega (\mathcal{O}_X)=(1,0,\ldots ,0,1)$, 
The Levi subalgebra $\mathfrak{l}$ is given by 
\begin{align*}
\mathfrak{l}=\mathfrak{a}\oplus 
	\bigoplus _{2\leq i<j\leq 2p+1}\mathfrak{g}_{\pm (e_{i}-e_{j})}
\simeq \mathfrak{sl}(2p,\mathbb{C})\oplus \mathbb{C}^2. 
\end{align*}
The $\operatorname{ad}(H)$-eigenspaces $\mathfrak{g}(2),\mathfrak{g}(3)$ are written as 
\begin{align*}
\mathfrak{g}(2)&=\mathfrak{g}_{e_1}
	\oplus \bigoplus _{2\leq i<j\leq 2p+1}\mathfrak{g}_{e_{i}+e_{j}}
	\simeq \mathbb{C}\oplus \operatorname{Alt}(2p,\mathbb{C}), \\
\mathfrak{g}(3)&=\bigoplus _{2\leq j\leq 2p+1}\mathfrak{g}_{e_1+e_{j}}
\simeq \mathbb{C}^{2p}. 
\end{align*}

We take $S_0'\simeq T_3^p\oplus T_p$ as in (\ref{eq:s0'-b3}) and 
$S_0''$ as 
\begin{align*}
S_0'':=\mathbb{R}E_{e_1}\simeq \mathbb{R}. 
\end{align*}
By Lemma \ref{lem:linear}, the vector space 
\begin{align*}
S_0:=S_0'\oplus S_0''\simeq (\mathbb{R}\oplus A_p)\oplus T_p
\end{align*}
satisfies $\mathfrak{n}=L_u\cdot S_0$. 
Hence, Theorem \ref{thm:slice-n} has been proved for Case (B\ref{case:b3}). 


\subsection{Type C$_n$}
\label{subsec:c}

In this subsection, 
we give a proof of Theorem \ref{thm:slice-n} for $\mathfrak{g}=\mathfrak{sp}(n,\mathbb{C})$. 
In this case, $\mathfrak{g}_{\mathbb{R}}=\mathfrak{sp}(n,\mathbb{R})$. 
Then, $\mathfrak{a}^*_{\mathbb{R}}\simeq V_{\mathbb{R}}$. 
A root system $\Delta \equiv \Delta (\mathfrak{g},\mathfrak{a})$ is 
$\Delta =\{ \pm e_i\pm e_j:1\leq i<j\leq n\} \sqcup \{ \pm 2e_i:1\leq i\leq n\}$. 
We fix a positive system $\Delta ^+$ as 
$\Delta ^+=\{ e_i\pm e_j:1\leq i<j\leq n\} 
\sqcup \{ 2e_i:1\leq i\leq n\}$. 
The simple roots $\alpha _1,\ldots ,\alpha _n$ is 
given by $\alpha _i=e_i-e_{i+1}$ $(1\leq i\leq n-1)$ and $\alpha _n=2e_n$. 
The highest root $\beta $ is written as $\beta =2e_1=2\alpha _1+2\alpha _2+
\cdots +2\alpha _{n-1}+\alpha _n$. 

Let $\mathcal{O}_X$ be a nilpotent orbit with characteristic element $H=h_1E_1+\cdots +h_nE_n\in \mathfrak{a}_+$ 
with $h_1\geq \cdots \geq h_n\geq 0$. 
Then, $m_i=h_i-h_{i+1}$ $(i=1,2,\ldots ,n-1)$, and $m_n=2h_n$. 
Hence, the weighted Dynkin diagram $\Omega (\mathcal{O}_X)=(m_1,\ldots ,m_n)$ is given by 
$\Omega (\mathcal{O}_X)=(h_1-h_2,\ldots ,h_{n-1}-h_n,2h_n)$. 

\begin{figure}[htbp]
\begin{align*}
\SelectTips{cm}{12}
	\objectmargin={1pt}
	\xygraph{!~:{@{<=}}
		\circ ([]!{+(0,-.3)} {\alpha_1},[]!{+(0,+.3)} {m_1}) - [r]
		\circ ([]!{+(0,-.3)} {\alpha_{2}},[]!{+(0,+.3)} {m_2}) - [r] 
		\cdots - [r] 
		\circ ([]!{+(0,-.3)} {\alpha_{n - 1}},[]!{+(0,+.3)} {m_{n-1}}) : [r]
		\circ ([]!{+(0,-.3)} {\alpha_n},[]!{+(0,+.3)} {m_n})
	} 
\end{align*}
\caption{Weighted Dynkin diagram of $\mathcal{O}_X$ in $\mathfrak{sp}(n,\mathbb{C})$}
\end{figure}

A nilpotent orbit $\mathcal{O}_X$ in $\mathfrak{sp}(n,\mathbb{C})$ is spherical 
if and only if $\Omega (\mathcal{O}_X)$ satisfies either Case (C) or Case (C$'$): 

\begin{enumerate}
	\item[(C)] $\Omega (\mathcal{O}_X)=(\underbrace{0,\ldots ,0}_{p-1},1,0,\ldots ,0)$ 
	for $1\leq p<n$, namely, 
	$m_p=1$ and $m_i=0$ $(i\neq p)$. 
	
	\item[(C$'$)] $\Omega (\mathcal{O}_X)=(0,\ldots ,0,2)$, namely, 
	$m_n=2$ and $m_i=0$ $(i\neq n)$. 
\end{enumerate}
Then, the characteristic element $H\in \mathfrak{a}$ is written as 
\begin{align}
H=E_1+\cdots +E_p 
\label{eq:H-typec}
\end{align}
for each $\Omega (\mathcal{O}_X)$. 
It follows from Lemma \ref{lem:height-max} that 
the height of $\mathcal{O}_X$ equals two. 
Further, 
Such $\mathcal{O}_X$ consists of complex matrices 
with Jordan type $(2^p,1^{2n-2p})$ $(1\leq p\leq n)$. 
In particular, 
$\mathcal{O}_X$ with weighted Dynkin diagram $\Omega (\mathcal{O}_X)=(1,0,\ldots ,0)$ 
is the minimal nilpotent orbit. 

First, let us consider the general $p\neq n$. 
The Levi subalgebra $\mathfrak{l}=\mathfrak{g}(0)$ is given by 
\begin{align*}
\mathfrak{l}=\mathfrak{a}\oplus \bigoplus _{1\leq i<j\leq p}\mathfrak{g}_{\pm (e_i-e_j)}
	\oplus \bigoplus _{p+1\leq i<j\leq n}\mathfrak{g}_{\pm e_{i}\pm e_{j}}
	\oplus \bigoplus _{p+1\leq i\leq n}\mathfrak{g}_{\pm 2e_{i}}. 
\end{align*}
This means that 
\begin{align*}
\mathfrak{l}\simeq \mathfrak{sl}(p,\mathbb{C})\oplus \mathfrak{sp}(n-p,\mathbb{C})
	\oplus \mathbb{C}. 
\end{align*}

The $\operatorname{ad}(H)$-eigenspace $\mathfrak{g}(2)$ is written as 
\begin{align*}
\mathfrak{g}(2)&=\bigoplus _{1\leq i<j\leq p}\mathfrak{g}_{e_i+e_j}
	\oplus \bigoplus _{1\leq i\leq p}\mathfrak{g}_{2e_i}. 
\end{align*}
Then, $\mathfrak{g}(2)$ is isomorphic to $\operatorname{Sym}(p,\mathbb{C})$, from which 
\begin{align*}
\mathfrak{n}\simeq \operatorname{Sym}(p,\mathbb{C}). 
\end{align*}

The action of 
the semisimple part $SL(p,\mathbb{C})\times Sp(n-p,\mathbb{C})$ on 
$\operatorname{Sym}(p,\mathbb{C})$ is written as follows: 
$SL(p,\mathbb{C})$ acts by 
\begin{align*}
g\cdot A=gA\,{}^tg, 
\end{align*}
and $Sp(n-p,\mathbb{C})$ acts trivially. 
Its center $\mathbb{C}^{\times}$ acts as the scalar multiplication. 
Then, the $L_{\mathbb{C}}$-action  on $\mathfrak{n}$ is geometrically equivalent 
to the $(SL(p,\mathbb{C})\times \mathbb{C}^{\times})$-action 
on $\operatorname{Sym}(p,\mathbb{C})$. 
It follows from (3) of Table \ref{table:linear} that 
this action is a multiplicity-free action. 

We take the subset $S_0$ as 
\begin{align}
S_0=\bigoplus _{1\leq i\leq p} \mathbb{R}E_{2e_i}. 
\label{eq:slice-c}
\end{align}
Then, $S_0$ is isomorphic to the slice $D_p$ of Table \ref{table:choice} 
for the strongly visible $(SU(p)\times \mathbb{T})$-action $\operatorname{Sym}(p,\mathbb{C})$. 
By Lemma \ref{lem:linear}, 
this $S_0$ satisfies $\mathfrak{n}=L_u\cdot S_0$. 

In case of $p=n$, 
the Levi subalgebra $\mathfrak{l}$ is 
\begin{align*}
\mathfrak{l}=\mathfrak{a}\oplus \bigoplus _{1\leq i<j\leq n}\mathfrak{g}_{\pm (e_i-e_j)}
\simeq \mathfrak{sl}(n,\mathbb{C})\oplus \mathbb{C}, 
\end{align*}
and $\mathfrak{g}(2)$ is 
\begin{align*}
\mathfrak{g}(2)&=\bigoplus _{1\leq i<j\leq n}\mathfrak{g}_{e_i+e_j}
	\oplus \bigoplus _{1\leq i\leq n}\mathfrak{g}_{2e_i}\simeq 
	\operatorname{Sym}(n,\mathbb{C}). 
\end{align*}
Then, the $L_{\mathbb{C}}$-action on $\mathfrak{n}$ is 
geometrically equivalent to the $(SL(n,\mathbb{C})\times \mathbb{C}^{\times})$-action 
on $\operatorname{Sym}(n,\mathbb{C})$. 
Hence, the equation $\mathfrak{n}=L_u\cdot S_0$ holds for 
the subset $S_0$ defined by (\ref{eq:slice-c}) 

Therefore, Theorem \ref{thm:slice-n} has been verified 
for $\mathfrak{g}=\mathfrak{sp}(n,\mathbb{C})$


\subsection{Type D$_n$}
\label{subsec:d}

In this subsection, we give a proof of Theorem \ref{thm:slice-n} for 
$\mathfrak{g}=\mathfrak{so}(2n,\mathbb{C})$ for integer $n\geq 4$. 
In this case, 
$\mathfrak{g}_{\mathbb{R}}$ is isomorphic to $\mathfrak{so}(n,n)$. 
Then, $\mathfrak{a}^*_{\mathbb{R}}=V_{\mathbb{R}}$. 
A root system $\Delta \equiv \Delta (\mathfrak{g},\mathfrak{a})$ is 
$\Delta =\{ \pm e_i\pm e_j:1\leq i<j\leq n\} $. 
We fix a positive system $\Delta ^+$ as 
$\Delta ^+=\{ e_i\pm e_j:1\leq i<j\leq n\} $. 
The simple roots $\alpha _1,\ldots ,\alpha _n$ is given by 
$\alpha _i=e_i-e_{i+1}$ $(1\leq i\leq n-1)$ and $\alpha _n=e_{n-1}+e_n$. 
The highest root $\beta $ is written as $\beta =e_1+e_2
=\alpha _1+2\alpha _2+2\alpha _3+\cdots +2\alpha _{n-2}+\alpha _{n-1}+\alpha _n$. 

Let $\mathcal{O}_X$ be a nilpotent orbit in $\mathfrak{so}(2n,\mathbb{C})$ 
with characteristic element $H=h_1E_1+\cdots +h_nE_n\in \mathfrak{a}_+$ 
with $h_1\geq \cdots \geq h_{n-1}\geq |h_n|$. 
Then, the weighted Dynkin diagram $\Omega (\mathcal{O}_X)=(m_1,\ldots ,m_n)$ 
is given by $(h_1-h_2,\ldots ,h_{n-1}-h_n,h_{n-1}+h_n)$. 

\begin{figure}[htbp]
\begin{align*}
\SelectTips{cm}{12}
	\objectmargin={1pt}
	\xygraph{
		\circ ([]!{+(0,-.3)} {\alpha_1},[]!{+(0,+.3)} {m_1}) - [r]
		\circ ([]!{+(0,-.3)} {\alpha_2},[]!{+(0,+.3)} {m_2}) - [r] 
		\cdots - [r]
		\circ ([]!{+(0,-.3)} {\alpha_{n - 2}},[]!{+(0,+.3)} {m_{n-2}}) 
			(
			- []!{+(1,0.7)} \circ ([]!{+(0.1,-.3)} {\alpha_{n-1}},[]!{+(0,+.3)} {m_{n-1}}),
			- []!{+(1,-0.7)} \circ ([]!{+(0.1,-.3)} {\alpha_{n}},[]!{+(0,+.3)} {m_n})
			)
	}
\end{align*}
\caption{Weighted Dynkin diagram of $\mathcal{O}_X$ in $\mathfrak{so}(2n,\mathbb{C})$}
\end{figure}

A nilpotent orbit $\mathcal{O}_X$ in $\mathfrak{so}(2n,\mathbb{C})$ is spherical 
if and only if $\Omega (\mathcal{O}_X)$ satisfies one of the following cases: 
\begin{enumerate}
	\renewcommand{\labelenumi}{(D\theenumi) }
	\item $\Omega (\mathcal{O}_X)=(2,0,\ldots ,0)$, 
	namely, $m_1=1$ and $m_i=0~(i\neq 1)$. 
	\label{case:d1}
	
	\item $\Omega (\mathcal{O}_X)=(\underbrace{0,\ldots ,0}_{2p-1},1,0,\ldots ,0)$ 
	for $1\leq p\leq \frac{n}{2}-1$, namely, 
	$m_{2p}=1$ and $m_i=0~(i\neq 2p)$. 
	\label{case:d2}
	
	\item[(D\ref{case:d2}$'$)]%
	$n=2p+1$ and $\Omega (\mathcal{O}_X)=(0,\ldots ,0,1,1)$, 
	namely, $m_{2p}=m_{2p+1}=1$ and $m_i=0~(i\neq 2p,2p+1)$. 
	
	\item[(D\ref{case:d2}$''$)]%
	$n=2p$ and $\Omega (\mathcal{O}_X)=(0,\ldots ,0,2)$, 
	namely, $m_{2p}=2$ and $m_i=0~(i\neq 2p)$. 
	
	\item[(D\ref{case:d2}$'''$)]%
	$n=2p$ and $\Omega (\mathcal{O}_X)=(0,\ldots ,0,2,0)$, 
	namely, $m_{2p-1}=2$ and $m_i=0~(i\neq 2p-1)$. 
	
	\item $\Omega (\mathcal{O}_X)=(1,\underbrace{0,\ldots ,0}_{2p-1},1,0,\ldots ,0)$ 
	for $1\leq p<\frac{n}{2}-1$, 
	namely, $m_1=m_{2p+1}=1$ and $m_i=0~(i\neq 1,2p+1)$. 
	\label{case:d3}
	
	\item[(D\ref{case:d3}$'$)] $n=2p+2$ and 
	$\Omega (\mathcal{O}_X)=(1,0,0,\ldots ,0,1,1)$, 
	namely, $m_1=m_{2p+1}=m_{2p+2}=1$ and $m_i=0~(i\neq 1,2p+1,2p+2)$. 
\end{enumerate}
By Lemma \ref{lem:height-max}, 
the height of $\mathcal{O}_X$ equals two if for Cases (D\ref{case:d1})--(D\ref{case:d2}$''$), 
and three for Cases (D\ref{case:d3}), (D\ref{case:d3}$'$). 

\subsubsection{Case (D\ref{case:d1})}

Let us consider the case $\Omega (\mathcal{O}_X)=(2,0,\ldots ,0)$. 
Then, the characteristic element $H$ is of the form 
\begin{align*}
H=2E_1. 
\end{align*}
This $\mathcal{O}_X$ consists of all complex matrices with Jordan type $(3,1^{2n-3})$. 

The Levi subalgebra $\mathfrak{l}$ is given by 
\begin{align*}
\mathfrak{l}=\mathfrak{a}\oplus \bigoplus _{2\leq i<j\leq n}\mathfrak{g}_{\pm e_{i}\pm e_{j}}. 
\end{align*}
This means that 
\begin{align*}
\mathfrak{l}\simeq \mathfrak{so}(2n-2,\mathbb{C})\oplus \mathbb{C}. 
\end{align*}

The $\operatorname{ad}(H)$-eigenspace $\mathfrak{g}(2)$ is written as 
\begin{align*}
\mathfrak{g}(2)=\bigoplus _{2\leq j\leq n}\mathfrak{g}_{e_1\pm e_{j}}. 
\end{align*}
Then, $\mathfrak{g}(2)$ is isomorphic to $\mathbb{C}^{2n-2}$, from which 
\begin{align*}
\mathfrak{n}\simeq \mathbb{C}^{2n-2}. 
\end{align*}

Similarly to Case (B\ref{case:b1}), 
it turns out that 
the $L_{\mathbb{C}}$-action on $\mathfrak{n}$ is geometrically equivalent to 
the $(SO(2n-2,\mathbb{C})\times \mathbb{C}^{\times })$-action on $\mathbb{C}^{2n-2}$. 
It follows from (6) of Table \ref{table:linear} that 
this action is a multiplicity-free action. 

We take the subset $S_0$ in $\mathfrak{n}$ as 
\begin{align*}
S_0=\mathbb{R}E_{e_1+e_2}\oplus \mathbb{R}E_{e_1-e_2}. 
\end{align*}
Then, $S_0$ is isomorphic to the slice $D_{1,1}$ of Table \ref{table:choice} 
for the $(SO(2n-2)\times \mathbb{T})$-action on $\mathbb{C}^{2n-2}$. 
By Lemma \ref{lem:linear}, $S_0$ satisfies $\mathfrak{n}=L_u\cdot S_0$. 
Therefore, Theorem \ref{thm:slice-n} has been verified for Case (D\ref{case:d1}). 

\subsubsection{Case (D\ref{case:d2})}
\label{subsubsec:d}

Let $\Omega (\mathcal{O}_X)$ satisfy $m_{2p}=1$ and $m_i=0$ $(i\neq 2p)$ for 
$1\leq p\leq \frac{n}{2}-1$. 
Then, 
\begin{align}
H=E_1+E_2+\cdots +E_{2p}. 
\label{eq:H-d2}
\end{align}
This $\mathcal{O}_X$ consists of all complex matrices with Jordan type $(2^{2p},1^{2n-4p})$. 
In particular, $\mathcal{O}_X$ with Jordan type $(2^2,1^{2n-4})$ $(p=1)$ 
is the minimal nilpotent orbit in $\mathfrak{so}(2n,\mathbb{C})$. 

The Levi subalgebra $\mathfrak{l}$ is given by 
\begin{align*}
\mathfrak{l}=\mathfrak{a}\oplus \bigoplus _{1\leq i<j\leq 2p}\mathfrak{g}_{\pm (e_i-e_j)}
	\oplus \bigoplus _{2p+1\leq i<j\leq n}\mathfrak{g}_{\pm e_{i}\pm e_{j}}. 
\end{align*}
This means that 
\begin{align*}
\mathfrak{l}\simeq \mathfrak{sl}(2p,\mathbb{C})\oplus \mathfrak{so}(2n-4p-2,\mathbb{C})
	\oplus \mathbb{C}. 
\end{align*}

The $\operatorname{ad}(H)$-eigenspace $\mathfrak{g}(2)$ is written as 
\begin{align*}
\mathfrak{g}(2)=\bigoplus _{1\leq i<j\leq 2p}\mathfrak{g}_{e_i+e_j}. 
\end{align*}
Then, $\mathfrak{g}(2)$ is isomorphic to $\operatorname{Alt}(2p,\mathbb{C})$, from which 
\begin{align*}
\mathfrak{n}&\simeq \operatorname{Alt}(2p,\mathbb{C}). 
\end{align*}

The semisimple part $SL(2p,\mathbb{C})\times SO(2n-4p-2,\mathbb{C})$ 
of the Levi subgroup $L_{\mathbb{C}}$ acts on $\operatorname{Alt}(2p,\mathbb{C})$ 
as follows: $SL(2p,\mathbb{C})$ by 
\begin{align*}
g\cdot A=gA\,{}^tg, 
\end{align*}
and $SO(2n-4p-2,\mathbb{C})$ trivially. 
Its center $\mathbb{C}^{\times}$ acts as scalar multiplications. 
Then, the $L_{\mathbb{C}}$-action on $\mathfrak{n}$ is geometrically equivalent to 
the $(SL(2p,\mathbb{C})\times \mathbb{C}^{\times })$-action 
on $\operatorname{Alt}(2p,\mathbb{C})$. 
It follows from (4) of Table \ref{table:linear} that 
this action is a multiplicity-free action. 

We take the subset $S_0$ in $\mathfrak{n}$ as 
\begin{align}
S_0:=\bigoplus _{j=1}^p \mathbb{R}E_{e_{2j-1}+e_{2j}}. 
\label{eq:slice-d2}
\end{align}
Then, $S_0$ is isomorphic to our slice $A_p$ of Table \ref{table:choice} 
for the $(SU(2p)\times \mathbb{T})$-action on $\operatorname{Alt}(2p,\mathbb{C})$. 
By Lemma \ref{lem:linear}, we have $\mathfrak{n}=L_u\cdot S_0$. 
Therefore, Theorem \ref{thm:slice-n} has been verified for Case (D\ref{case:d2}). 

\subsubsection{Case (D\ref{case:d2}\,$'$)}

Let us consider the case where $\mathfrak{g}=\mathfrak{so}(4p+2,\mathbb{C})$ 
$(n=2p+1)$ and 
$\Omega (\mathcal{O}_X)=(0,\ldots ,0,1,1)$. 
This nilpotent orbit is the set of all complex matrices 
with Jordan type $(2^{2p},1^2)$. 
Then, the proof for Case (D\ref{case:d2}\,$'$) can be given similarly to Case (D\ref{case:d2}). 
In fact, the characteristic element $H$ forms 
\begin{align*}
H=E_1+\cdots +E_{2p} 
\end{align*}
which is the same as in (\ref{eq:H-d2}). 

The Levi subalgebra $\mathfrak{l}$ is given by 
\begin{align*}
\mathfrak{l}=\mathfrak{a}\oplus \bigoplus _{1\leq i<j\leq 2p}\mathfrak{g}_{\pm (e_i-e_j)}
	\simeq \mathfrak{sl}(2p,\mathbb{C})\oplus \mathbb{C}^2. 
\end{align*}
The $\operatorname{ad}(H)$-eigenspace $\mathfrak{g}(2)$ is written as 
\begin{align*}
\mathfrak{g}(2)=\bigoplus _{1\leq i<j\leq 2p}\mathfrak{g}_{e_i+e_j}
	\simeq \operatorname{Alt}(2p,\mathbb{C}). 
\end{align*}
Then, the $L_{\mathbb{C}}$-action on $\mathfrak{n}$ is geometrically equivalent to 
the $(SL(2p,\mathbb{C})\times \mathbb{C}^{\times })$-action 
on $\operatorname{Alt}(2p,\mathbb{C})$. 
Similarly to the previous case, 
$\mathfrak{n}=L_u\cdot S_0$ holds for 
the subset $S_0$ defined by (\ref{eq:slice-d2}). 
Hence, Theorem \ref{thm:slice-n} for Case (D\ref{case:d2}\,$'$) has been verified. 

\subsubsection{Case (D\ref{case:d2}\,$''$)}

Let us treat the case where $\mathfrak{g}=\mathfrak{so}(4p,\mathbb{C})$ and 
$\Omega (\mathcal{O}_X)=(0,\ldots ,0,2)$. 
This nilpotent orbit $\mathcal{O}_X$ is very even, namely, 
$\mathcal{O}_X$ is the set of all complex matrices 
with Jordan type $(2^{2p})$. 
The characteristic element $H$ forms the same as in (\ref{eq:H-d2}). 
Then, 
\begin{align*}
\mathfrak{l}&=\mathfrak{a}\oplus \bigoplus _{1\leq i<j\leq 2p}\mathfrak{g}_{\pm (e_i-e_j)}
	\simeq \mathfrak{sl}(2p,\mathbb{C})\oplus \mathbb{C}, \\
\mathfrak{n}&=\mathfrak{g}(2)=\bigoplus _{1\leq i<j\leq 2p}\mathfrak{g}_{e_i+e_j}
	\simeq \operatorname{Alt}(2p,\mathbb{C}). 
\end{align*}

Similarly to Cases (D\ref{case:d2}) and (D\ref{case:d2}$'$), 
the equality $\mathfrak{n}=L_u\cdot S_0$ holds for $S_0$ defined by (\ref{eq:slice-d2}). 
Hence, Theorem \ref{thm:slice-n} for Case (D\ref{case:d2}\,$''$) has been checked. 

\subsubsection{Case (D\ref{case:d2}\,$'''$)}

There are two weighted Dynkin diagrams corresponding to nilpotent orbits $\mathcal{O}_X$ 
in $\mathfrak{so}(4p,\mathbb{C})$ with Jordan type $(2^{2p})$. 
One is Case (D\ref{case:d2}\,$''$), 
the other is $\Omega (\mathcal{O}_X)=(0,\ldots ,0,2,0)$, namely, Case (D\ref{case:d2}\,$'''$). 
For the latter case, $H$ is of the form 
\begin{align*}
H=E_1+\cdots +E_{2p-1}-E_{2p} 
\end{align*}
which is slightly different from the form in (\ref{eq:H-d2}). 

The Levi subalgebra $\mathfrak{l}$ is given by 
\begin{align*}
\mathfrak{l}=\mathfrak{a}\oplus \bigoplus _{1\leq i<j\leq 2p-1}\mathfrak{g}_{\pm (e_i-e_j)}
	\oplus \bigoplus _{1\leq i\leq 2p-1}\mathfrak{g}_{\pm (e_i+e_{2p})}. 
\end{align*}
This means that 
\begin{align*}
\mathfrak{l}\simeq \mathfrak{sl}(2p,\mathbb{C})\oplus \mathbb{C}. 
\end{align*}

The $\operatorname{ad}(H)$-eigenspace $\mathfrak{g}(2)$ is written as 
\begin{align*}
\mathfrak{g}(2)=\bigoplus _{1\leq i<j\leq 2p-1}\mathfrak{g}_{e_i+e_j}
	\oplus \bigoplus_{1\leq i\leq 2p-1}\mathfrak{g}_{e_i-e_{2p}}. 
\end{align*}
Then, 
\begin{align*}
\mathfrak{n}\simeq \operatorname{Alt}(2p,\mathbb{C}). 
\end{align*}
Hence, the $L_{\mathbb{C}}$-action on $\mathfrak{n}$ is geometrically equivalent to 
the $(SL(2p,\mathbb{C})\times \mathbb{C}^{\times})$-action on $\operatorname{Alt}(2p,\mathbb{C})$, 
which is a multiplicity-free action. 

We take $S_0$ as 
\begin{align*}
S_0=\bigoplus _{1\leq i\leq p-1}\mathbb{R}E_{e_{2j-1}+e_{2j}}
	\oplus \mathbb{R}E_{e_{2p-1}-e_{2p}}. 
\end{align*}
Then, $S_0$ is isomorphic to our slice $A_{p}$ of Table \ref{table:choice} 
for the $(SU(2p)\times \mathbb{T})$-action 
via $\mathfrak{n}\simeq \operatorname{Alt}(2p,\mathbb{C})$. 
By Lemma \ref{lem:linear}, 
$\mathfrak{n}=L_u\cdot S_0$. 
Hence, Theorem \ref{thm:slice-n} has been verified. 

\subsubsection{Case (D\ref{case:d3})}

Let $\Omega (\mathcal{O}_X)$ satisfy 
$m_1=m_{2p+1}=1$ and $m_i=0$ $(i\neq 1,2p+1)$ for $1\leq p\leq \frac{n}{2}-1$. 
Then, $H$ is of the form 
\begin{align}
H=2E_1+E_2+\cdots +E_{2p+1}. 
\label{eq:H-d3}
\end{align}
This $\mathcal{O}_X$ is the set of all complex matrices 
with Jordan type $(3,2^{2p},1^{2n-4p-3})$. 

The Levi subalgebra $\mathfrak{l}$ is given by 
\begin{align*}
\mathfrak{l}&=\mathfrak{a}\oplus \bigoplus _{2\leq i<j\leq 2p+1}
	\mathfrak{g}_{\pm (e_{i}-e_{j})}\oplus 
	\bigoplus _{2p+2\leq i<j\leq n}\mathfrak{g}_{\pm e_{i}\pm e_{j}}, 
\end{align*}
from which $\mathfrak{l}$ is isomorphic to 
\begin{align*}
\mathfrak{l}\simeq \mathfrak{sl}(2p,\mathbb{C})\oplus \mathfrak{so}(2n-4p-2,\mathbb{C})
	\oplus \mathbb{C}^2. 
\end{align*}

Next, the $\operatorname{ad}(H)$-eigenspace $\mathfrak{g}(2)$ is written as 
\begin{align*}
\mathfrak{g}(2)&=\bigoplus _{2\leq i<j\leq 2p+1}\mathfrak{g}_{e_{i}+e_{j}}
	\oplus \bigoplus _{2p+2\leq j\leq n}\mathfrak{g}_{e_1\pm e_{j}}. 
\end{align*}
Then, $\mathfrak{g}(2)$ is isomorphic to 
the direct sum $\operatorname{Alt}(2p,\mathbb{C})\oplus 
\mathbb{C}^{2n-4p-2}$. 
The eigenspace $\mathfrak{g}(3)$ forms 
\begin{align*}
\mathfrak{g}(3)&=\bigoplus _{2\leq j\leq 2p+1}\mathfrak{g}_{e_1+e_{j}}
	\simeq \mathbb{C}^{2p}. 
\end{align*}
This implies that 
\begin{align*}
\mathfrak{n}
\simeq (\operatorname{Alt}(2p,\mathbb{C})\oplus \mathbb{C}^{2n-4p-2})\oplus \mathbb{C}^{2p}. 
\end{align*}

The semisimple part $SL(2p,\mathbb{C})\times SO(2n-4p-2,\mathbb{C})$ 
of the Levi subgroup $L_{\mathbb{C}}$ acts on $\operatorname{Alt}(2p,\mathbb{C})\oplus 
\mathbb{C}^{2n-4p-2}\oplus \mathbb{C}^{2p}$ by 
\begin{align*}
(g,h)\cdot (A,v,w)=(gA\,{}^tg,hv,gw), 
\end{align*}
and its center $(\mathbb{C}^{\times })^2$ acts by 
\begin{align*}
(s,t)\cdot (A,v,w)=(t^2A,st^2v,st^3w). 
\end{align*}
Then, the $L_{\mathbb{C}}$-action on $\mathfrak{n}$ 
is geometrically equivalent to 
the decomposable action consisting 
of the indecomposable $(SL(2p,\mathbb{C})\times \mathbb{C}^{\times})$-action 
on $\operatorname{Alt}(2p,\mathbb{C})\oplus \mathbb{C}^{2p}$ 
((8) of Table \ref{table:linear}) 
and the irreducible $(SO(2n-4p-2,\mathbb{C})\times \mathbb{C}^{\times})$-action 
on $\mathbb{C}^{2n-4p-2}$ ((6) of Table \ref{table:linear}). 
Hence, this action is a multiplicity-free action. 

Here, we define 
\begin{align*}
S_0'&:=\bigoplus _{1\leq i\leq p}\mathbb{R}E_{e_{2i}+e_{2i+1}}
	\oplus \bigoplus _{1\leq j\leq p}\mathbb{R}E_{e_1+e_{2j}}, \\
S_0''&:=\mathbb{R}E_{e_1+e_{2p+2}}\oplus \mathbb{R}E_{e_1-e_{2p+2}}
\end{align*}
Then, $S_0'$ is isomorphic to the slice $A_p\oplus T_p$ of Table \ref{table:choice} 
for the $(SU(2p)\times \mathbb{T})$-action 
on $\operatorname{Alt}(2p,\mathbb{C})\oplus \mathbb{C}^{2p}$, 
and $S_0''$ to the slice $D_{1,1}$ 
for the $(SO(2n-4p-2)\times \mathbb{T})$-action on $\mathbb{C}^{2n-4p-2}$. 
We set 
\begin{align}
S_0:=S_0'\oplus S_0''\simeq 
(A_p\oplus D_{1,1})\oplus T_p. 
\label{eq:slice-d3}
\end{align}
Then, it follows from Lemma \ref{lem:linear} that $\mathfrak{n}=L_u\cdot S_0$. 
Therefore, Theorem \ref{thm:slice-n} has been verified for Case (D\ref{case:b3}). 

\subsubsection{Case (D\ref{case:d3}\,$'$)}

Let us consider the case where $\mathfrak{g}=\mathfrak{so}(4p+4,\mathbb{C})$ and 
$\Omega (\mathcal{O}_X)=(1,0,\ldots ,0,1,1)$. 
Then, the characteristic element $H$ forms the same as in (\ref{eq:H-d3}). 
This $\mathcal{O}_X$ is the set of all nilpotent elements with Jordan type $(3,2^{2p},1)$. 

The Levi subalgebra $\mathfrak{l}$ is given by 
\begin{align*}
\mathfrak{l}=\mathfrak{a}\oplus \bigoplus _{2\leq i<j\leq 2p+1}\mathfrak{g}_{\pm (e_{i}-e_{j})}
	\simeq \mathfrak{sl}(2p,\mathbb{C})\oplus \mathbb{C}^3. 
\end{align*}

The $\operatorname{ad}(H)$-eigenspace $\mathfrak{g}(2)$ is written as 
\begin{align*}
\mathfrak{g}(2)=\bigoplus _{2\leq i<j\leq 2p+1}\mathfrak{g}_{e_{i}+e_{j}}
	\oplus \mathfrak{g}_{e_1\pm e_{2p+2}}
	\simeq \operatorname{Alt}(2p,\mathbb{C})\oplus \mathbb{C}^2. 
\end{align*}
Further, 
$\mathfrak{g}(3)$ is of the form 
\begin{align*}
\mathfrak{g}(3)=\bigoplus _{2\leq j\leq 2p+1}\mathfrak{g}_{e_1+e_{j}}
\simeq \mathbb{C}^{2p}. 
\end{align*}
This implies that the nilpotent subalgebra $\mathfrak{n}$ is isomorphic to 
\begin{align*}
\mathfrak{n}\simeq (\operatorname{Alt}(2p,\mathbb{C})\oplus \mathbb{C}^2)
	\oplus \mathbb{C}^{2p}. 
\end{align*}

The subset $S_0$ in $\mathfrak{n}$ defined by (\ref{eq:slice-d3}) 
satisfies $\mathfrak{n}=L_u\cdot S_0$. 
Therefore, Theorem \ref{thm:slice-n} has been verified. 


\subsection{Type E$_6$}
\label{subsec:e6}

In Sections \ref{subsec:e6}--\ref{subsec:g2}, 
we deal with $\mathfrak{g}$ of exceptional type. 
In this subsection, 
we give a proof of Theorem \ref{thm:slice-n} for 
$\mathfrak{g}=\mathfrak{e}_6(\mathbb{C})$. 
In this case, $\mathfrak{g}_{\mathbb{R}}\simeq \mathfrak{e}_{6(6)}$. 
Then, $\mathfrak{a}^*_{\mathbb{R}}=\{ v\in \mathbb{R}e_1\oplus \cdots +\mathbb{R}e_8
:\langle v,e_8+e_7\rangle =\langle v,e_7-e_6\rangle =0\} 
=\mathbb{R}(e_8-e_7-e_6)\oplus \mathbb{R}e_5\oplus \cdots \oplus \mathbb{R}e_1
$. 
A root system $\Delta \equiv \Delta (\mathfrak{g},\mathfrak{a})$ is given by 
$\Delta =\{ \pm (e_i-e_j):1\leq j<i\leq 5\} 
\sqcup \{ \pm \frac{1}{2}(e_8-e_7-e_6+\sum _{j=1}^5 (-1)^{n(j)}e_j): 
\sum _{j=1}^5 n(j)=0,2,4\} $, 
where $n(1),\ldots n(5)\in \{ 0,1\} $. 
We fix a positive system $\Delta ^+$ of $\mathfrak{g}$ as follows: 
$\Delta ^+=\{ e_i-e_j:1\leq j<i\leq 5\} 
\sqcup \{ \frac{1}{2}(e_8-e_7-e_6+\sum _{j=1}^5 (-1)^{n(j)}e_j): 
	\sum _{j=1}^5 n(j)=0,2,4\} $. 
The simple roots $\alpha _1,\ldots ,\alpha _6$ are given by 
$\alpha _1=\frac{1}{2}(e_8-e_7-e_6-e_5-e_4-e_3-e_2+e_1)$, 
$\alpha _2=e_2+e_1$, and $\alpha _i=e_{i-1}-e_{i-2}$ $(3\leq i\leq 6)$. 
The highest root $\beta $ is written as 
$\beta =\frac{1}{2}(e_8-e_7-e_6+e_5+e_4+e_3+e_2+e_1)
=\alpha _1+2\alpha _2+2\alpha _3+3\alpha _4+2\alpha _5+\alpha _6$. 

Let $\mathcal{O}_X$ be a nilpotent orbit in $\mathfrak{e}_6(\mathbb{C})$ 
with characteristic element $H=h_6(E_8-E_7-E_6)+h_5E_5+\cdots +h_1E_1\in \mathfrak{a}_+$. 
Then, 
$\alpha _1(H)=\frac{1}{2}(3h_6-h_5-h_4-h_3-h_2+h_1),~
\alpha _2(H)=h_2+h_1$, and 
$\alpha _i(H)=h_{i-1}-h_{i-2}\ (i=3,4,5,6)$. 
Hence, the weighted Dynkin diagram $\Omega (\mathcal{O}_X)
=(m_1,m_2,m_3,m_4,m_5,m_6)$ is given by 
$(\frac{1}{2}(3h_6-h_5-h_4-h_3-h_2+h_1),h_2+h_1,
h_2-h_1,h_3-h_2,h_4-h_3,h_5-h_4)$. 

\begin{figure}[htbp]
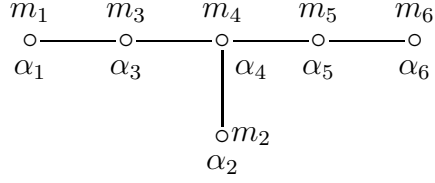

\begin{align*}
\SelectTips{cm}{12}
	\objectmargin={1pt}
	\xygraph{
		\circ ([]!{+(0,-.3)} {\alpha_1},[]!{+(0,+.3)} {m_1}) - [r]
		\circ ([]!{+(0,-.3)} {\alpha_3},[]!{+(0,+.3)} {m_3}) - [r]
		\circ ([]!{+(.3,-.3)} {\alpha_4},[]!{+(0,+.3)} {m_4}) (
			- [d] \circ ([]!{+(0,-0.3)} {\alpha_2},[]!{+(0.3,0)} {m_2}),
			- [r] \circ ([]!{+(0,-.3)} {\alpha_5},[]!{+(0,+.3)} {m_5})
			- [r] \circ ([]!{+(0,-.3)} {\alpha_6},[]!{+(0,+.3)} {m_6}))
	}
\end{align*}
\caption{Weighted Dynkin diagram of $\mathcal{O}_X$ in $\mathfrak{e}_6(\mathbb{C})$}
\end{figure}

A nilpotent orbit $\mathcal{O}_X$ is spherical if and only if $\Omega (\mathcal{O}_X)$ 
satisfies one of the following cases: 
\begin{enumerate}
	\renewcommand{\labelenumi}{(E$_6$\theenumi) }
	\item $\Omega (\mathcal{O}_X)=(0,1,0,0,0,0)$. 
	\label{case:e61}
	\item $\Omega (\mathcal{O}_X)=(1,0,0,0,0,1)$. 
	\label{case:e62}
	\item $\Omega (\mathcal{O}_X)=(0,0,0,1,0,0)$. 
	\label{case:e63}
\end{enumerate}

By Lemma \ref{lem:height-max}, 
the height of $\mathcal{O}_X$ equals two for Cases (E$_6$\ref{case:e61}), (E$_6$\ref{case:e62}), 
and three for Case (E$_6$\ref{case:e63}). 

\subsubsection{Case (E$_6$\ref{case:e61})}
\label{subsubsec:e6}

We consider the case $\Omega (\mathcal{O}_X)=(0,1,0,0,0,0)$. 
Then, $H\in \mathfrak{a}$ is written by 
\begin{align*}
H=\frac{1}{2}(E_8-E_7-E_6+E_5+E_4+E_3+E_2+E_1). 
\end{align*}
This $\mathcal{O}_X$ is the minimal nilpotent orbit with dimension 22. 

The Levi subalgebra $\mathfrak{l}$ is given by 
\begin{align*}
\mathfrak{l}&=\mathfrak{a}\oplus 
	\bigoplus _{1\leq j<i\leq 5}\mathfrak{g}_{\pm (e_i-e_j)}
	\oplus \bigoplus _{\sum _{j=1}^5 n(j)=4}
		\mathfrak{g}_{\pm \frac{1}{2}(e_8-e_7-e_6+\sum _{j=1}^5(-1)^{n(j)}e_j)}. 
\end{align*}
This means that 
\begin{align*}
\mathfrak{l}\simeq \mathfrak{sl}(6,\mathbb{C})\oplus \mathbb{C}. 
\end{align*}

The $\operatorname{ad}(H)$-eigenspace $\mathfrak{g}(2)$ is written as 
\begin{align*}
\mathfrak{g}(2)=\mathfrak{g}_{\frac{1}{2}(e_8-e_7-e_6+e_5+e_4+e_3+e_2+e_1)}
\simeq \mathbb{C}. 
\end{align*}
Hence, the semisimple part $SL(6,\mathbb{C})$ of 
the Levi subgroup $L_{\mathbb{C}}$ acts trivially 
on $\mathbb{C}$, and the center $\mathbb{C}^{\times }$ by 
\begin{align*}
s\cdot z=s^2z. 
\end{align*}
Then, the $L_{\mathbb{C}}$-action on $\mathfrak{n}$ is geometrically equivalent 
to the $\mathbb{C}^{\times}$-action on $\mathbb{C}$. 
It follows from (1) of Table \ref{table:linear} 
that this action is a multiplicity-free action. 

We take the subspace $S_0$ as 
\begin{align*}
S_0:=\mathbb{R}E_{\frac{1}{2}(e_8-e_7-e_6+e_5+e_4+e_3+e_2+e_1)}. 
\end{align*}
Then, $S_0$ is isomorphic to the slice $\mathbb{R}$ of Table \ref{table:linear} 
for the $\mathbb{T}$-action on $\mathbb{C}$. 
By Lemma \ref{lem:linear}, $\mathfrak{n}=L_u\cdot S_0$. 
Hence, Theorem \ref{thm:slice-n} has been verified for Case (E$_6$\ref{case:e61}). 

\subsubsection{Case (E$_6$\ref{case:e62})}

Let $\Omega (\mathcal{O}_X)$ be $(1,0,0,0,0,1)$. 
Then, $H\in \mathfrak{a}$ is 
\begin{align*}
H=E_8-E_7-E_6+E_5. 
\end{align*}

The Levi subalgebra $\mathfrak{l}$ is given by 
\begin{align*}
\mathfrak{l}&=\mathfrak{a}\oplus \bigoplus _{1\leq j<i\leq 4}\mathfrak{g}_{\pm e_i\pm e_j}. 
\end{align*}
This means that 
\begin{align*}
\mathfrak{l}\simeq \mathfrak{so}(8,\mathbb{C})\oplus \mathbb{C}^2. 
\end{align*}

The $\operatorname{ad}(H)$-eigenspace $\mathfrak{g}(2)$ is written as 
\begin{align*}
\mathfrak{g}(2)=\bigoplus _{\sum _{j=1}^4 n(j)=0,2,4}
	\mathfrak{g}_{\frac{1}{2}(e_8-e_7-e_6+e_5)+\frac{1}{2}\sum _{j=1}^4 (-1)^{n(j)}e_j}. 
\end{align*}
Then, the semisimple part $Spin(8,\mathbb{C})$ of $L_{\mathbb{C}}$ acts on 
\begin{align*}
\mathfrak{n}\simeq \mathbb{C}^8 
\end{align*}
as the half-spin representation. 
We know that this action is geometrically equivalent to the $SO(8,\mathbb{C})$-action 
on $\mathbb{C}^8$. 
On the other hand, the center $(\mathbb{C}^{\times})^2$ of $L_{\mathbb{C}}$ acts on $\mathfrak{n}$ by 
\begin{align*}
(s,t)\cdot v=a^2t^3v. 
\end{align*}
Hence, the $L_{\mathbb{C}}$-action on $\mathfrak{n}$ 
is geometrically equivalent to the $(SO(8,\mathbb{C})\times \mathbb{C}^{\times})$-action 
on $\mathbb{C}^8$. 
It follows from (6) of Table \ref{table:linear} that 
this action is a multiplicity-free action. 

We take the subspace $S_0$ in $\mathfrak{n}$ as 
\begin{align*}
S_0&=\mathbb{R}E_{\frac{1}{2}(e_8-e_7-e_6+e_5+e_4+e_3+e_2+e_1)}
	\oplus \mathbb{R}E_{\frac{1}{2}(e_8-e_7-e_6+e_5-e_4-e_3-e_2-e_1)}. 
\end{align*}
Then, $S_0$ is isomorphic to the slice $D_{1,1}$ of Table \ref{table:choice} 
for the $(SO(8)\times \mathbb{T})$-action on $\mathbb{C}^8$. 
By Lemma \ref{lem:linear}, 
$\mathfrak{n}=L_u\cdot S_0$. 
Hence, Theorem \ref{thm:slice-n} for Case (E$_6$\ref{case:e62}) has been verified . 

\subsubsection{Case (E$_6$\ref{case:e63})}

Let $\Omega (\mathcal{O}_X)=(0,0,0,1,0,0)$. 
Then, $H$ is written by 
\begin{align*}
H=E_8-E_7-E_6+E_5+E_4+E_3. 
\end{align*}

The Levi subalgebra $\mathfrak{l}$ is given by 
\begin{multline*}
\mathfrak{l}=\mathfrak{a}\oplus 
	\mathfrak{g}_{\pm \frac{1}{2}(e_8-e_7-e_6-e_5-e_4-e_3)\pm \frac{1}{2}(e_2-e_1)}
	\oplus 
	\mathfrak{g}_{\pm e_2\pm e_1}\oplus 
	\bigoplus _{3\leq j<i\leq 5}\mathfrak{g}_{\pm (e_{i}-e_{j})}. 
\end{multline*}
This means that 
\begin{align*}
\mathfrak{l}\simeq \mathfrak{sl}(3,\mathbb{C})\oplus \mathfrak{sl}(2,\mathbb{C})\oplus 
	\mathfrak{sl}(3,\mathbb{C})\oplus \mathbb{C}. 
\end{align*}

The $\operatorname{ad}(H)$-eigenspace $\mathfrak{g}(2)$ is written as 
\begin{align*}
\mathfrak{g}(2)&=\bigoplus _{3\leq j<i\leq 5}\mathfrak{g}_{e_{i}+e_{j}}
	\oplus \bigoplus _{\sum _{j=3}^5 n(j)=1}
	\mathfrak{g}_{\frac{1}{2}(e_8-e_7-e_6+\sum _{j=3}^5 (-1)^{n(j)}e_j)
		\pm \frac{1}{2}(e_2-e_1)}. 
\end{align*}
Then, $\mathfrak{g}(2)$ is isomorphic to the vector space $M(3,\mathbb{C})$. 
Further, the eigenspace $\mathfrak{g}(3)$ forms 
\begin{align*}
\mathfrak{g}(3)&=\mathfrak{g}_{\frac{1}{2}(e_8-e_7-e_6+e_5+e_4+e_3)
	\pm \frac{1}{2}(e_2+e_1)}, 
\end{align*}
which is isomorphic to $\mathbb{C}^2$. 
Hence, 
\begin{align*}
\mathfrak{n}=\mathfrak{g}(2)\oplus \mathfrak{g}(3)\simeq M(3,\mathbb{C})\oplus \mathbb{C}^2. 
\end{align*}

The semisimple part $SL(3,\mathbb{C})\times SL(2,\mathbb{C})\times SL(3,\mathbb{C})$ 
of the Levi subgroup $L_{\mathbb{C}}$ acts on $M(3,\mathbb{C})\oplus \mathbb{C}^2$ as follows: 
$SL(3,\mathbb{C})\times SL(3,\mathbb{C})$ acts by 
\begin{align*}
(g_1,g_2)\cdot (A,v)=(g_1Ag_2^{-1},v), 
\end{align*}
and $SL(2,\mathbb{C})$ acts by 
\begin{align*}
h\cdot (A,v)=(A,hv). 
\end{align*}
The center $\mathbb{C}^{\times }$ acts on $M(3,\mathbb{C})\oplus \mathbb{C}^2$ by 
\begin{align*}
s\cdot (A,v)=(s^2A,s^3v). 
\end{align*}
Then, the $L_{\mathbb{C}}$-action on $\mathfrak{n}$ is geometrically equivalent to 
the decomposable action consisting of two irreducible actions; 
the $(SL(3,\mathbb{C})\times SL(3,\mathbb{C})\times \mathbb{C}^{\times})$-action 
on $M(3,\mathbb{C})$ ((5) of Table \ref{table:linear}); 
the $SL(2,\mathbb{C})$-action on $\mathbb{C}^2$ ((2) of Table \ref{table:linear}). 
Thus, 
this action is a multiplicity-free action. 

We take the subspace $S_0'$ in $\mathfrak{g}(2)$ as 
\begin{multline*}
S_0':=\mathbb{R}E_{\frac{1}{2}(e_8-e_7-e_6+e_5+e_4-e_3+e_2-e_1)}\\
	\oplus \mathbb{R}E_{\frac{1}{2}(e_8-e_7-e_6+e_5-e_4+e_3-e_2+e_1)}
	\oplus \mathbb{R}E_{e_4+e_3}. 
\end{multline*}
Then, $S_0'$ is isomorphic to the slice $D_3$ of Table \ref{table:choice} 
for the $(SU(3)\times SU(3)\times \mathbb{T})$-action on $M(3,\mathbb{C})$. 
It follows from Lemma \ref{lem:linear} that 
$\mathfrak{g}(2)=(SU(3)\times SU(3)\times \mathbb{T})\cdot S_0'$. 
Similarly, the subspace $S_0''$ in $\mathfrak{g}(3)$ given by 
\begin{align*}
S_0'':=\mathbb{R}E_{\frac{1}{2}(e_8-e_7-e_6+e_5+e_4+e_3+e_2+e_1)} 
\end{align*}
is isomorphic to $T_1$ of Table \ref{table:linear}, 
and then $\mathfrak{g}(3)=SU(2)\cdot S_0''$. 
Thus, 
\begin{align*}
S_0:=S_0'\oplus S_0''\simeq D_3\oplus T_1. 
\end{align*}
satisfies $\mathfrak{n}=L_u\cdot S_0$. 
Therefore, Theorem \ref{thm:slice-n} for Case (E$_6$\ref{case:e63}) has been verified. 


\subsection{Type E$_7$}
\label{subsec:e7}

In this subsection, 
we give a proof of Theorem \ref{thm:slice-n} for $\mathfrak{g}=\mathfrak{e}_7(\mathbb{C})$. 
In this case, $\mathfrak{g}_{\mathbb{R}}\simeq \mathfrak{e}_{7(7)}$. 
Then, 
$\mathfrak{a}_{\mathbb{R}}^*=\{ v\in \mathbb{R}e_1\oplus \cdots \oplus \mathbb{R}e_8
:\langle v,e_8+e_7\rangle =0\} =\mathbb{R}(e_8-e_7)\oplus \mathbb{R}e_6\oplus \cdots \oplus 
\mathbb{R}e_1$. 
A root system $\Delta \equiv \Delta (\mathfrak{g},\mathfrak{a})$ is 
$\Delta =\{ \pm e_i\pm e_j:1\leq j<i\leq 6\} \sqcup \{ \pm (e_8-e_7)\} 
\sqcup \{ \pm \frac{1}{2}(e_8-e_7+\sum _{j=1}^6 (-1)^{n(j)}e_j):\sum _{j=1}^6 n(j)=1,3,5\} $ 
where $n(1),\ldots ,n(6)\in \{ 0,1\} $. 
We fix a positive system $\Delta ^+$ as 
$\Delta ^+=\{ e_i\pm e_j:1\leq j<i\leq 6\} \sqcup \{ e_8-e_7\} 
	\sqcup \{ \frac{1}{2}(e_8-e_7+\sum _{j=1}^6 (-1)^{n(j)}e_j):\sum _{j=1}^6 n(j)=1,3,5\} $. 
The simple roots 
$\alpha _1,\ldots ,\alpha _7 $ are given by 
$\alpha _1=\frac{1}{2}(e_8-e_7-e_6-e_5-e_4-e_3-e_2+e_1)$, 
$\alpha _2=e_2+e_1$, and $\alpha _{j+2}=e_{j+1}-e_j$ $(j=1,\ldots ,5)$. 
The highest root is written as $\beta =e_8-e_7=2\alpha _1+2\alpha _2+3\alpha _3
+4\alpha _4+3\alpha _5+2\alpha _6+\alpha _7$. 

Let $\mathcal{O}_X$ be a nilpotent orbit in $\mathfrak{e}_7(\mathbb{C})$ 
with characteristic element $H=h_7(E_8-E_7)+h_6E_6+\cdots +h_1E_1\in \mathfrak{a}_+$. 
Then, 
$\alpha _1(H)=\frac{1}{2}(2h_7-h_6-h_5-h_4-h_3-h_2+h_1),~
\alpha _2(H)=h_2+h_1$, and 
$\alpha _i(H)=h_{i-1}-h_{i-2}\ (i=3,4,5,6,7)$. 
Hence, the weighted Dynkin diagram $\Omega (\mathcal{O}_X)=(m_1,\ldots ,m_7)$ 
is given by 
$(\frac{1}{2}(2h_7-h_6-h_5-h_4-h_3-h_2+h_1),h_2+h_1,h_2-h_1,h_3-h_2,h_4-h_3,h_5-h_4,h_6-h_5)$. 

\begin{figure}[htbp]
\begin{align*}
\SelectTips{cm}{12}
	\objectmargin={1pt}
	\xygraph{
		\circ ([]!{+(0,-.3)} {\alpha_1},[]!{+(0,+.3)} {m_1}) - [r]
		\circ ([]!{+(0,-.3)} {\alpha_3},[]!{+(0,+.3)} {m_3}) - [r]
		\circ ([]!{+(.3,-.3)} {\alpha_4},[]!{+(0,+.3)} {m_4}) (
			- [d] \circ ([]!{+(0,-0.3)} {\alpha_2},[]!{+(0.3,0)} {m_2}),
			- [r] \circ ([]!{+(0,-.3)} {\alpha_5},[]!{+(0,+.3)} {m_5})
			- [r] \circ ([]!{+(0,-.3)} {\alpha_6},[]!{+(0,+.3)} {m_6})
			- [r] \circ ([]!{+(0,-.3)} {\alpha_7},[]!{+(0,+.3)} {m_7}))
	}
\end{align*}
\caption{Weighted Dynkin diagram of $\mathcal{O}_X$ in $\mathfrak{e}_7(\mathbb{C})$}
\end{figure}

A nilpotent orbit $\mathcal{O}_X$ is spherical if and only if and only if 
$\Omega (\mathcal{O}_X)$ satisfies one of the following cases: 
\begin{enumerate}
	\renewcommand{\labelenumi}{(E$_7$\theenumi) }
	\item $\Omega (\mathcal{O}_X)=(1,0,0,0,0,0,0)$. 
	\label{case:e71}
	\item $\Omega (\mathcal{O}_X)=(0,0,0,0,0,1,0)$. 
	\label{case:e72}
	\item $\Omega (\mathcal{O}_X)=(0,0,0,0,0,0,2)$. 
	\label{case:e73}
	\item $\Omega (\mathcal{O}_X)=(0,0,1,0,0,0,0)$. 
	\label{case:e74}
	\item $\Omega (\mathcal{O}_X)=(0,1,0,0,0,0,1)$. 
	\label{case:e75}
\end{enumerate}
By Lemma \ref{lem:height-max}, 
the height of $\mathcal{O}_X$ equals two for Cases 
(E$_7$\ref{case:e71})--(E$_7$\ref{case:e73}) and three 
for Cases (E$_7$\ref{case:e74}), (E$_7$\ref{case:e75}). 

\subsubsection{Case (E$_7$\ref{case:e71})}
\label{subsubsec:e7}

Let $\Omega (\mathcal{O}_X)=(1,0,0,0,0,0,0)$. 
Then, $H$ is of the form 
\begin{align*}
H=E_8-E_7. 
\end{align*}
This $\mathcal{O}_X$ is the minimal nilpotent orbit with dimension $34$. 

The Levi subalgebra $\mathfrak{l}$ is given by 
\begin{align*}
\mathfrak{l}=\mathfrak{a}\oplus \bigoplus _{1\leq j<i\leq 6}\mathfrak{g}_{\pm e_i\pm e_j}. 
\end{align*}
This means that 
\begin{align*}
\mathfrak{l}\simeq \mathfrak{so}(12,\mathbb{C})\oplus \mathbb{C}. 
\end{align*}

The $\operatorname{ad}(H)$-eigenspace $\mathfrak{g}(2)$ is written as 
\begin{align*}
\mathfrak{g}(2)=\mathfrak{g}_{e_8-e_7}\simeq \mathbb{C}. 
\end{align*}
Hence, the $L_{\mathbb{C}}$-action on $\mathfrak{n}$ is 
geometrically equivalent to the $\mathbb{C}^{\times}$-action on $\mathbb{C}$. 
It follows from (1) of Table \ref{table:linear} 
that this action is a multiplicity-free action. 

We take 
\begin{align*}
S_0=\mathbb{R}E_{e_8-e_7}\simeq \mathbb{R}. 
\end{align*}
By Lemma \ref{lem:linear}, $S_0$ satisfies $\mathfrak{n}=L_u\cdot S_0$. 
Hence, Theorem \ref{thm:slice-n} for Case (E$_7$\ref{case:e71}) has been verified. 

\subsubsection{Case (E$_7$\ref{case:e72})}

We consider the case $\Omega (\mathcal{O}_X)=(0,0,0,0,0,1,0)$. 
Then, 
\begin{align*}
H=E_8-E_7+E_6+E_5. 
\end{align*}

The Levi subalgebra $\mathfrak{l}$ is given by 
\begin{multline*}
\mathfrak{l}=\mathfrak{a}\oplus \bigoplus _{1\leq j<i\leq 4}\mathfrak{g}_{\pm e_i\pm e_j}
	\oplus \mathfrak{g}_{\pm (e_6-e_5)}\\
	\oplus \bigoplus _{\sum _{j=1}^4 n(j)=1,3}\mathfrak{g}_{\pm \frac{1}{2}(e_8-e_7-e_6-e_5
		+\sum _{j=1}^4 (-1)^{n(j)}e_j)}. 
\end{multline*}
This means that 
\begin{align*}
\mathfrak{l}\simeq 
\mathfrak{so}(10,\mathbb{C})\oplus \mathfrak{sl}(2,\mathbb{C})\oplus \mathbb{C}. 
\end{align*}

The $\operatorname{ad}(H)$-eigenspace $\mathfrak{g}(2)$ is written as 
\begin{align*}
\mathfrak{g}(2)=\mathfrak{g}_{e_8-e_7}\oplus \mathfrak{g}_{e_6+e_5}
	\oplus \bigoplus _{\sum _{j=1}^4n(j)=1,3}
	\mathfrak{g}_{\frac{1}{2}(e_8-e_7+e_6+e_5)+\frac{1}{2}\sum _{j=1}^4 (-1)^{n(j)}e_j}. 
\end{align*}
Then, $\mathfrak{g}(2)$ is isomorphic to $\mathbb{C}^{10}$, from which 
\begin{align*}
\mathfrak{n}\simeq \mathbb{C}^{10}. 
\end{align*}

The action of the semisimple part $SO(10,\mathbb{C})\times SL(2,\mathbb{C})$ of the Levi subgroup 
$L_{\mathbb{C}}$ on $\mathbb{C}^{10}$ is given as follows: 
$SO(10,\mathbb{C})$ acts by the standard representation; 
$SL(2,\mathbb{C})$ acts trivially, 
and the action of the center $\mathbb{C}^{\times}$ is the scalar multiplication. 
Then, the $L_{\mathbb{C}}$-action on $\mathfrak{n}$ is geometrically equivalent to 
the $(SO(10,\mathbb{C})\times \mathbb{C}^{\times})$-action on $\mathbb{C}^{10}$. 
It follows from (6) of Table \ref{table:linear} that 
this action is a multiplicity-free action. 

We take a subspace $S_0$ as 
\begin{align*}
S_0:=\mathbb{R}E_{e_8-e_7}\oplus \mathbb{R}E_{e_6+e_5}. 
\end{align*}
Then, $S_0$ is isomorphic to the slice $D_{1,1}$ of Table \ref{table:choice} 
for the $(SO(10)\times \mathbb{T})$-action via $\mathfrak{n}\simeq \mathbb{C}^{10}$. 
By Lemma \ref{lem:linear}, $\mathfrak{n}=L_u\cdot S_0$. 
Hence, Theorem \ref{thm:slice-n} for Case (E$_7$\ref{case:e72}) has been verified. 

\subsubsection{Case (E$_7$\ref{case:e73})}

Let $\Omega (\mathcal{O}_X)$ be $(0,0,0,0,0,0,2)$. 
Then, 
\begin{align*}
H=E_8-E_7+2E_6. 
\end{align*}

The Levi subalgebra $\mathfrak{l}$ is given by 
\begin{align*}
\mathfrak{l}=\mathfrak{a}\oplus 
	\bigoplus _{1\leq j<i\leq 5}\mathfrak{g}_{\pm e_i\pm e_j}
	\oplus \bigoplus _{\sum _{j=1}^5 n(j)=0,2,4}
		\mathfrak{g}_{\pm \frac{1}{2}(e_8-e_7-e_6+\sum _{j=1}^5(-1)^{n(j)}e_j)}. 
\end{align*}
This means that 
\begin{align*}
\mathfrak{l}\simeq \mathfrak{e}_6(\mathbb{C})\oplus \mathbb{C}. 
\end{align*}

The $\operatorname{ad}(H)$-eigenspace $\mathfrak{g}(2)$ is written as 
\begin{multline*}
\mathfrak{g}(2)=\bigoplus _{1\leq j\leq 5}\mathfrak{g}_{e_6\pm e_j}
	\oplus \mathfrak{g}_{e_8-e_7}
	\oplus \bigoplus _{\sum _{j=1}^5 n(j)=0,2,4}
		\mathfrak{g}_{\frac{1}{2}(e_8-e_7+e_6+\sum _{j=1}^5(-1)^{n(j)}e_j)}. 
\end{multline*}
Then, $\mathfrak{g}(2)$ is isomorphic to the complexified Jordan algebra $\mathfrak{J}_{\mathbb{C}}$, 
from which 
\begin{align*}
\mathfrak{n}\simeq \mathfrak{J}_{\mathbb{C}}. 
\end{align*}
This implies that the $L_{\mathbb{C}}$-action on $\mathfrak{n}$ is geometrically equivalent to 
the $(E_6(\mathbb{C})\times \mathbb{C}^{\times})$-action on $\mathfrak{J}_{\mathbb{C}}$. 
It follows from (7) of Table \ref{table:linear} that 
this action is a multiplicity-free action. 

We take the subspace $S_0$ in $\mathfrak{n}$ as 
\begin{align*}
S_0=\mathbb{R}E_{e_8-e_7}\oplus \mathbb{R}E_{e_6+e_5}
\oplus \mathbb{R}E_{e_6-e_5}. 
\end{align*}
Then, this is isomorphic to the slice $D_3$ of Table \ref{table:choice} 
for the $(E_6\times \mathbb{T})$-action on $\mathfrak{J}_{\mathbb{C}}$. 
By Lemma \ref{lem:linear}, $S_0$ satisfies $\mathfrak{n}=L_u\cdot S_0$. 
Hence, Theorem \ref{thm:slice-n} for Case (E$_7$\ref{case:e73}) has been verified in this case. 

\subsubsection{Case (E$_7$\ref{case:e74})}

We consider the case $\Omega (\mathcal{O}_X)=(0,0,1,0,0,0,0)$. 
Then, the characteristic element $H$ of $\mathcal{O}_X$ is expressed by 
\begin{align*}
H=\frac{1}{2}(3E_8-3E_7+E_6+E_5+E_4+E_3+E_2-E_1). 
\end{align*}

The Levi subalgebra $\mathfrak{l}$ is given by 
\begin{multline*}
\mathfrak{l}=\mathfrak{a}\oplus \bigoplus _{2\leq i\leq 6}\mathfrak{g}_{\pm (e_{i}+e_1)}
	\oplus \bigoplus _{2\leq j<i\leq 6}\mathfrak{g}_{\pm (e_{i}-e_{j})}
	\oplus \mathfrak{g}_{\pm \frac{1}{2}(e_8-e_7-e_6-e_5-e_4-e_3-e_2+e_1)}. 
\end{multline*}
This means that 
\begin{align*}
\mathfrak{l}\simeq \mathfrak{sl}(6,\mathbb{C})\oplus \mathfrak{sl}(2,\mathbb{C})
\oplus \mathbb{C}. 
\end{align*}

The $\operatorname{ad}(H)$-eigenspace $\mathfrak{g}(2)$ is written as 
\begin{multline*}
\mathfrak{g}(2)=\bigoplus _{\sum _{j=2}^6n(j)=1}
	\mathfrak{g}_{\frac{1}{2}(e_8-e_7+\sum _{j=2}^6 (-1)^{n(j)}e_j+e_1))}\\
	\oplus \bigoplus _{\sum _{j=2}^6n(j)=2}
	\mathfrak{g}_{\frac{1}{2}(e_8-e_7+\sum _{j=2}^6 (-1)^{n(j)}e_j-e_1))}. 
\end{multline*}
Then, $\mathfrak{g}(2)$ is isomorphic to $\operatorname{Alt}(6,\mathbb{C})$. 
Further, $\mathfrak{g}(3)$ forms 
\begin{align*}
\mathfrak{g}(3)=\mathfrak{g}_{e_8-e_7}\oplus 
\mathfrak{g}_{\frac{1}{2}(e_8-e_7+e_6+e_5+e_4+e_3+e_2-e_1)}\simeq \mathbb{C}^2. 
\end{align*}
Thus, 
\begin{align*}
\mathfrak{n}\simeq \operatorname{Alt}(6,\mathbb{C})\oplus \mathbb{C}^2. 
\end{align*}

The semisimple part $SL(6,\mathbb{C})\times SL(2,\mathbb{C})$ 
of the Levi subgroup 
$L_{\mathbb{C}}$ acts on $\operatorname{Alt}(6,\mathbb{C})\oplus \mathbb{C}^2$ by 
\begin{align*}
(g,h)\cdot (A,v)=(gA\,{}^tg,hv), 
\end{align*}
and its center $\mathbb{C}^{\times}$ acts by 
\begin{align*}
s\cdot (A,v)=(s^2A,s^3v). 
\end{align*}
Then, the $L_{\mathbb{C}}$-action on $\mathfrak{n}$ is geometrically equivalent to 
the decomposable action consisting of the $(SL(6,\mathbb{C})\times \mathbb{C}^{\times})$-action on 
$\operatorname{Alt}(6,\mathbb{C})$ ((4) of Table \ref{table:linear}) 
and the $SL(2,\mathbb{C})$-action on $\mathbb{C}^2$ ((2) of Table \ref{table:linear}). 
It follows from Lemma \ref{lem:linear} that 
this action is a multiplicity-free action. 

We take the subspace $S_0'$ in $\mathfrak{g}(2)$ as 
\begin{multline*}
S_0':=\mathbb{R}E_{\frac{1}{2}(e_8-e_7+e_6+e_5+e_4+e_3-e_2+e_1)}
	\oplus \mathbb{R}E_{\frac{1}{2}(e_8-e_7+e_6+e_5-e_4-e_3+e_2-e_1)}\\
	\oplus \mathbb{R}E_{\frac{1}{2}(e_8-e_7-e_6-e_5+e_4+e_3+e_2-e_1)}. 
\end{multline*}
Then, $S_0'$ is isomorphic to the slice $A_3$ of Table \ref{table:choice} 
for the $(SU(6)\times \mathbb{T})$-action on $\operatorname{Alt}(6,\mathbb{C})$. 
By Lemma \ref{lem:linear}, 
$\mathfrak{g}(2)=(SU(6)\times \mathbb{T})\cdot S_0'$. 
Similarly, the subspace 
\begin{align*}
S_0''&:=\mathbb{R}E_{e_8-e_7}. 
\end{align*}
in $\mathfrak{g}(3)$ is isomorphic to $T_1$ and 
satisfies $\mathfrak{g}(3)=SU(2)\cdot S_0''$. 
Hence, 
\begin{align*}
S_0&:=S_0'\oplus S_0''\simeq A_3\oplus T_1
\end{align*}
satisfies $\mathfrak{n}=L_u\cdot S_0$. 
Therefore, Theorem \ref{thm:slice-n} has been verified for Case (E$_7$\ref{case:e74}). 

\subsubsection{Case (E$_7$\ref{case:e75})}

Let $\Omega (\mathcal{O}_X)$ be $(0,1,0,0,0,0,1)$. 
Then, 
\begin{align*}
H=\frac{1}{2}(3E_8-3E_7+3E_6+E_5+E_4+E_3+E_2+E_1). 
\end{align*}

The Levi subalgebra $\mathfrak{l}$ is given by 
\begin{align*}
\mathfrak{l}=\mathfrak{a}\oplus \bigoplus _{1\leq j<i\leq 5}\mathfrak{g}_{\pm (e_i-e_j)}
	\oplus \bigoplus _{\sum _{j=1}^5 n(j)=4}
	\mathfrak{g}_{\pm \frac{1}{2}(e_8-e_7-e_6+\sum _{j=1}^5 (-1)^{n(j)}e_j)}. 
\end{align*}
This implies that 
\begin{align*}
\mathfrak{l}\simeq \mathfrak{sl}(6,\mathbb{C})\oplus \mathbb{C}^2. 
\end{align*}

The $\operatorname{ad}(H)$-eigenspace $\mathfrak{g}(2)$ is written as 
\begin{multline*}
\mathfrak{g}(2)=\bigoplus _{1\leq j\leq 5}\mathfrak{g}_{e_6+e_j}\oplus 
	\bigoplus _{\sum _{j=1}^5 n(j)=3}
	\mathfrak{g}_{\pm \frac{1}{2}(e_8-e_7+e_6+\sum _{j=1}^5 (-1)^{n(j)}e_j)}\\
	\oplus \mathfrak{g}_{\frac{1}{2}(e_8-e_7-e_6+e_5+e_4+e_3+e_2+e_1)}. 
\end{multline*}
Then, $\mathfrak{g}(2)$ is isomorphic to the direct sum $\operatorname{Alt}(6,\mathbb{C})\oplus \mathbb{C}$. 
Further, $\mathfrak{g}(3)$ is of the form 
\begin{align*}
\mathfrak{g}(3)=\mathfrak{g}_{e_8-e_7}\oplus \bigoplus _{\sum _{j=1}^5 n(j)=1}
	\mathfrak{g}_{\frac{1}{2}(e_8-e_7+e_6+\sum _{j=1}^5 (-1)^{n(j)}e_j)}
	\simeq \mathbb{C}^6. 
\end{align*}
Hence, $\mathfrak{n}$ is isomorphic to 
\begin{align*}
\mathfrak{n}\simeq \operatorname{Alt}(6,\mathbb{C})\oplus \mathbb{C}\oplus \mathbb{C}^6. 
\end{align*}

The semisimple part $SL(6,\mathbb{C})$ of the Levi subgroup $L_{\mathbb{C}}$ 
acts on $\operatorname{Alt}(6,\mathbb{C})\oplus \mathbb{C}\oplus \mathbb{C}^6$ by 
\begin{align*}
g\cdot (A,z,v)
=(gA\,{}^tg,z,gv), 
\end{align*}
and its center $(\mathbb{C}^{\times})^2$ acts by 
\begin{align*}
(s,t)\cdot (A,z,v)=(s^2tA,s^2t^3z,s^3t^2v). 
\end{align*}
This implies that 
the $L_{\mathbb{C}}$-action on $\mathfrak{n}$ is geometrically equivalent to 
the decomposable action which consists two actions: 
the indecomposable $(SL(6,\mathbb{C})\times \mathbb{C}^{\times})$-action on 
$\operatorname{Alt}(6,\mathbb{C})\oplus \mathbb{C}^6$ ((8) of Table \ref{table:linear}); 
the $\mathbb{C}^{\times}$-action on $\mathbb{C}$ ((1) of Table \ref{table:linear}). 
By Lemma \ref{lem:linear}, 
this action is a multiplicity-free action. 

We take the subspace $S_0'$ as 
\begin{multline*}
S_0':=\mathbb{R}E_{\frac{1}{2}(e_8-e_7+e_6+e_5+e_4-e_3-e_2-e_1)}\\
	\oplus 
	\mathbb{R}E_{\frac{1}{2}(e_8-e_7+e_6-e_5-e_4+e_3+e_2-e_1)}
	\oplus \mathbb{R}E_{e_6+e_1}, 
\end{multline*}
and $S_1'$ as 
\begin{multline*}
S_1':=\mathbb{R}E_{e_8-e_7}
	\oplus \mathbb{R}E_{\frac{1}{2}(e_8-e_7+e_6+e_5+e_4+e_3-e_2+e_1)}\\
	\oplus \mathbb{R}E_{\frac{1}{2}(e_8-e_7+e_6+e_5-e_4+e_3+e_2+e_1)}. 
\end{multline*}
Then, the direct sum $S_0'\oplus S_1'$ is isomorphic to the slice $A_3\oplus T_3$ of Table \ref{table:choice} 
for the $(SU(6)\times \mathbb{T})$-action on $\operatorname{Alt}(6,\mathbb{C})\oplus \mathbb{C}^6$. 
Further, the subspace 
\begin{align*}
S_0''&:=\mathbb{R}E_{\frac{1}{2}(e_8-e_7-e_6+e_5+e_4+e_3+e_2+e_1)} 
\end{align*}
in $\mathfrak{g}(2)$ is isomorphic to the slice $\mathbb{R}$ 
for the $\mathbb{T}$-action on $\mathbb{C}$. 
We set 
\begin{align*}
S_0&:=(S_0'\oplus S_0'')\oplus S_1'\simeq (A_3\oplus \mathbb{R})\oplus T_3. 
\end{align*}
By Lemma \ref{lem:linear}, 
$S_0$ satisfies $\mathfrak{n}=L_u\cdot S_0$. 
Therefore, Theorem \ref{thm:slice-n} has been verified for Case (E$_7$\ref{case:e75}). 


\subsection{Type E$_8$}
\label{subsec:e8}

In this subsection, 
we give a proof of Theorem \ref{thm:slice-n} for $\mathfrak{g}=\mathfrak{e}_8(\mathbb{C})$. 
In this case, $\mathfrak{g}_{\mathbb{R}}\simeq \mathfrak{e}_{8(8)}$. 
Then, $\mathfrak{a}_{\mathbb{R}}^*=\mathbb{R}^8$. 
A root system of $\mathfrak{g}$ is given by 
$\Delta =\{ \pm e_i\pm e_j:1\leq j<i\leq 8\} 
\sqcup \{ \pm \frac{1}{2}(e_8+\sum _{j=1}^7 (-1)^{n(j)}e_j):\sum _{j=1}^7 n(j)=0,2,4,6\} $ 
where $n(1),\ldots ,n(7)\in \{ 0,1\}$. 
We fix a positive system as 
$\Delta ^+=\{ e_i\pm e_j:1\leq j<i\leq 8\} 
\sqcup \{ \frac{1}{2}(e_8+\sum _{j=1}^7 (-1)^{n(j)}e_j):\sum _{j=1}^7 n(j)=0,2,4,6\} $. 
The simple roots $\alpha _1,\ldots ,\alpha _8$ are given by 
$\alpha _1=\frac{1}{2}(e_8-e_7-e_6-e_5-e_4-e_3-e_2+e_1)$, 
$\alpha _2=e_2+e_1$, and $\alpha _{i}=e_{i-1}-e_{i-2}$ $(3\leq i\leq 8)$. 
The highest root $\beta $ is written as 
$\beta =e_8+e_7=2\alpha _1+3\alpha _2+4\alpha _3+6\alpha _4+5\alpha _5
+4\alpha _6+3\alpha _7+2\alpha _8$. 

Let $\mathcal{O}_X$ be a nilpotent orbit in $\mathfrak{e}_8(\mathbb{C})$ 
with characteristic element $H=h_8E_8+\cdots +h_1E_1\in \mathfrak{a}_+$. 
Then, 
$\alpha _1(H)=\frac{1}{2}(h_8-h_7-h_6-h_5-h_4-h_3-h_2+h_1),~
\alpha _2(H)=h_2+h_1$, and 
$\alpha _i(H)=h_{i-1}-h_{i-2}\ (i=3,4,\ldots ,8)$. 
Hence, the weighted Dynkin diagram $\Omega (\mathcal{O}_X)=(m_1,\ldots ,m_8)$ is given by 
$(\frac{1}{2}(h_8-h_7-h_6-h_5-h_4-h_3-h_2+h_1),h_2+h_1,h_2-h_1,h_3-h_2,h_4-h_4,h_5-h_4,h_6-h_5,
h_7-h_6)$. 

\begin{figure}[htbp]
\begin{align*}
\SelectTips{cm}{12}
	\objectmargin={1pt}
	\xygraph{
		\circ ([]!{+(0,-.3)} {\alpha_1},[]!{+(0,+.3)} {m_1}) - [r]
		\circ ([]!{+(0,-.3)} {\alpha_3},[]!{+(0,+.3)} {m_3}) - [r]
		\circ ([]!{+(.3,-.3)} {\alpha_4},[]!{+(0,+.3)} {m_4}) (
			- [d] \circ ([]!{+(0,-0.3)} {\alpha_2},[]!{+(0.3,0)} {m_2}),
			- [r] \circ ([]!{+(0,-.3)} {\alpha_5},[]!{+(0,+.3)} {m_5})
			- [r] \circ ([]!{+(0,-.3)} {\alpha_6},[]!{+(0,+.3)} {m_6})
			- [r] \circ ([]!{+(0,-.3)} {\alpha_7},[]!{+(0,+.3)} {m_7})
			- [r] \circ ([]!{+(0,-.3)} {\alpha_8},[]!{+(0,+.3)} {m_8})
		)
	}
\end{align*}
\caption{Weighted Dynkin diagram of $\mathcal{O}_X$ in $\mathfrak{e}_8(\mathbb{C})$}
\end{figure}

A nilpotent orbit $\mathcal{O}_X$ in $\mathfrak{e}_8(\mathbb{C})$ is spherical if and only if 
$\Omega (\mathcal{O}_X)$ satisfies one of the following cases: 

\begin{enumerate}
	\renewcommand{\labelenumi}{(E$_8$\theenumi) }
	\item $\Omega (\mathcal{O}_X)=(0,0,0,0,0,0,0,1)$. 
	\label{case:e81}
	\item $\Omega (\mathcal{O}_X)=(1,0,0,0,0,0,0,0)$. 
	\label{case:e82}
	\item $\Omega (\mathcal{O}_X)=(0,0,0,0,0,0,1,0)$. 
	\label{case:e83}
	\item $\Omega (\mathcal{O}_X)=(0,1,0,0,0,0,0,0)$. 
	\label{case:e84}
\end{enumerate}

It follows from Lemma \ref{lem:height-max} that 
the the height of $\mathcal{O}_X$ equals two 
for Cases (E$_8$\ref{case:e81}), (E$_8$\ref{case:e82}), 
and three for Cases (E$_8$\ref{case:e83}), (E$_8$\ref{case:e84}), 

\subsubsection{Case (E$_8$\ref{case:e81})}
\label{subsubsec:e8}

Let $\Omega (\mathcal{O}_X)$ be $(0,0,0,0,0,0,0,1)$. 
Then, $H\in \mathfrak{a}$ is given by 
\begin{align*}
H=E_8+E_7. 
\end{align*}
This $\mathcal{O}_X$ is the minimal nilpotent orbit in $\mathfrak{e}_8(\mathbb{C})$ 
with dimension 58. 

The Levi subalgebra $\mathfrak{l}=\mathfrak{g}(0)$ is given by 
\begin{multline*}
\mathfrak{l}=\mathfrak{a}\oplus \bigoplus _{1\leq j<i\leq 6}\mathfrak{g}_{\pm e_i\pm e_j}
	\oplus \mathfrak{g}_{\pm (e_8-e_7)} \\
\oplus \bigoplus _{\sum _{j=1}^6 n(j)=1,3,5}
	\mathfrak{g}_{\pm \frac{1}{2}(e_8-e_7+\sum _{j=1}^6(-1)^{n(j)}e_j)}. 
\end{multline*}
This implies that 
\begin{align*}
\mathfrak{l}\simeq \mathfrak{e}_7(\mathbb{C})\oplus \mathbb{C}. 
\end{align*}

The $\operatorname{ad}(H)$-eigenspace $\mathfrak{g}(2)$ is 
\begin{align*}
\mathfrak{g}(2)=\mathfrak{g}_{e_8-e_7}\simeq \mathbb{C}. 
\end{align*}
Thus, the action of the Levi subgroup $L_{\mathbb{C}}$ on $\mathfrak{n}$ 
is geometrically equivalent to the $\mathbb{C}^{\times}$-action on $\mathbb{C}$. 
It follows from (1) of Table \ref{table:linear} that 
this action is a multiplicity-free action. 

We take the subspace $S_0$ as 
\begin{align*}
S_0:=\mathbb{R}E_{e_8-e_7}\simeq \mathbb{R}. 
\end{align*}
Then, $\mathfrak{n}=L_u\cdot S_0$. 
Therefore, Theorem \ref{thm:slice-n} has been verified in this case. 

\subsubsection{Case (E$_8$\ref{case:e82})}

Let $\Omega (\mathcal{O}_X)$ be $(1,0,0,0,0,0,0,0)$. 
Then, 
\begin{align*}
H=2E_8. 
\end{align*}

The Levi subalgebra $\mathfrak{l}=\mathfrak{g}(0)$ is given by 
\begin{align*}
\mathfrak{l}=\mathfrak{a}\oplus \bigoplus _{1\leq j<i\leq 7}\mathfrak{g}_{\pm e_i\pm e_j}. 
\end{align*}
This means that 
\begin{align*}
\mathfrak{l}\simeq \mathfrak{so}(14,\mathbb{C})\oplus \mathbb{C}. 
\end{align*}

The $\operatorname{ad}(H)$-eigenspace $\mathfrak{g}(2)$ is written as 
\begin{align*}
\mathfrak{g}(2)=\bigoplus _{1\leq j\leq 7}\mathfrak{g}_{e_8\pm e_j}
\end{align*}
Then, $\mathfrak{g}(2)$ is isomorphic to $\mathbb{C}^{14}$, 
from which 
\begin{align*}
\mathfrak{n}\simeq \mathbb{C}^{14}. 
\end{align*}
This implies that 
the $L_{\mathbb{C}}$-action on $\mathfrak{n}$ 
is geometrically equivalent to the $(SO(14,\mathbb{C})\times \mathbb{C}^{\times})$-action 
on $\mathbb{C}^{14}$. 
It follows from (6) of Table \ref{table:linear} 
that this action is a multiplicity-free space. 

We take the subspace in $\mathfrak{n}$ as 
\begin{align*}
S_0:=\mathbb{R}E_{e_8+e_7}\oplus \mathbb{R}E_{e_8-e_7}. 
\end{align*}
Then, $S_0$ is isomorphic to the slice $D_{1,1}$ of Table \ref{table:choice} 
for the $(SO(14)\times \mathbb{T})$-action on $\mathbb{C}^{14}$. 
By Lemma \ref{lem:linear}, $\mathfrak{n}=L_u\cdot S_0$. 
Therefore, Theorem \ref{thm:slice-n} for Case (E$_8$\ref{case:e82}) has been verified. 

\subsubsection{Case (E$_8$\ref{case:e83})}

Let $\Omega (\mathcal{O}_X)$ be $(0,0,0,0,0,0,1,0)$. 
Then, $H$ forms 
\begin{align*}
H=2E_8+E_7+E_6. 
\end{align*}

The Levi subalgebra $\mathfrak{l}$ is given by 
\begin{multline*}
\mathfrak{l}=\mathfrak{a}\oplus \mathfrak{g}_{\pm (e_7-e_6)}
	\oplus \bigoplus _{1\leq j<i\leq 5}\mathfrak{g}_{\pm e_i\pm e_j}
	\\
	\oplus \bigoplus _{\sum _{j=1}^5 n(j)=0,2,4}\mathfrak{g}_
	{\pm \frac{1}{2}(e_8-e_7-e_6+\sum _{j=1}^5 (-1)^{n(j)}e_j)}. 
\end{multline*}
This implies that 
\begin{align*}
\mathfrak{l}\simeq \mathfrak{e}_6(\mathbb{C})\oplus \mathfrak{sl}(2,\mathbb{C})
	\oplus \mathbb{C}. 
\end{align*}

The $\operatorname{ad}(H)$-eigenspace $\mathfrak{g}(2)$ is written as 
\begin{align*}
\mathfrak{g}(2)&=\mathfrak{g}_{e_7+e_6}
	\oplus \bigoplus _{1\leq j\leq 5}\mathfrak{g}_{e_8\pm e_j}
	\oplus \bigoplus _{\sum _{j=1}^5 n(j)=0,2,4}\mathfrak{g}_
	{\frac{1}{2}(e_8+e_7+e_6+\sum _{j=1}^5 (-1)^{n(j)}e_j)}. 
\end{align*}
Then, $\mathfrak{g}(2)$ is isomorphic to the complexified Jordan algebra $\mathfrak{J}_{\mathbb{C}}$. 
Further, 
\begin{align*}
\mathfrak{g}(3)&=\mathfrak{g}_{e_8+e_7}\oplus \mathfrak{g}_{e_8+e_6}\simeq \mathbb{C}^2. 
\end{align*}
Hence, 
\begin{align*}
\mathfrak{n}\simeq \mathfrak{J}_{\mathbb{C}}\oplus \mathbb{C}^2. 
\end{align*}

The semisimple part of the Levi subgroup $L_{\mathbb{C}}$ is isomorphic to 
$E_6(\mathbb{C})\times SL(2,\mathbb{C})$, 
which acts on $\mathfrak{J}_{\mathbb{C}}\oplus \mathbb{C}^2$ diagonally, namely, 
\begin{align*}
(g,h)\cdot (v,w)=(gv,hw). 
\end{align*}
Further. the center $\mathbb{C}^{\times}$ acts by 
\begin{align*}
s\cdot (v,w)=(s^2v,s^3w). 
\end{align*}
This implies that 
the $L_{\mathbb{C}}$-action on $\mathfrak{n}$ is geometrically equivalent to 
the decomposable action consisting of the irreducible 
$(E_6(\mathbb{C})\times \mathbb{C}^{\times})$-action on $\mathbb{C}^{27}$ 
((7) of Table \ref{table:linear}) and the irreducible $SL(2,\mathbb{C})$-action on $\mathbb{C}^2$ 
((2) of Table \ref{table:linear}). 
By Lemma \ref{lem:linear}, 
this action is a multiplicity-free action. 

We take the subspace in $\mathfrak{g}(2)$ as 
\begin{multline*}
S_0':=\mathbb{R}E_{e_8+e_5}\oplus 
	\mathbb{R}E_{\frac{1}{2}(e_8+e_7+e_6-e_5+e_4+e_3+e_2-e_1)}\\
	\oplus \mathbb{R}E_{\frac{1}{2}(e_8+e_7+e_6-e_5-e_4-e_3-e_2+e_1)}. 
\end{multline*}
Then, $S_0'$ is isomorphic to the slice $D_3$ for the $(E_6\times \mathbb{T})$-action 
on $\mathbb{C}^{27}$. 
By Lemma \ref{lem:linear}, $\mathfrak{g}(2)=(E_6\times \mathbb{T})\cdot S_0'$. 
Similarly, the subspace 
\begin{align*}
S_0''=\mathbb{R}E_{e_8+e_7}. 
\end{align*}
in $\mathfrak{g}(3)$ is isomorphic to $T_1$, and then 
$\mathfrak{g}(3)=SU(2)\cdot S_0''$. 
Hence, 
\begin{align*}
S_0:=S_0'\oplus S_0''\simeq D_3\oplus T_1
\end{align*}
satisfies $\mathfrak{n}=L_u\cdot S_0$, 
from which Theorem \ref{thm:slice-n} for Case (E$_8$\ref{case:e83}) has been verified. 

\subsubsection{Case (E$_8$\ref{case:e84})}

Let $\Omega (\mathcal{O}_X)$ be $(0,1,0,0,0,0,0,0)$. 
Then, 
\begin{align*}
H=\frac{1}{2}(5E_8+E_7+E_6+\cdots +E_1). 
\end{align*}

The Levi subalgebra $\mathfrak{l}=\mathfrak{g}(0)$ is given by 
\begin{align*}
\mathfrak{l}=\mathfrak{a}\oplus \bigoplus _{1\leq j<i\leq 7}\mathfrak{g}_{\pm (e_i-e_j)}
	\oplus \bigoplus _{\sum _{j=1}^7n(j)=6}
	\mathfrak{g}_{\pm \frac{1}{2}(e_8+\sum _{j=1}^7 (-1)^{n(j)}e_j)}. 
\end{align*}
This means that 
\begin{align*}
\mathfrak{l}\simeq \mathfrak{sl}(8,\mathbb{C})\oplus \mathbb{C}. 
\end{align*}

The $\operatorname{ad}(H)$-eigenspace $\mathfrak{g}(2)$ is written as 
\begin{align*}
\mathfrak{g}(2)&=\bigoplus _{1\leq j\leq 7}\mathfrak{g}_{e_8-e_j}\oplus 
	\bigoplus _{\sum _{j=1}^7n(j)=2}
	\mathfrak{g}_{\frac{1}{2}(e_8+\sum _{j=1}^7 (-1)^{n(j)}e_j)}. 
\end{align*}
Then, $\mathfrak{g}(2)$ is isomorphic to $\operatorname{Alt}(8,\mathbb{C})$. 
Moreover, the eigenspace $\mathfrak{g}(3)$ forms 
\begin{align*}
\mathfrak{g}(3)&=\bigoplus _{1\leq j\leq 7}\mathfrak{g}_{e_8+e_j}
	\oplus \mathfrak{g}_{\frac{1}{2}(e_8+e_7+e_6+e_5+e_4+e_3+e_2+e_1)}\simeq \mathbb{C}^8. 
\end{align*}
Hence, 
\begin{align*}
\mathfrak{n}\simeq \operatorname{Alt}(8,\mathbb{C})\oplus \mathbb{C}^8. 
\end{align*}

The semisimple part $SL(8,\mathbb{C})$ 
of the Levi subgroup $L_{\mathbb{C}}$ acts on 
$\operatorname{Alt}(8,\mathbb{C})\oplus \mathbb{C}^8$ diagonally, namely,  
\begin{align*}
g\cdot (A,v)=(gA\,{}^tg,gv), 
\end{align*}
and its center $\mathbb{C}^{\times}$ acts by 
\begin{align*}
s\cdot (A,v)=(s^2A,s^3v). 
\end{align*}
Then, the $L_{\mathbb{C}}$-action on $\mathfrak{n}$ is geometrically equivalent to 
the indecomposable $(SL(8,\mathbb{C})\times \mathbb{C}^{\times})$-action 
on $\operatorname{Alt}(8,\mathbb{C})\oplus \mathbb{C}^8$. 
It follows from (8) of Table \ref{table:linear} that 
this action is a multiplicity-free action. 

We take the subspace $S_0'\subset \mathfrak{g}(2)$ as 
\begin{multline*}
S_0':=\mathbb{R}E_{e_8-e_1}
	\oplus \mathbb{R}E_{\frac{1}{2}(e_8+e_7+e_6+e_5+e_4-e_3-e_2+e_1)}\\
	\oplus \mathbb{R}E_{\frac{1}{2}(e_8+e_7+e_6-e_5-e_4+e_3+e_2+e_1)}
	\oplus \mathbb{R}E_{\frac{1}{2}(e_8-e_7-e_6+e_5+e_4+e_3+e_2+e_1)}, 
\end{multline*}
and $S_0''\subset \mathfrak{g}(3)$ as
\begin{align*}
S_0'':=\mathbb{R}E_{e_8+e_7}\oplus \mathbb{R}E_{e_8+e_5}\oplus 
	\mathbb{R}E_{e_8+e_3}\oplus \mathbb{R}E_{e_8+e_1}. 
\end{align*}
We set 
\begin{align*}
S_0:=S_0'\oplus S_0''. 
\end{align*}
Then, $S_0$ is isomorphic to the slice $A_4\oplus T_4$ 
of Table \ref{table:choice} for the $(SU(8)\times \mathbb{T})$-action 
on $\operatorname{Alt}(8,\mathbb{C})\oplus \mathbb{C}^8$. 
It follows from Lemma \ref{lem:linear} that $\mathfrak{n}=L_u\cdot S_0$. 
Therefore, Theorem \ref{thm:slice-n} has been verified. 


\subsection{Type F$_4$}
\label{subsec:f4}

In this subsection, we give a proof of Theorem \ref{thm:slice-n} 
for $\mathfrak{g}=\mathfrak{f}_4(\mathbb{C})$. 
In this case, $\mathfrak{g}_{\mathbb{R}}\simeq \mathfrak{f}_{4(4)}$. 
Then, $\mathfrak{a}^*_{\mathbb{R}}=\mathbb{R}e_1\oplus \cdots \oplus \mathbb{R}e_4$. 
A root system $\Delta $ is $\Delta =\{ \pm e_i\pm e_j:1\leq i<j\leq 4\} 
\sqcup \{ \pm e_i:1\leq i\leq 4\} \sqcup \{ \pm \frac{1}{2}(e_1\pm e_2\pm e_3\pm e_4)\} $. 
We fix a positive system $\Delta ^+$ as 
$\Delta ^+=\{ e_i\pm e_j:1\leq i<j\leq 4\} 
\sqcup \{ e_i:1\leq i\leq 4\} \sqcup \{ \frac{1}{2}(e_1\pm e_2\pm e_3\pm e_4)\} $. 
The simple roots $\alpha _1,\ldots ,\alpha _4$ are given by 
$\alpha _1=\frac{1}{2}(e_1-e_2-e_3-e_4)$, 
$\alpha _2=e_4$, $\alpha _3=e_3-e_4$, and $\alpha _4=e_2-e_3$. 
The highest root $\beta =e_1+e_2$ is written as 
$\beta =2\alpha _1+4\alpha _2+3\alpha _3+2\alpha _4$. 

Let $\mathcal{O}_X$ be a nilpotent orbit in $\mathfrak{f}_4(\mathbb{C})$ 
with characteristic element $H=h_1H_1+\cdots +h_4H_4\in \mathfrak{a}_+$. 
Then, the weighted Dynkin diagram $\Omega (\mathcal{O}_X)=(m_1,\ldots ,m_4)$ is given by 
$(\frac{1}{2}(h_1-h_2-h_3-h_4),h_4,h_3-h_4,h_2-h_3)$. 

\begin{figure}[htbp]
\begin{align*}
\SelectTips{cm}{12}
	\objectmargin={1pt}
	\xygraph{!~:{@{<=}}
		\circ ([]!{+(0,-.3)} {\alpha_1},[]!{+(0,+.3)} {m_1}) - [r]
		\circ ([]!{+(0,-.3)} {\alpha_2},[]!{+(0,+.3)} {m_2}) : [r]
		\circ ([]!{+(0,-.3)} {\alpha_3},[]!{+(0,+.3)} {m_3}) - [r]
		\circ ([]!{+(0,-.3)} {\alpha_4},[]!{+(0,+.3)} {m_4})
	}
\end{align*}
\caption{Weighted Dynkin diagram of $\mathcal{O}_X$ in $\mathfrak{f}_4(\mathbb{C})$}
\end{figure}

A nilpotent orbit $\mathcal{O}_X$ is spherical if and only if 
$\Omega (\mathcal{O}_X)$ satisfies one of the following cases: 
\begin{enumerate}
	\renewcommand{\labelenumi}{(F$_4$\theenumi) }
	\item $\Omega (\mathcal{O}_X)=(0,0,0,1)$. 
	\label{case:f41}
	\item $\Omega (\mathcal{O}_X)=(1,0,0,0)$. 
	\label{case:f42}
	\item $\Omega (\mathcal{O}_X)=(0,0,1,0)$. 
	\label{case:f43}
\end{enumerate}

The height of $\mathcal{O}_X$ equals two for Cases (F$_4$\ref{case:f41}) and (F$_4$\ref{case:f42}), 
and three for Case (F$_4$\ref{case:f43}). 

\subsubsection{Case (F$_4$\ref{case:f41})}
\label{subsubsec:f4}

Let $\Omega (\mathcal{O}_X)$ be $(0,0,0,1)$. 
Then, $H\in \mathfrak{a}$ is given by 
\begin{align*}
H=E_1+E_2. 
\end{align*}
This $\mathcal{O}_X$ is the minimal nilpotent orbit with dimension 16. 

The Levi subalgebra $\mathfrak{l}$ is given by 
\begin{align*}
\mathfrak{l}=\mathfrak{a}\oplus 
	\mathfrak{g}_{\frac{1}{2}(\pm (e_1-e_2)\pm e_3\pm e_4)}
	\oplus \mathfrak{g}_{\pm (e_1-e_2)}\oplus \mathfrak{g}_{\pm e_3\pm e_4}
	\oplus \mathfrak{g}_{\pm e_3}
	\oplus \mathfrak{g}_{\pm e_4}. 
\end{align*}
This means that 
\begin{align*}
\mathfrak{l}\simeq \mathfrak{sp}(3,\mathbb{C})\oplus \mathbb{C}. 
\end{align*}

The $\operatorname{ad}(H)$-eigenspace $\mathfrak{g}(2)$ is written as 
\begin{align*}
\mathfrak{g}(2)=\mathfrak{g}_{e_1+e_2}\simeq \mathbb{C}. 
\end{align*}
This implies that 
the action of the Levi subgroup $L_{\mathbb{C}}$ on $\mathfrak{n}$ 
is geometrically equivalent to the $\mathbb{C}^{\times}$-action on $\mathbb{C}$. 
It follows from (1) of Table \ref{table:linear} that 
this action is a multiplicity-free actions. 

We take 
\begin{align*}
S_0=\mathbb{R}E_{e_1+e_2}. 
\end{align*}
Then, it follows from Lemma \ref{lem:linear} that 
$\mathfrak{n}=L_u\cdot S_0$. 
Therefore, Theorem \ref{thm:slice-n} has been verified for Case (F$_4$\ref{case:f41}). 

\subsubsection{Case (F$_4$\ref{case:f42})}

Let $\Omega (\mathcal{O}_X)=(1,0,0,0)$. 
Then, 
\begin{align*}
H=2E_1. 
\end{align*}

The Levi subalgebra $\mathfrak{l}$ is given by 
\begin{align*}
\mathfrak{l}=\mathfrak{a}\oplus \bigoplus _{2\leq i<j\leq 4}
	\mathfrak{g}_{\pm e_{i}\pm e_{j}}
	\oplus \bigoplus _{2\leq i\leq 4}\mathfrak{g}_{\pm e_{i}}
\simeq \mathfrak{so}(7,\mathbb{C})\oplus \mathbb{C}. 
\end{align*}
The eigenspace $\mathfrak{g}(2)$ is of the form: 
\begin{align*}
\mathfrak{g}(2)&=\mathfrak{g}_{e_1}\oplus 
	\bigoplus _{2\leq i<j\leq 4}\mathfrak{g}_{e_1\pm e_{j}}. 
\end{align*}
Then, 
\begin{align*}
\mathfrak{n}\simeq \mathbb{C}^7. 
\end{align*}
This implies that the $L_{\mathbb{C}}$-action on $\mathfrak{n}$ 
is geometrically equivalent to the irreducible $(SO(7,\mathbb{C})\times \mathbb{C}^{\times})$-action 
on $\mathbb{C}^7$. 
It follows from (6) of Table \ref{table:linear} that 
this action is a multiplicity-free action. 

We take the subspace $S_0$ as 
\begin{align*}
S_0=\mathbb{R}E_{e_1+e_2}\oplus \mathbb{R}E_{e_1-e_2}. 
\end{align*}
Then, this is isomorphic to the slice $D_{1,1}$ of Table \ref{table:choice} 
for the $(SO(7)\times \mathbb{T})$-action on $\mathbb{C}^7$. 
By Lemma \ref{lem:linear}, $\mathfrak{n}=L_u\cdot S_0$. 
Hence, Theorem \ref{thm:slice-n} has been verified for Case (F$_4$\ref{case:f42}). 

\subsubsection{Case (F$_4$\ref{case:f43})}

Let $\Omega (\mathcal{O}_X)$ be $(0,0,1,0)$. 
Then, 
\begin{align*}
H=2E_1+E_2+E_3. 
\end{align*}

The Levi subalgebra $\mathfrak{l}$ is given by 
\begin{align*}
\mathfrak{l}=\mathfrak{a}\oplus \mathfrak{g}_{\pm (e_2-e_3)}\oplus \mathfrak{g}_{\pm e_4}
	\oplus \mathfrak{g}_{\pm \frac{1}{2}(e_1-e_2-e_3)\pm \frac{1}{2}e_4}. 
\end{align*}
This means that 
\begin{align*}
\mathfrak{l}\simeq \mathfrak{sl}(3,\mathbb{C})\oplus \mathfrak{sl}(2,\mathbb{C})\oplus 
	\mathbb{C}. 
\end{align*}

The $\operatorname{ad}(H)$-eigenspace $\mathfrak{g}(2)$ is written as 
\begin{align*}
\mathfrak{g}(2)&=\mathfrak{g}_{e_1}\oplus \mathfrak{g}_{e_2+e_3}
	\oplus \mathfrak{g}_{e_1\pm e_4}
	\oplus \mathfrak{g}_{\frac{1}{2}(e_1+e_2+e_3)\pm \frac{1}{2}e_4}. 
\end{align*}
Then, $\mathfrak{g}(2)$ is isomorphic to $\operatorname{Sym}(3,\mathbb{C})$. 
Further, $\mathfrak{g}(3)$ forms 
\begin{align*}
\mathfrak{g}(3)=\mathfrak{g}_{e_1+e_2}\oplus \mathfrak{g}_{e_1+e_3}\simeq \mathbb{C}^2. 
\end{align*}
Hence, 
\begin{align*}
\mathfrak{n}\simeq \operatorname{Sym}(3,\mathbb{C})\oplus \mathbb{C}^2. 
\end{align*}

The semisimple part $SL(3,\mathbb{C})\times SL(2,\mathbb{C})$ of the 
Levi subgroup $L_{\mathbb{C}}$ acts on $\operatorname{Sym}(3,\mathbb{C})\oplus \mathbb{C}^2$ by 
\begin{align*}
(g,h)\cdot (A,v)=(gA\,{}^tg,hv), 
\end{align*}
and its center $\mathbb{C}^{\times}$ by 
\begin{align*}
s\cdot (A,v)=(s^2A,s^3v). 
\end{align*}
Then, the $L_{\mathbb{C}}$-action on $\mathfrak{n}$ is geometrically equivalent to 
the decomposable action consisting of 
the $(SL(3,\mathbb{C})\times \mathbb{C}^{\times})$-action on $\operatorname{Sym}(3,\mathbb{C})$ 
((3) of Table \ref{table:linear}) 
and the $SL(2,\mathbb{C})$-action on $\mathbb{C}^2$ ((2) of Table \ref{table:linear}). 
Hence, this action is a multiplicity-free action. 

We take the subspace in $\mathfrak{g}(2)$ as
\begin{align*}
S_0':=\mathbb{R}E_{e_1+e_4}\oplus \mathbb{R}E_{e_1-e_4}
	\oplus \mathbb{R}E_{e_2+e_3}. 
\end{align*}
Then, $S_0'$ is isomorphic to the slice $D_3$ of Table \ref{table:choice}. 
By Lemma \ref{lem:linear}, $\mathfrak{g}(2)=(SU(3)\times \mathbb{T})\cdot S_0'$. 
Similarly, 
\begin{align*}
S_0'':=\mathbb{R}E_{e_1+e_2}
\end{align*}
is isomorphic to $T_1$ of Table \ref{table:linear}, and then 
$\mathfrak{g}(3)=SU(2)\cdot S_0''$. 
We set 
\begin{align*}
S_0=S_0'\oplus S_0''\simeq D_3\oplus T_1. 
\end{align*}
Then, the linear space $S_0$ 
satisfies $\mathfrak{n}=L_u\cdot S_0$. 
Therefore, Theorem \ref{thm:slice-n} has been verified for Case (F$_4$\ref{case:f43}). 


\subsection{Type G$_2$}
\label{subsec:g2}

In this subsection, we give a proof of Theorem \ref{thm:slice-n} 
for $\mathfrak{g}=\mathfrak{g}_2(\mathbb{C})$. 
In this case, $\mathfrak{g}_{\mathbb{R}}\simeq \mathfrak{g}_{2(2)}$. 
Then, $\mathfrak{a}_{\mathbb{R}}^*=\{ v\in \mathbb{R}e_1\oplus \mathbb{R}e_2\oplus \mathbb{R}e_3
:\langle v,e_1+e_2+e_3\rangle =0\} $. 
A root system $\Delta \equiv \Delta (\mathfrak{g},\mathfrak{a})$ is 
$\Delta =\{ \pm (e_i-e_j):1\leq i<j\leq 3\} \sqcup \{ \pm (-2e_1+e_2+e_3),
\pm (e_1-2e_2-e_3),\pm (-e_1-e_2+2e_3)\} $. 
We fix a positive system $\Delta ^+$ as $\Delta ^+=\{ e_1-e_2,~-e_2+e_3,~-e_1+e_3,~
-2e_1+e_2+e_3,~e_1-2e_2+e_3,~-e_1-e_2+2e_3\} $. 
The simple roots are $\alpha _1:=e_1-e_2$ and $\alpha _2:=-2e_1+e_2+e_3$. 
The highest root $\beta =-e_1-e_2+2e_3$ is written as 
$\beta =3\alpha _1+2\alpha _2$. 

Let $\mathcal{O}_X$ be a nilpotent orbit in $\mathfrak{g}_2(\mathbb{C})$ 
with characteristic element $H=h_1E_1+h_2E_2+h_3E_3\in \mathfrak{a}_+$ with $h_1+h_2+h_3=0$. 
Then, $\alpha _1(H)=h_1-h_2$ and $\alpha _2(H)=-2h_1+h_2+h_3$. 
Hence, the weighted Dynkin diagram $\Omega (\mathcal{O}_X)=(m_1,m_2)$ is given by 
$(h_1-h_2,-2h_1+h_2+h_3)$. 

\begin{figure}[htbp]
\begin{align*}
\SelectTips{cm}{12}
	\objectmargin={1pt}
	\xygraph{!~:{@3{<-}}
		\circ ([]!{+(0,-.3)} {\alpha_1},[]!{+(0,+.3)} {m_1}) : [r]
		\circ ([]!{+(0,-.3)} {\alpha_2},[]!{+(0,+.3)} {m_2})
	}
\end{align*}
\caption{Weighted Dynkin diagram of $\mathcal{O}_X$ in $\mathfrak{g}_2$}
\end{figure}

A nilpotent orbit $\mathcal{O}_X$ is spherical 
if and only if $\Omega (\mathcal{O}_X)$ satisfies either Case (G$_2$\ref{case:g21}) or Case (G$_2$\ref{case:g22}): 
\begin{enumerate}
	\renewcommand{\labelenumi}{(G$_2$\theenumi) }
	\item $\Omega (\mathcal{O}_X)=(0,1)$. 
	\label{case:g21}
	\item $\Omega (\mathcal{O}_X)=(1,0)$. 
	\label{case:g22}
\end{enumerate}

It follows from Lemma \ref{lem:height-max} that 
the height of $\mathcal{O}_X$ equals two if $\Omega (\mathcal{O}_X)=(0,1)$ 
(Case (G$_2$\ref{case:g21})), 
and three if $\Omega (\mathcal{O}_X)=(1,0)$ (Case (G$_2$\ref{case:g22})). 

\subsubsection{Case (G$_2$\ref{case:g21})} 
\label{subsubsec:g2}

Let $\Omega (\mathcal{O}_X)$ be $(0,1)$. 
This $\mathcal{O}_X$ is the minimal nilpotent orbit with dimension six. 
Then, $H\in \mathfrak{a}$ is of the form 
\begin{align*}
H=\frac{1}{3}(-E_1-E_2+2E_3). 
\end{align*}

The Levi subalgebra $\mathfrak{l}$ is given by 
\begin{align*}
\mathfrak{l}=\mathfrak{a}\oplus \mathfrak{g}_{\pm (e_1-e_2)}. 
\end{align*}
This means that 
\begin{align*}
\mathfrak{l}\simeq \mathfrak{sl}(2,\mathbb{C})\oplus \mathbb{C}. 
\end{align*}

The $\operatorname{ad}(H)$-eigenspace $\mathfrak{g}(2)$ is 
\begin{align*}
\mathfrak{g}(2)=\mathfrak{g}_{-e_1-e_2+2e_3}\simeq \mathbb{C}. 
\end{align*}
Then, the action of the Levi subgroup $L_{\mathbb{C}}$ on $\mathfrak{n}$ 
is geometrically equivalent to the $\mathbb{C}^{\times}$-action on $\mathbb{C}$. 
It follows from (1) of Table \ref{table:linear} that 
this action is a multiplicity-free action. 

We take the subspace $S_0$ to be 
\begin{align*}
S_0=\mathbb{R}E_{-e_1-e_2+2e_3}. 
\end{align*}
Then, $\mathfrak{n}=L_u\cdot S_0$. 
Therefore, 
Theorem \ref{thm:slice-n} has been verified for Case (G$_2$\ref{case:g21}). 

\subsubsection{Case (G$_2$\ref{case:g22})} 

Let $\Omega (\mathcal{O}_X)$ be $(1,0)$. 
Then, 
\begin{align*}
H=-E_2+E_3. 
\end{align*}

The Levi subalgebra $\mathfrak{l}$ is given by 
\begin{align*}
\mathfrak{l}=\mathfrak{a}\oplus \mathfrak{g}_{\pm (-2e_1+e_2+e_3)}. 
\end{align*}
This means that 
\begin{align*}
\mathfrak{l}\simeq \mathfrak{sl}(2,\mathbb{C})\oplus \mathbb{C}. 
\end{align*}

The $\operatorname{ad}(H)$-eigenspace $\mathfrak{g}(2)$ is written as 
\begin{align*}
\mathfrak{g}(2)&=\mathfrak{g}_{-e_2+e_3}\simeq \mathbb{C}. 
\end{align*}
Further, the eigenspace $\mathfrak{g}(3)$ is 
\begin{align*}
\mathfrak{g}(3)&=\mathfrak{g}_{e_1-2e_2+e_3}\oplus \mathfrak{g}_{-e_1-e_2+2e_3}
\simeq \mathbb{C}^2. 
\end{align*}
Then, 
\begin{align*}
\mathfrak{n}&\simeq \mathbb{C}\oplus \mathbb{C}^2. 
\end{align*}

The semisimple part $SL(2,\mathbb{C})$ of $L_{\mathbb{C}}$ acts 
on $\mathbb{C}$ trivially and on $\mathbb{C}^2$ as the standard representation. 
Its center $\mathbb{C}^{\times}$ acts by
\begin{align*}
s\cdot (z,v)=(s^2z,s^3v). 
\end{align*}
Then, the $L_{\mathbb{C}}$-action on $\mathfrak{n}$ is geometrically equivalent to 
the decomposable action consisting of the $\mathbb{C}^{\times}$-action on $\mathbb{C}$ 
((1) of Table \ref{table:linear}) 
and the $SL(2,\mathbb{C})$-action on $\mathbb{C}^2$ ((2) of Table \ref{table:linear}). 
Hence, this action is a multiplicity-free action. 

We take 
\begin{align*}
S_0':=\mathbb{R}E_{-e_2+e_3}\simeq \mathbb{R}. 
\end{align*}
Then, $S_0'$ satisfies $\mathfrak{g}(2)=\mathbb{T}\cdot S_0''$. 
Further, we take the subset $S_0''$ of $\mathfrak{g}(3)$ as 
\begin{align*}
S_0'':=\mathbb{R}E_{-e_1-e_2+2e_3}. 
\end{align*}
Then, $S_0''$ is isomorphic to the slice $T_1$ for the $SU(2)$-action on $\mathbb{C}^2$. 
It follows from Lemma \ref{lem:linear} that $\mathfrak{g}(3)=SU(2)\cdot S_0''$. 
Hence, we set 
\begin{align*}
S_0:=S_0'\oplus S_0''\simeq \mathbb{R}\oplus T_1. 
\end{align*}
Then, $S_0$ satisfies $\mathfrak{n}=L_u\cdot S_0$. 
Therefore, Theorem \ref{thm:slice-n} has been verified. 

\subsection{Proof of Theorem \ref{thm:slice-n}} 
\label{subsec:proof-n}

\begin{proof}[Proof of Theorem \ref{thm:slice-n}]
In Sections \ref{subsec:a}--\ref{subsec:g2}, 
we have given $S_0$ for the $L_u$-action on $\mathfrak{n}$ 
satisfying the properties 
(\ref{cond:vectorspace})--(\ref{cond:s1-n}) explicitly for each $\mathcal{O}_X$ in $\mathfrak{g}$. 
Then, Theorem \ref{thm:slice-n} follows from Lemmas \ref{lem:linear} and \ref{lem:linear2}. 
\end{proof}


\begin{table}[p]
\rotatebox[origin=c]{90}{
	\begin{minipage}{\textheight}
	\begin{align*}
	\begin{array}{cccccc}
	\hline 
	 \mathfrak{g} & \Omega (\mathcal{O}_X)=(m_1,m_2,\ldots ) & \mathfrak{l}=\mathfrak{g}(0) & \mathfrak{g}(2) & \mathfrak{g}(3) 
	& S_0 \\
	\hline 
	\mathfrak{a}_{n-1} & m_p=m_{n-p}=1~(p<\frac{n}{2}) 
		& \mathfrak{sl}(p,\mathbb{C})\oplus \mathfrak{sl}(n-2p,\mathbb{C})\oplus 
			\mathfrak{sl}(p,\mathbb{C})\oplus \mathbb{C}^2 
		& M(p,\mathbb{C}) 
		& \{ 0\} 
		& \mathbb{R}^p\\
	\mathfrak{a}_{2p-1}& m_{p}=2 
		& \mathfrak{sl}(p,\mathbb{C})\oplus 
			\mathfrak{sl}(p,\mathbb{C})\oplus \mathbb{C}^2 
		& M(p,\mathbb{C}) 
		& \{ 0\} 
		& \mathbb{R}^p\\
	\mathfrak{b}_n & m_1=1 
		& \mathfrak{so}(2n-1,\mathbb{C})\oplus \mathbb{C} 
		& \mathbb{C}^{2n-1} 
		& \{ 0\} 
		& \mathbb{R}^2\\
	\mathfrak{b}_n & m_{2p}=1 
		& \mathfrak{sl}(2p,\mathbb{C})\oplus \mathfrak{so}(2n-4p+1,\mathbb{C})
			\oplus \mathbb{C} 
		& \operatorname{Alt}(2p,\mathbb{C}) 
		& \{ 0\} 
		& \mathbb{R}^p\\
	\mathfrak{b}_n & m_1=m_{2p+1}=1 ~(p<\frac{n-1}{2}) 
		& \mathfrak{sl}(2p,\mathbb{C})\oplus \mathfrak{so}(2n-4p-1,\mathbb{C})
			\oplus \mathbb{C}^2 
		& \mathbb{C}^{2n-4p-1}\oplus \operatorname{Alt}(2p,\mathbb{C}) 
		& \mathbb{C}^{2p} 
		& (\mathbb{R}^2\oplus \mathbb{R}^p)\oplus \mathbb{R}^p\\
	\mathfrak{b}_{2p+1} & m_1=m_{2p+1}=1 
		& \mathfrak{sl}(2p,\mathbb{C})\oplus \mathbb{C}^2 
		& \mathbb{C}\oplus \operatorname{Alt}(2p,\mathbb{C}) 
		& \mathbb{C}^{2p} 
		& (\mathbb{R}\oplus \mathbb{R}^p)\oplus \mathbb{R}^p\\
	\mathfrak{c}_n & m_p=1~(p<n) 
		& \mathfrak{sl}(p,\mathbb{C})\oplus \mathfrak{sp}(n-p,\mathbb{C})\oplus \mathbb{C} 
		& \operatorname{Sym}(p,\mathbb{C}) 
		& \{ 0\} 
		& \mathbb{R}^p\\
	\mathfrak{c}_n & m_n=2 
		& \mathfrak{sl}(n,\mathbb{C})\oplus \mathbb{C} 
		& \operatorname{Sym}(n,\mathbb{C}) 
		& \{ 0\} 
		& \mathbb{R}^n\\
	\mathfrak{d}_n & m_1=2 
		& \mathfrak{so}(2n-2,\mathbb{C})\oplus \mathbb{C} 
		& \mathbb{C}^{2n-2} 
		& \{ 0\} 
		& \mathbb{R}^2\\
	\mathfrak{d}_n & m_{2p}=1~(p\leq \frac{n-2}{2}) 
		& \mathfrak{sl}(2p,\mathbb{C})\oplus \mathfrak{so}(2n-4p,\mathbb{C})\oplus \mathbb{C}
		& \operatorname{Alt}(2p,\mathbb{C}) 
		& \{ 0\} 
		& \mathbb{R}^p\\
	\mathfrak{d}_{2p+1} & m_{2p}=m_{2p+1}=1 
		& \mathfrak{sl}(2p,\mathbb{C})\oplus \mathbb{C}^2 
		& \operatorname{Alt}(2p,\mathbb{C}) 
		& \{ 0\} 
		& \mathbb{R}^p\\
	\mathfrak{d}_{2p} & m_{2p}=2 
		& \mathfrak{sl}(2p,\mathbb{C})\oplus \mathbb{C} 
		& \operatorname{Alt}(2p,\mathbb{C}) 
		& \{ 0\} 
		& \mathbb{R}^p\\
	\mathfrak{d}_{2p} & m_{2p-1}=2 
		& \mathfrak{sl}(2p,\mathbb{C})\oplus \mathbb{C} 
		& \operatorname{Alt}(2p,\mathbb{C}) 
		& \{ 0\} 
		& \mathbb{R}^p\\
	\mathfrak{d}_n & m_1=m_{2p+1}=1~(p<\frac{n-2}{2}) 
		& \mathfrak{sl}(2p,\mathbb{C})\oplus \mathfrak{so}(2n-4p-2,\mathbb{C})\oplus 
			\mathbb{C}^2 
		& \mathbb{C}^{2n-4p-2}\oplus \operatorname{Alt}(2p,\mathbb{C}) 
		& \mathbb{C}^{2p} 
		& (\mathbb{R}^2\oplus \mathbb{R}^p)\oplus \mathbb{R}^p\\
	\mathfrak{d}_{2p+2} & m_1=m_{2p+1}=m_{2p+2}=1 
		& \mathfrak{sl}(2p,\mathbb{C})\oplus \mathbb{C}^3 
		& \mathbb{C}^2\oplus \operatorname{Alt}(2p,\mathbb{C}) 
		& \mathbb{C}^{2p} 
		& (\mathbb{R}^2\oplus \mathbb{R}^p)\oplus \mathbb{R}^p \\
	\hline 
	\end{array}
	\end{align*}
	\caption{Spherical nilpotent orbits and slices for the 
	$L_{u}$-action on $\mathfrak{n}$
	: $\mathfrak{g}$ is of classical type}
	\label{table:classical}
	\end{minipage}
}
\end{table}

\begin{table}[p]
\rotatebox[origin=c]{90}{
	\begin{minipage}{\textheight}
	\begin{align*}
	\begin{array}{cccccc}
	\hline 
	\mathfrak{g} & \Omega (\mathcal{O}_X) & \mathfrak{l}=\mathfrak{g}(0) & \mathfrak{g}(2) 
		& \mathfrak{g}(3) & S_0 \\
	\hline 
	\mathfrak{e}_6 
		& (0,1,0,0,0,0) 
		& \mathfrak{sl}(6,\mathbb{C})\oplus \mathbb{C} 
		& \mathbb{C} 
		& \{ 0\} 
		& \mathbb{R} \\
	\mathfrak{e}_6 
		& (1,0,0,0,0,1) 
		& \mathfrak{so}(8,\mathbb{C})\oplus \mathbb{C}^2 
		& \mathbb{C}^8 
		& \{ 0\} 
		& \mathbb{R}^2 \\
	\mathfrak{e}_6 
		& (0,0,0,1,0,0) 
		& \mathfrak{sl}(3,\mathbb{C})\oplus \mathfrak{sl}(3,\mathbb{C})
			\oplus \mathfrak{sl}(2,\mathbb{C}) \oplus \mathbb{C} 
		& M(3,\mathbb{C}) 
		& \mathbb{C}^2 
		& \mathbb{R}^3\oplus \mathbb{R}\\
	\mathfrak{e}_7 
		& (1,0,0,0,0,0,0) 
		& \mathfrak{so}(12,\mathbb{C})\oplus \mathbb{C} 
		& \mathbb{C} 
		& \{ 0\} 
		& \mathbb{R} \\
	\mathfrak{e}_7 
		& (0,0,0,0,0,1,0) 
		& \mathfrak{so}(10,\mathbb{C})\oplus \mathfrak{sl}(2,\mathbb{C})\oplus \mathbb{C} 
		& \mathbb{C}^{10} 
		& \{ 0\} 
		& \mathbb{R}^2 \\
	\mathfrak{e}_7 
		& (0,0,0,0,0,0,2)
		& \mathfrak{e}_6(\mathbb{C})\oplus \mathbb{C} 
		& \mathbb{C}^{27}
		& \{ 0\} 
		& \mathbb{R}^3 \\
	\mathfrak{e}_7 
		& (0,0,1,0,0,0,0) 
		& \mathfrak{sl}(6,\mathbb{C})\oplus \mathfrak{sl}(2,\mathbb{C})\oplus \mathbb{C} 
		& \operatorname{Alt}(6,\mathbb{C}) 
		& \mathbb{C}^2 
		& \mathbb{R}^3\oplus \mathbb{R} \\
	\mathfrak{e}_7 
		& (0,1,0,0,0,0,1) 
		& \mathfrak{sl}(6,\mathbb{C})\oplus \mathbb{C}^2 
		& \mathbb{C}\oplus \operatorname{Alt}(6,\mathbb{C}) 
		& \mathbb{C}^6 
		& (\mathbb{R}\oplus \mathbb{R}^3)\oplus \mathbb{R}^3 \\
	\mathfrak{e}_8 
		& (0,0,0,0,0,0,0,1) 
		& \mathfrak{e}_7(\mathbb{C})\oplus \mathbb{C} 
		& \mathbb{C} 
		& \{ 0\} 
		& \mathbb{R} \\
	\mathfrak{e}_8 
		& (1,0,0,0,0,0,0,0) 
		& \mathfrak{so}(14,\mathbb{C})\oplus \mathbb{C} 
		& \mathbb{C}^{14} 
		& \{ 0\} 
		& \mathbb{R}^2 \\
	\mathfrak{e}_8 
		& (0,0,0,0,0,0,1,0) 
		& \mathfrak{e}_6(\mathbb{C})\oplus \mathfrak{sl}(2,\mathbb{C})\oplus \mathbb{C} 
		& \mathbb{C}^{27}
		& \mathbb{C}^2 
		& \mathbb{R}^3\oplus \mathbb{R} \\
	\mathfrak{e}_8 
		& (0,1,0,0,0,0,0,0) 
		& \mathfrak{sl}(8,\mathbb{C})\oplus \mathbb{C} 
		& \operatorname{Alt}(8,\mathbb{C}) 
		& \mathbb{C}^8 
		& \mathbb{R}^4\oplus \mathbb{R}^4 \\
	\mathfrak{f}_4 
		& (0,0,0,1) 
		& \mathfrak{sp}(3,\mathbb{C})\oplus \mathbb{C} 
		& \mathbb{C} 
		& \{ 0\} 
		& \mathbb{R} \\
	\mathfrak{f}_4 
		& (1,0,0,0) 
		& \mathfrak{so}(7,\mathbb{C})\oplus \mathbb{C} 
		& \mathbb{C}^7 
		& \{ 0\} 
		& \mathbb{R}^2 \\
	\mathfrak{f}_4 
		& (0,0,1,0) 
		& \mathfrak{sl}(3,\mathbb{C})\oplus \mathfrak{sl}(2,\mathbb{C})\oplus \mathbb{C} 
		& \operatorname{Sym}(3,\mathbb{C}) 
		& \mathbb{C}^2 
		& \mathbb{R}^3\oplus \mathbb{R} \\
	\mathfrak{g}_2 
		& (0,1) 
		& \mathfrak{sl}(2,\mathbb{C})\oplus \mathbb{C} 
		& \mathbb{C} 
		& \{ 0\} 
		& \mathbb{R} \\
	\mathfrak{g}_2 
		& (1,0) 
		& \mathfrak{sl}(2,\mathbb{C})\oplus \mathbb{C} 
		& \mathbb{C} 
		& \mathbb{C}^2 
		& \mathbb{R}\oplus \mathbb{R} \\
	\hline 
	\end{array}
	\end{align*}
	\caption{Spherical nilpotent orbits and slices for the 
	$L_{u}$-action on $\mathfrak{n}$
	: $\mathfrak{g}$ is of exceptional type}
	\label{table:exceptional}
	\end{minipage}
}
\end{table}


\subsection{Corollary of proof for Theorem \ref{thm:slice-n}}
\label{subsec:corollary}

Finally, 
we give two corollaries of the proof of Theorem \ref{thm:slice-n} 
on $\mathcal{O}_X$ with $\operatorname{ht}(\mathcal{O}_X)=2$. 

The first corollary below is concerned to the property 
on the $L_u$-action on $\mathfrak{n}=\mathfrak{g}(2)$. 

\begin{corollary}
\label{cor:two}
Let $\mathcal{O}_X$ be a a nilpotent orbit in a complex simple Lie algebra $\mathfrak{g}$. 
If $\operatorname{ht}(\mathcal{O}_X)=2$, then we have: 
\begin{enumerate}
	\item The $L_u$-action on $\mathfrak{n}$ is geometrically equivalent to the isotropy representation 
	for some non-compact irreducible Hermitian symmetric space $G/K$, 
	namely, 
	$L_u$ is locally isomorphic to $K$ 
	and the $L_u$-action on $\mathfrak{n}$ is isomorphic to the $K$-action on 
	the tangent space $T_{eK}(G/K)$ at the origin $eK\in G/K$. 
	
	\item One can take a slice $S$ for the strongly visible $G_u$-action on $\mathcal{O}_X$ 
	satisfying $\dim _{\mathbb{R}}S=\operatorname{rank}G/K$. 
\end{enumerate}
\end{corollary}

\begin{proof}
The first statement follows from the proof of Theorem \ref{thm:slice-n} 
given in Sections \ref{subsec:a}--\ref{subsec:g2} and \cite{irr}. 
The second one is an immediate consequence of \cite{symmetric}. 
\end{proof}


A special case of height two nilpotent orbits 
is the minimal nilpotent orbit. 
Here, a nilpotent orbit $\mathcal{O}_X$ in the complex semisimple Lie algebra $\mathfrak{g}$ is 
called \textit{minimal} if the closure of $\mathcal{O}_X$ is contained in that of 
any non-zero nilpotent orbit in $\mathfrak{g}$, namely, $\overline{\mathcal{O}_X}\subset 
\overline{\mathcal{O'}}$ for any $\mathcal{O}'\in \mathcal{N}^*/G_{\mathbb{C}}$. 
It is known that there exists uniquely the minimal nilpotent orbit 
in a complex simple Lie algebra $\mathfrak{g}$. 

The second corollary gives a new characterization for a complex nilpotent orbit 
to be minimal by the nilpotent subalgebra $\mathfrak{n}$ as follows. 

\begin{corollary}
\label{cor:minimal}
For a nilpotent orbit $\mathcal{O}_X$ in a complex simple Lie algebra $\mathfrak{g}$, 
the following two conditions are equivalent: 
\begin{enumerate}
	\renewcommand{\theenumi}{\roman{enumi}}
	\item $\mathcal{O}_X$ is minimal. 
	\item $\dim _{\mathbb{C}}\mathfrak{n}=1$. 
\end{enumerate}
Moreover, 
the $G_u$-action on the minimal $\mathcal{O}_X$ is strongly visible 
with one-dimensional slice. 
\end{corollary}

\begin{proof}
Table \ref{table:minimal} gives a list of the minimal nilpotent orbit 
for each complex simple Lie algebra $\mathfrak{g}$. 
This shows that 
$\dim _{\mathbb{C}}\mathfrak{n}=\dim _{\mathbb{C}}\mathfrak{g}(2)=1$ 
for minimal $\mathcal{O}_X$. 
By Theorem \ref{thm:slice-n}, 
the $L_u$-action on $\mathfrak{n}$ is strongly visible 
with slice $S\simeq \mathbb{R}$. 
Hence, $S=S_0\cap \mathfrak{n}^{\circ}\simeq \mathbb{R}^{\times}$ becomes 
a slice for the $G_u$-action on $\mathcal{O}_X$.

Conversely, suppose that $\dim _{\mathbb{C}}\mathfrak{n}=1$. 
Then, the height of $\mathcal{O}_X$ has to be equal to two. 
Indeed, 
if $\operatorname{ht}(\mathcal{O}_X)=d>2$, 
then $\mathfrak{n}$ contains the complex vector subspace $\mathfrak{g}(2)\oplus \mathfrak{g}(d)$. 
Obviously, $\dim _{\mathbb{C}}(\mathfrak{g}(2)\oplus \mathfrak{g}(d))\geq 2$, 
from which we have $\dim _{\mathbb{C}}\mathfrak{n}\geq 2$. 
Let us assume that $\operatorname{ht}(\mathcal{O}_X)=2$. 
From our case-by-case analysis on the $L_u$-action on $\mathfrak{n}$ 
(see also Tables \ref{table:classical} and \ref{table:exceptional}), 
it turns out that 
$\dim _{\mathbb{C}}\mathfrak{n}\neq 1$ 
if $\mathcal{O}_X$ is not minimal. 
Therefore, we have proved that 
$\dim _{\mathbb{C}}\mathfrak{n}=1$ only if $\mathcal{O}_X$ is minimal. 

Consequently, Corollary \ref{cor:minimal} has been proved. 
\end{proof}

\begin{table}[htbp]
\begin{align*}
\begin{array}{ccccc}
\hline 
\mathfrak{g} & \Omega (\mathcal{O}_X) & \mathfrak{g}(2) & S_0 & \text{Section}\\
\hline 
\mathfrak{a}_{n-1} & (1,0,\ldots ,0,1) & \mathbb{C} & \mathbb{R} & \ref{subsubsec:a}\\
\mathfrak{b}_{n} & (0,1,0,\ldots ,0) & \mathbb{C} & \mathbb{R} & \ref{subsubsec:b}\\
\mathfrak{c}_{n} & (1,0,\ldots ,0) & \mathbb{C} & \mathbb{R} & \ref{subsec:c}\\
\mathfrak{d}_{n} & (0,1,0,\ldots ,0) & \mathbb{C} & \mathbb{R} & \ref{subsubsec:d}\\
\mathfrak{e}_{6} & (0,1,0,0,0,0)& \mathbb{C} & \mathbb{R} & \ref{subsubsec:e6}\\
\mathfrak{e}_{7} & (1,0,0,0,0,0,0)& \mathbb{C} & \mathbb{R} & \ref{subsubsec:e7}\\
\mathfrak{e}_{8} & (0,0,0,0,0,0,0,1)& \mathbb{C} & \mathbb{R} & \ref{subsubsec:e8}\\
\mathfrak{f}_{4} & (0,0,0,1)& \mathbb{C} & \mathbb{R} & \ref{subsubsec:f4}\\
\mathfrak{g}_{2} & (0,1)& \mathbb{C} & \mathbb{R} & \ref{subsubsec:g2}\\
\hline 
\end{array}
\end{align*}
\caption{Minimal nilpotent orbit in $\mathfrak{g}$}
\label{table:minimal}
\end{table}



\end{document}